\documentclass[11pt]{article}

\usepackage[margin=1in]{geometry}
\usepackage{amssymb, color}
\usepackage{hyperref}

\usepackage{tikz}
 \usepackage{pgfplots}
\usepackage{graphicx,xcolor}
\usepackage{amsthm,amsfonts,euscript,amsmath,comment,slashed,here,todonotes}
\usepackage[affil-it]{authblk}
\usepackage{mathrsfs}

\newtheorem{theorem}{Theorem}[section]
\newtheorem{lemma}[theorem]{Lemma}

\newtheorem{proposition}{Proposition}[section]
\newtheorem{corollary}{Corollary}[section]

\theoremstyle{definition}

\theoremstyle{remark}
\newtheorem{remark}[theorem]{Remark}

\newcommand{\ud}{\,\mathrm{d}}
\newcommand{\p}{\ensuremath{\partial}}
\newcommand{\n}{\ensuremath{\nonumber}}
\newcommand{\eps}{\ensuremath{\varepsilon}}

\newcommand\be{\begin{equation}}
\newcommand\ee{\end{equation}}
\newcommand\bea{\begin{eqnarray}}
\newcommand\eea{\end{eqnarray}}
\newcommand\bi{\begin{itemize}}
\newcommand\ei{\end{itemize}}
\newcommand\ben{\begin{enumerate}}
\newcommand\bena{\begin{enumerate}[(a)]}
\newcommand\een{\end{enumerate}}

\newcommand\bp{\begin{proof}}
\newcommand\ep{\end{proof}}

\allowdisplaybreaks

\title{Stability of the Favorable Falkner-Skan Profiles \\ for the Stationary Prandtl Equations}
\author{ Sameer Iyer \footnote{Department of Mathematics, University of California, Davis, email: \url{sameer@math.ucdavis.edu}}}

\begin{document}

\maketitle

\begin{abstract}
The (favorable) Falkner-Skan boundary layer profiles are a one parameter ($\beta \in [0,2]$) family of self-similar solutions to the stationary Prandtl system which describes the flow over a wedge with angle $\beta \frac{\pi}{2}$. The most famous member of this family is the endpoint Blasius profile, $\beta = 0$, which exhibits pressureless flow over a flat plate. In contrast, the $\beta > 0$ profiles are physically expected to exhibit a \textit{favorable pressure gradient}, a common adage in the physics literature. In this work, we prove quantitative scattering estimates as $x \rightarrow \infty$ which precisely captures the effect of this favorable gradient through the presence of new ``CK" (Cauchy-Kovalevskaya) terms that appear in a quasilinear energy cascade. 
\end{abstract}

\setcounter{tocdepth}{2}
{\small\tableofcontents}

\section{Introduction}

The stationary Prandtl system reads as follows: 
\begin{align}
\left.  \begin{aligned} \label{Pr:intro:1}
    & u_P \p_x u_P + v_P \p_y u_P - \p_y^2 u_P =  - \p_x p_E(x)  \\
    & \p_x u_P + \p_y v_P = 0   \\
    & u_P|_{y = 0} = 0, \qquad v_P|_{y = 0} = 0, \qquad \lim_{y \rightarrow \infty} u_P = u_E(x)  \\
    & u_P|_{x = 1} = u_{P, IN}(y)  
        \end{aligned}\right| (x, y) \in (1, \infty) \times \mathbb{R}_+
\end{align}
The unknown field $[u_P, v_P]$ (formally) models the behavior of a steady incompressible flow near a physical boundary (here represented by $\{y = 0\}$). More precisely the presence of physical boundaries creates boundary layers which bridge the no-slip boundary condition at $\{y = 0\}$ to the Eulerian velocity field in the bulk. These equations were derived by L. Prandtl in his seminal work, \cite{Prandtl}. 

Above, the Eulerian quantities $u_E(x)$ and $p_E(x)$ are related through Bernoulli's law:
\begin{align}
u_E \p_x u_E = - \p_x p_E. 
\end{align}
From the point of view of the Prandtl system, these quantities are prescribed externally. Physically, they represent how the bulk flow communicates the physics of flow to the boundary layer. 

Mathematically, one notices that the equation \eqref{Pr:intro:1} is parabolic in $x$ (as long as $u_P \ge 0$), and therefore we think of $x$ as a time-like variable. For this reason, boundary data is provided at $\{y = 0\}$ and $\{y = \infty\}$, whereas initial datum is provided at $\{x = 1\}$. Note that the choice of $\{x = 1\}$ is a convention we take in this paper, and we could just as well prescribe data at any other slice $\{x = x_0\}$. This paper concerns the asymptotic in $x$ behavior of \eqref{Pr:intro:1} (which plays the role of the global in time behavior for a parabolic equation) of solutions prescribed near classical self-similar solutions, known as the Falkner-Skan profiles, which we will describe shortly. 

The essential qualitative features one notices from \eqref{Pr:intro:1} are that (a) the equation is quasilinear (the transport term); (b) it is nonlocal (the quantity $v_P$ is recovered from $u_P$ through the formula $v_P = - \int_0^y \p_x u_P$); (c) it is degenerate (the coefficient $u_P$ in front of the $\p_x$ vanishes linearly as $y \downarrow 0$); and (d) it is scalar (this is perhaps the most notable simplification compared to the Navier-Stokes equations).

Overall, there are two categories of questions concerning \eqref{Pr:intro:1}:
\begin{itemize}
\item[(PR)] The study of the Prandtl system, \eqref{Pr:intro:1}, as a stand-alone equation. Typically these questions mirror the types of questions one would ask of any parabolic Cauchy problem (local/ global wellposedness, large time behavior, singularity formation, etc...). The methods used here are normally parabolic techniques, and are often times available precisely due to the scalar feature of the equations.  
\item[(NS)] The stability/ instability of \eqref{Pr:intro:1} in the inviscid limit from Navier-Stokes. There are far fewer results in this direction; the available methods are limited, primarily due to the fact that the equations are not scalar. 
\end{itemize}

These two categories are, of course, connected. On the one hand, the importance of (PR) stems from the belief that \eqref{Pr:intro:1} affirmatively describes the Navier-Stokes velocity field at low viscosities. On the other hand, in order to prove convergence results, (NS), one typically needs even stronger information on the background flow and hence turns to results in (PR).  Recently there have been many exciting works in the (NS) category. The interested reader may turn to \cite{GuoIyerCPAM}, \cite{GVMSteady}, \cite{IM20}, \cite{GZ} and the references therein for some recent works in this direction. This article falls into the (PR) category, and hence the forthcoming discussion will exclusively be about the relevant works falling into this first category. 

The local-wellposedness of \eqref{Pr:intro:1} is a classical result due to Oleinik, \cite{Oleinik}. The result of Oleinik imposes structural or geometric criteria on the data (for example, nonnegativity $u_0 \ge 0$ and monotonicity at $0$ $u_0'(0) > 0$) as well as quantitative regularity criteria. Subsequently, a unique strong solution is obtained that has relatively low regularity (for instance $u_{yy}, u_x$ are continuous functions). The regularity theory for \eqref{Pr:intro:1} is relatively nontrivial, and does not follow immediately from Oleinik's work. Higher regularity estimates for local solutions were obtained relatively recently by in \cite{MR4232771}. More recently, the work of \cite{cinfreg} obtained $C^\infty$ regularity of solutions. 

Let us turn now to the discussion of asymptotic in $x$ behavior of solutions to \eqref{Pr:intro:1}, sometimes referred to as ``downstream" dynamics. This question is more subtle, as one might expect. As pointed out by Prandtl himself, essentially two types of behavior could occur: 
\begin{itemize}
\item[] \textbf{Stability \& Self-Similarity:} In this regime, the flow is expected to become asymptotically self-similar as $x \rightarrow \infty$. This regime is physically expected to arise from a \textit{favorable pressure gradient}, namely $\p_x p_E(x) \le 0$.

\item[] \textbf{Separation \& Reversal:} In this regime, the flow is expected to generically form singularities. This regime is expected to arise from an \textit{adverse pressure gradient}, namely $\p_x p_E(x) > 0$.
\end{itemize}

There have been a number of very important advances in recent years regarding the singularity phenomenon of boundary layer separation. The interested reader should consult the work of \cite{DalibardMasmoudi} which demonstrated separation using self-similar blowup techniques, the work of \cite{WangZhangSep} which demonstrated a different regime of separation using parabolic maximum principle based techniques. After separation, one expects the flow to reverse. Studying the Prandtl system in this regime is incredibly challenging due to the sign change of $u_P$, which places such flows entirely outside of the classical point-of-view. For an introduction to this ``mixed-type" phenomenon, the reader should see the works of \cite{IM22}, \cite{IM22b}, as well as the work of \cite{DMR}. 

This work, however, is concerned with the Prandtl system in the asymptotically stable regime, that is in the presence of a favorable pressure gradient. In this regime, the dynamics are expected to be asymptotically self-similar. We now turn to the specifics.

\subsection{Favorable Falkner-Skan Profiles}

Perhaps the simplest type of nontrivial outer Euler flow $u_E(x)$ one could take as a boundary condition are \textit{power-laws}:
\begin{align} \label{power:law}
u_E(x) = a x^m, \qquad m \ge 0, \qquad a > 0.
\end{align}
The choice of $a$ plays essentially no role, and we will set $a = 1$. While the boundary condition $u_E(x)$ is prescribed from the point-of-view of the Prandtl equations, this particular choice of Euler trace turns out to arise naturally from \textit{wedge-flows}. We include in Appendix \ref{app:pl} a computation of a solution to steady Euler around a wedge, and demonstrate that the resulting Euler trace is precisely of power-law type. In particular, one discerns the following relation between two important parameters:
\begin{align}
m = \frac{\beta}{2-\beta}, \qquad \beta = \frac{2m}{m+1}, 
\end{align}
where $\beta \frac{\pi}{2}$ corresponds to the angle of the wedge. It is useful to keep in mind the image, Figure \ref{fig:HS}, from \cite{Schlichting}.
\begin{figure} \label{fig:HS}
\centering
\includegraphics[scale=.5]{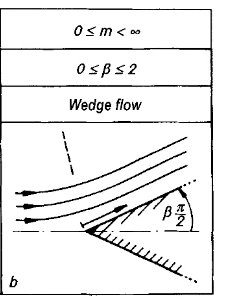}
\caption{Flow over a wedge, \cite{Schlichting}}
\end{figure}

It turns out due to the special structure imposed by the power law, the Prandtl equations admit exact self-similar profiles, the \textbf{Falkner-Skan (FS) profiles}. They are of the following form for a function $f = f(\eta)$:
\begin{align} \label{fs:Form:1}
f''' + f f'' + \beta (1 - (f')^2) = 0, \qquad 0 < \eta < \infty, \\ \label{fs:Form:2}
f(0) = f'(0) = 0, \qquad \lim_{\eta \rightarrow \infty} \frac{f(\xi)}{\xi} = 1, 
\end{align}
The function $f$ gives rise to a solution to \eqref{Pr:intro:1} with $u_E(x) = x^m$ as follows (see, for instance, \cite{Schlichting}, P. 172):
\begin{align} \label{psiFSform}
\psi_{FS; m}(x, y) = \sqrt{\frac{2}{m+1}} x^{\frac{1}{2} + \frac{m}{2}} f_{FS; m}(\xi), \qquad \xi :=\sqrt{\frac{m+1}{2}} \eta, \qquad \eta := \frac{y}{x^{\frac12(1-m)}}.
\end{align}
We note that the special case $m = 0, \beta = 0$ is the famous \textbf{Blasius} flow, which physically corresponds to flow over a flat plate. 

In this work, we are inspired by the following original result due to Serrin:
\begin{theorem}[Serrin, \cite{Serrin}] Let $m \ge 0$, and set $u_E(x) = x^m$. Let $f_{FS}$ be the corresponding solution to \eqref{fs:Form:1} -- \eqref{fs:Form:2}, with $\beta = \frac{2m}{m+1}$. Let $u_P$ be any solution to \eqref{Pr:intro:1}. Then the following scattering estimate holds \footnote{Of course, in comparison with our results below, it is more natural to multiply through by $x^m$ and state this as: $|u_P - u_{FS; m}| = o(1 + m \ln(x))$}:   
\begin{align}
\Big|\frac{u_P}{x^m} - f_{FS; m}\Big| = o(\frac{1 + m \ln(x)}{x^m}), \qquad x \rightarrow \infty.
\end{align}
\end{theorem}
Given Serrin's result, we can ask two related questions: 
\begin{itemize}
\item[(Q1)] Let $m \ge 0$. What are the precise scattering rates as $x \rightarrow \infty$?
\item[(Q2)] Let $m > 0$. How can we quantitatively realize the effect of the favorable pressure gradient? 
\end{itemize}
Out of these questions, (Q1) has been answered for the special case $m = 0$ in \cite{MR4097332} and \cite{WangZhangBl}, which we will describe below. The aim of this paper is to provide answers to both of these questions in their full parameter regimes of relevance ($m \ge 0$ for (Q1) and $m > 0$ for (Q2)).  

\vspace{2 mm}

\noindent \underline{Blasius $(m = 0)$:} For the special case of the Blasius flow, $m = 0$, this result has been revisited by several authors, in particular \cite{MR4097332}, \cite{WangZhangBl}. These authors essentially put stronger assumptions on the initial data, $u_{IN}(y)$, and produce a more refined quantitative scattering estimates in two senses: regularity and rates. In particular, we recall here 
\begin{theorem}[Iyer, \cite{MR4097332}, informal statement] Let $m = 0$. There exists $\eps > 0$ such that the following statement is valid. Suppose initial datum to \eqref{Pr:intro:1} satisfies $\| u_{IN} - u_{FS; 0}(1, \cdot) \|_{E} < \eps$. Assume standard parabolic compatibility conditions to a high order. Then the following scattering estimates hold: 
\begin{align} \label{thm:main:est:1:BL}
| \p_x^k \p_y^j  (u_P - u_{FS; 0}) \langle \eta \rangle^{M} | \lesssim x^{-\frac12 +  \frac{1}{200} - k - \frac{j}{2}},
\end{align}
for all $0 \le k \le k_{max}, 0 \le j \le j_{max}, 0 \le M \le m_{max}$, where $k_{max}, j_{max}, m_{max}$ are determined by the initial data (as usual for parabolic IBVPs: the higher order compatibility conditions and the higher regularity on $u_{IN}$, the higher $j_{max}, k_{max}$ can be). 
\end{theorem}
Above, we leave the norm $E$ unspecified, though we remark it is a high order, weighted Sobolev norm. The work of \cite{MR4097332} relies on the so-called von-Mise transform, as did the original work of \cite{Serrin}. In the current work, we insist on avoiding the so-called von-Mise coordinate system and work entirely in physical variables (which is extremely important for applications to Navier-Stokes). In this sense, there is a parallel with the work \cite{IM21}, and also a parallel in the time-dependent setting with the works of \cite{MR3385340}, \cite{AWXY12} who introduced several new techniques to avoid using the Crocco transform (also motivated by eventual applicability to questions of Navier-Stokes convergence).   

We also refer the reader to the recent work of \cite{WangZhangBl}. Compared to the result of \cite{MR4097332}, the work \cite{WangZhangBl} uses a delicate maximum principle based approach and obtains logarithmic improvements on the decay rate. Moreover, the smallness assumption of the data in \cite{MR4097332} is replaced with a concavity assumption.

\vspace{2 mm}

\noindent \underline{FS Wedge-Flows $(0 < m < \infty)$:} As stated above, the main purpose of the present article is to generalize the work of \cite{MR4097332} for the complete family of Falkner-Skan boundary layer profiles. 

\begin{theorem}[Scattering Estimates, $m \ge 0$] \label{thm:1} Fix any $m \ge 0$. There exists an $\eps_\ast = \eps_\ast(m) > 0$ and numbers $k_0, M_0$ such that the following statement is valid. Assume for $\eps < \eps_\ast(m)$,  
\begin{align}
\| u_{P,IN} - u_{FS; m}(1, \cdot) \langle y \rangle^{M_0} \|_{H^{3k_0}_y} \le \eps. 
\end{align}
Assume further the parabolic compatibility conditions at the corner $(x, y) = (1, 0)$ up to order $k_0$ as described in \eqref{comp:cond:0}, \eqref{comp:cond:1} -- \eqref{comp:cond:2}. Then the following scattering estimates hold for $k \le k_0 -2$ and $j = 0, 1$: 
\begin{align} \label{thm:main:est:1}
| \p_x^k \p_y^j  (u_P - u_{FS; m}) \langle \eta \rangle^{M_{k,j}} | \lesssim x^{-(\frac14 + \frac{5m}{4} - \frac{1}{200}) - k - \frac{j}{2}(1 - m) },
\end{align}
where $M_{k,j} := M_0 - (k+1) - 2j$.
\end{theorem}

\begin{theorem}[Enhanced Scattering, $m \ge 0$] \label{thm:2} Fix any $m \ge 0$. There exists an $\eps_\ast = \eps_\ast(m) > 0$ and numbers $k_0, M_0$ such that the following statement is valid. Assume for $\eps < \eps_\ast(m)$, 
\begin{align}
\| u_{P,IN} - u_{FS; m}(1, \cdot) \langle y \rangle^{M_0} \|_{H^{3k_0}_y} \le \eps. 
\end{align}
Assume further the parabolic compatibility conditions at the corner $(x, y) = (1, 0)$ up to order $k_0$ as described in  \eqref{comp:cond:0}, \eqref{comp:cond:1} -- \eqref{comp:cond:2}. Then the following scattering estimates hold for $k \le k_0 -2$ and $j = 0, 1$: 
\begin{align} \label{thm:main:est:2}
| \p_x^k \p_y^j  (u_P - u_{FS; m}) \langle \eta \rangle^{M_{k,j}} | \lesssim x^{-(\frac12 + \frac{3}{2}m - \frac{1}{200}) - k - \frac{j}{2}(1 - m) },
\end{align}
where $M_{k,j} := M_0 -(k+1) - 2j - 1$.
\end{theorem}

\begin{remark} Even though Theorem \ref{thm:2} is strictly stronger than Theorem \ref{thm:1}, we have purposefully chosen to separate out the two theorems for the sake of presentation. The main new ideas we introduce are already contained in the proof of Theorem \ref{thm:1}, and therefore the reader may focus on these ideas which are implemented in Sections \ref{sec:2} -- \ref{sec:7}. On the other hand, essentially two additional ingredients are required to upgrade Theorem \ref{thm:1} to Theorem \ref{thm:2}. These are a virial estimate with an appropriately designed weight, which is then coupled with a Nash-type inequality. The presentation therefore allows for the isolation of these extra ingredients, which is contained in Section \ref{sec:8}. 
\end{remark}

\begin{remark} \textit{The role of $x^{-\frac{1}{200}}$:} The reader will notice the peculiar weight of $x^{-\frac{1}{200}}$ appearing in \eqref{thm:main:est:1}. Such a time-decaying weight is introduced to produce extra $CK$ terms that are used in turn to absorb nonlinear contributions. We chose the specific value of $x^{-\frac{1}{200}}$ as opposed to another parameter for the sake of making a concrete choice. It is likely that we can choose $x^{-\delta}$, where $0 < \delta << \delta_\ast(\eps)$ where $\delta_\ast(\eps)$ depends on the size of the initial datum. One may thus interpret this factor as allowing room for nonlinear fluctuations. It is an interesting question to determine the necessity of this loss-factor. 
\end{remark}

\begin{remark} \textit{The role of $\langle \eta \rangle^{M_{k,j}}$:} One of the main complicating factors in the case $m > 0$ compared to $m = 0$ is the fact that for $m > 0$, $v_{FS; m}$ grows linearly (like $\eta$) for $\eta \rightarrow \infty$. On the other hand, when $m = 0$, $v_{FS; 0}$ remains bounded when $\eta \rightarrow \infty$. This is quantified below in \eqref{v:deco}. This growth creates some analytical difficulties in our energy scheme, and is the reason we need to introduce a decreasing weight cascade as quantified by the $M_{k,j}$.   
\end{remark}

\begin{remark} \textit{The case $m = 0$:} In particular, when we set $m = 0$, we recover the scattering estimates in \cite{MR4097332}. Therefore, one interpretation of the present work is to generalize \cite{MR4097332} to all $m > 0$. Moreover, we observe a gain in decay in \eqref{thm:main:est:2} as $m$ increases. This gain is precisely our quantitative characterization of the favorable pressure gradient. As we will discuss below, this gain appears from the extraction of new negative definite CK terms due to the sign of the favorable pressure. 
\end{remark}

\begin{remark} \textit{The asymptotics as $m \rightarrow \infty$:} In this work, we do not pursue the asymptotics as $m \rightarrow \infty$. In particular, there is no claim regarding the uniformity of the basin of attraction $\lim_{m \rightarrow \infty} \eps_\ast(m)$. We leave this interesting question for future work. 
\end{remark}

\begin{remark} \textit{Sharp Scattering Rates:} It is unclear to the author if the rates obtained in this article are sharp. This is an interesting question to determine.  
\end{remark}

One immediate purpose of this paper is to answer the natural follow-up question to Serrin's result of establishing precise quantitative decay rates. However, perhaps an even more important purpose for future applications, is to introduce an energetic framework which is completely different from Serrin's approach (and that which uses crucially the structure of the background FS profiles) which could feasibly become applicable in more general settings. We turn now to an explanation of the ingredients in our framework.

\subsection{Main Ideas} 

Overall, the proof proceeds by introducing perturbations of the form $u_{FS; m} + \eps u$, $v_{FS; m} + \eps v$, and analyzing the resulting nonlinear equation on $[u, v]$. The linearized Prandtl operator around the FS profiles reads
\begin{align}
\mathcal{L}_m[u,v] = u_{FS; m} \p_x u + v_{FS; m} \p_y u + \p_x u_{FS; m} u + \p_y u_{FS; m} v - \p_y^2 u,
\end{align} 
which is the operator that in fact governs all higher $x$ derivatives of $[u, v]$ as well. 

\vspace{2 mm}

\noindent \underline{\textsc{Good Unknown \& Self-Similarity:}} One of the core issues in studying the Cauchy problem 
\begin{align}
&\mathcal{L}_m[u, v] = F, \\
&u|_{x = 0} = u_{IN}(y), \\
&u|_{y = 0} = 0, \qquad u|_{y = \infty} = 0, 
\end{align}
is the to determine the correct transport unknown. To address this issue, we follow the strategy used by the author with N. Masmoudi in the work \cite{IM20}, and work on the renormalized ``von-Mise" unknown:
\begin{align} \label{vmGU}
U =  \frac{1}{u_{FS; m}}(u - \frac{\p_y u_{FS; m}}{u_{FS; m}} \psi),
\end{align}
which enjoys subtle cancellations when paired with the linearized operator $\mathcal{L}$. 

Due to a possible loss of derivative at the top order of tangential derivative, it turns out we need to modify this renormalized good unknown with nonlinear quantities. More precisely, at the top order of derivative we define
\begin{align}
U_{\text{top}} =  \frac{1}{u_{FS; m} + \eps u }(u - \frac{\p_y u_{FS; m} + \eps \p_y u}{u_{FS; m} + \eps u} \psi).
\end{align} 
A large portion of the technical work in Section \ref{sec:3} is to translate between these unknowns and the original $[u, v]$ in appropriate norms. The process of doing so requires a precise understanding of the self-similar structure of $u_{FS; m}$, as weights in $\eta$ (and the corresponding tradeoffs between $y$ and $x^{\frac12(1-m)}$)play a crucial role. 

\vspace{2 mm}

\noindent \underline{\textsc{Favorable Pressure Gradient \& New CK Terms:}} One of our primary observation is to actually quantify how the ``favorable pressure" created by the outer Euler flow, $u_E(x) = x^m$ for $m \ge 0$ enters to create enhanced decay. Indeed, this is achieved through the presence of newly identified Cauchy-Kovalevskaya (``CK") Terms. Define 
\begin{align}
\mathcal{CK}^{(P)}(x) := & \int_{\mathbb{R}_+} (-\frac{3}{2}\p_x p_E(x))U^2 \ud y = \frac32 \int_{\mathbb{R}_+} m x^{2m-1} U^2  \ud y.
\end{align}
It turns out our use of the good unknown, $U$, \eqref{vmGU}, allows us to extract the CK term above in an energy argument: 
\begin{align*}
\frac{\p_x}{2} \mathcal{E}(x) + \mathcal{CK}^{(P)}(x) + \mathcal{D}(x) \le 0,
\end{align*} 
for appropriate energy-dissipation functionals. In turn, this pressure-CK term enhances the decay of the energy by a factor of $x^{3m}$, which explains the gain of decay in estimate \ref{thm:main:est:1} as compared to the $m = 0$ case in \eqref{thm:main:est:1:BL}. 
\vspace{2 mm}

\noindent \underline{\textsc{Downward Weight Cascade:}} One core issue that emerges when one adds $x$-dependence to the background $u_E(x)$ (in other words, when $m > 0$ as opposed to $m = 0$ in \eqref{power:law}) is the corresponding linear growth in $v_{FS}$. Indeed, by the divergence free condition 
\begin{align*}
v_{FS} = - \int_0^y \p_x u_{FS} = - \int_0^y \p_x u_{E}(x) \ud y' - \int_0^y\p_x [u_{FS} - u_E(x)] \ud y',
\end{align*}
and the first term on the right-hand side clearly grows in $y$ (but is absent when $m = 0$). In the context of a higher-order energy scheme, this creates the following type of difficulty. After one performs the basic energy estimate, one wants to compute the inner product of the linearized operator $\mathcal{L}_m$ with the differentiated unknown, $\p_x U$, (along with some weights in $x$ which for the purpose of this introduction, let us caricature this as $x^1$):
\begin{align} \n
\int_{\mathbb{R}_+} \mathcal{L}_m[u, v] \p_x U x \ud y = & \frac{\p_x}{2} \underbrace{\int_{\mathbb{R}_+} u_{FS; m} |\p_y U|^2 x \ud y}_{\text{Energy}} + \underbrace{\int_{\mathbb{R}_+} u_{FS; m}^2 |\p_x U|^2 x  \ud y}_{\text{Diffusion}} \\ \n
&+ \underbrace{ \int_{\mathbb{R}_+} v_{FS; m} \p_y U \p_x U x \ud y}_{\text{Growing Term}}  +  \text{ Controllable Commutator Terms}
\end{align}
The term ``growing term" above is referring to the growth of $v_{FS}$ at $y = \infty$. Indeed, due to the lack of structure in such a term, we are forced to put both the $\p_y U$ and the $\p_x U$ term in $L^2$, but due to the fact that $v_{FS}$ grows at $y = \infty$, we ``lose one weight" on the lower order quantity. 

Due to this structure, we introduce a ``downward cascade" of weights: we will choose the weight (really the weight in $\langle \eta \rangle$ as opposed to $y$) to depend as an inverse function of the regularity. This is seen more precisely in the form of our energy functionals, \eqref{dfr:1}, where the index of weight changes based on the regularity level. This is a core new feature compared to the Blasius work, \cite{MR4097332}. 

\vspace{2 mm}

\noindent \underline{\textsc{Virial-Type Coercivity:}} Such a downward cascade has appeared in local in $x$, or finite in $x$, results for example the work of \cite{IM22}. However, for global problems, implementing this downward cascade is challenging and requires some further structural properties which we crucially capitalize on in this work. Indeed, the starting point will be a weighted in $\eta$ estimate at the bottom order derivative. To make matters precise, we start the scheme by performing the inner product (for any $n \ge 1$): 
\begin{align}
\int_{\mathbb{R}_+} \mathcal{L}_m[u, v] U \langle \eta \rangle^{2n} \ud y.
\end{align}
Doing so produces commutator terms of the form: 
\begin{align} \n
&\text{Commutator Term with Fastest Growth at $n \rightarrow \infty$} \\ \label{bass:1}
& \qquad = \frac{n}{2}(1-m) \int_{\mathbb{R}_+} u_{FS; m}^2 U^2 \langle \eta \rangle^{2n-1} \eta x^{-1} \ud y  - n \int_{\mathbb{R}_+} u_{FS; m}v_{FS; m} U^2 \frac{ \langle \eta \rangle^{2n-1}}{x^{\frac12(1-m)}}  \ud y,
\end{align}
which would normally be controllable for small $x$ or finite $x$, but does not obviously integrate globally in $x$. Such a commutator term could potentially destroy a global energy scheme such as the one we are implementing. Our key observation is that the structure of $v_{FS; m}$ for $m \ge 0$ gives coercivity of this combination. More precisely, we have 
\begin{lemma} For all $m$, there exists a decomposition
\begin{align} \label{v:deco}
v_{FS; m}(x, y) = -m  \eta  x^{-\frac12(1-m)} + v_{FS,\ast}(x, y)x^{-\frac12(1-m)},
\end{align}
where $|v_{FS,\ast}| \lesssim 1$.
\end{lemma}
\begin{proof} We differentiate \eqref{psiFSform} in $x$ to obtain
\begin{align}
v_{FS; m} = - \sqrt{\frac{m+1}{2}} x^{\frac{m-1}{2}}(f(\xi) + \frac{m-1}{m+1}\xi f'(\xi)).
\end{align}
We subsequently use that the leading term as $\xi \rightarrow \infty$ is of the form 
\begin{align*}
f(\xi) + \frac{m-1}{m+1}\xi f'(\xi) = & f'(\xi)[\xi + \frac{m-1}{m+1}\xi] + E_1(\xi) \\
= &  (\xi + \frac{m-1}{m+1}\xi) + (f'(\xi) - 1) [\xi + \frac{m-1}{m+1}\xi]+ E_1(\xi) \\
= & (\xi + \frac{m-1}{m+1}\xi)  + E_2(\xi) \\
= & \frac{2m}{m+1}\xi + E_2(\xi)
\end{align*}
where $E_i(\xi)$ are bounded functions. Subsequently inserting above, we extract 
\begin{align*}
v_{FS; m} = & - \sqrt{\frac{m+1}{2}} x^{\frac{m-1}{2}} (\frac{2m}{m+1}\xi + E_2(\xi)) \\
= &  - \sqrt{\frac{m+1}{2}} x^{\frac{m-1}{2}} (\frac{2m}{m+1}\sqrt{\frac{m+1}{2}}\eta + E_2(\xi)) \\
= & ( - m \eta + v_{FS, \ast}) x^{\frac{m-1}{2}},
\end{align*}
as required in the lemma. 
\end{proof}

Inserting the representation \eqref{v:deco}, we essentially extract the leading order term from \eqref{bass:1} as 
\begin{align*}
\eqref{bass:1}  \sim &  \frac{n}{2}(1-m) \int_{\mathbb{R}_+} \bar{u}^2 |U_k|^2 \langle \eta \rangle^{2n}  x^{2k-1- \frac{1}{100}} \ud y + nm \int_{\mathbb{R}_+} \bar{u}^2 |U_k|^2 \langle \eta \rangle^{2n}  x^{2k-1- \frac{1}{100}} \ud y \\
 \ge & 0 \text{ for } m \ge 0. 
\end{align*}
This coercivity is the crucial ingredient which allows us to close the downward cascade energy scheme (at the linear level), and is more precisely seen in lines \eqref{beebop:1} -- \eqref{beebop:2}.

\section{Derivation of Canonical Systems} \label{sec:2}

In this section, we derive precisely the nonlinear equations on the quantities we will be analyzing. 

\subsection{Commutation Properties}

First of all, we use as shorthand notation 
\begin{align}
[\bar{u}, \bar{v}] = [u_{FS}, v_{FS}].
\end{align}
In this article, we consider small perturbations of the Falkner-Skan profiles: 
\begin{align}
u_P = \bar{u} + \eps u, \qquad v_P = \bar{v} + \eps v, \qquad 0 < \eps << 1. 
\end{align}
Around these profiles, we have the following system 
\begin{align}
\left. \begin{aligned} \label{lov:1}
&\bar{u} \p_x u + \bar{v} \p_y u + \bar{u}_x u + \bar{u}_y v - \p_y^2 u = -\eps \mathcal{Q}[u, v], \\
&\p_x u + \p_y v = 0, \\
&u|_{y = 0} = 0, \qquad \lim_{y \rightarrow \infty} u(x, y)= 0, \\
&u|_{x = 1} = u_{IN}(y) = u_{P,IN}(y) - u_{FS,m}(1,y), \\
&\mathcal{Q}[u, v]:= u \p_x u + v \p_y u, \\
\end{aligned} \right| (x, y) \in (1, \infty) \times \mathbb{R}_+
\end{align}
We now apply $\p_x^k$ of the above system. This produces for $k \ge 1$:
\begin{align}
\left. \begin{aligned}
&\bar{u} u^{(k)}_x + \bar{u}_x u^{(k)} + \bar{u}_y v^{(k)} + \bar{v} u^{(k)}_y - \p_y^2 u^{(k)} = F_{comm, k} - \eps \p_x^k \mathcal{Q}[u, v], \\
&\p_x u^{(k)} + \p_y v^{(k)} = 0, \\
&u^{(k)}|_{y = 0} = 0, \qquad \lim_{y \rightarrow \infty} u^{(k)}(x, y)= 0, \\
&u^{(k)}|_{x = 1} = u_{IN, k}(y), 
\end{aligned} \right| (x, y) \in (1, \infty) \times \mathbb{R}_+
\end{align}
where we define the forcing terms contributed by commutators via 
\begin{align}  \n
F_{comm, k} := &\sum_{k' = 0}^{k-1} \binom{k}{k'} \p_x^{k-k'} \bar{u} u^{(k'+1)} + \sum_{k' = 0}^{k-1} \binom{k}{k'} \p_x^{k-k'+1} \bar{u} u^{(k')} \\ \label{Fcommk}
& + \sum_{k' = 0}^{k-1} \binom{k}{k'} \p_x^{k-k'} \bar{v} u_y^{(k')} + \sum_{k' = 0}^{k-1} \binom{k}{k'} \p_x^{k-k'} \bar{u}_{y} v^{(k')} =: \sum_{i = 1}^4 F^{(i)}_{comm, k}.
\end{align}
We note that the Cauchy data, $u_{IN,k}(\cdot)$, for $k \ge 1$, will be computed as a functional of $u_{IN}$ itself in the standard manner using the equation. This is computed in Section \ref{sec:data}.

It is useful at this stage to the introduce a notation for the linearized Prandtl operator: 
\begin{align}
\mathcal{L}[u, v] := \bar{u} u_x + \bar{u}_x u + \bar{u}_y v + \bar{v} u_y - \p_y^2 u.
\end{align}

\subsection{Linear Good Variables \& Algebraic Structure}

We introduce the good-unknown: 
\begin{align} \label{QU}
Q_k := \frac{\psi^{(k)}}{\bar{u}}, \qquad U_k = \p_y Q_k = \p_y \Big( \frac{\psi^{(k)}}{\bar{u}} \Big) = \frac{1}{\bar{u}}(u^{(k)} - \frac{\bar{u}_y}{\bar{u}} \psi^{(k)}). 
\end{align}
We now perform some algebraic manipulations. First, we record the inversion formula
\begin{align} \label{inv:form}
u^{(k)} = \bar{u} U_k + \bar{u}_y Q_k
\end{align}
Next, we rewrite the entire transport part of the operator in terms of $U_k, Q_k$: 
\begin{align} \n
&\overline{u} \p_x u^{(k)} + u^{(k)} \p_x \overline{u} + \overline{v} \p_y u^{(k)} + v^{(k)} \p_y \overline{u} \\
= &\bar{u}^2 \p_x U_k + \bar{u} \bar{v} \p_y U_k + (\bar{u} \bar{u}_{xy} + \bar{v} \bar{u}_{yy})Q_k + 2 (\bar{u} \bar{u}_x + \bar{v} \bar{u}_y)U_k \\
= &\bar{u}^2 \p_x U_k + \bar{u} \bar{v} \p_y U_k + \bar{u}_{yyy}Q_k + 2 (\bar{u}_{yy} - \p_x p_E(x))U_k.
\end{align}
We thus retain the following form of the linearized Prandtl equation: 
\begin{align} \label{Pr:lin2}
\left.  \begin{aligned}
    &\bar{u}^2 \p_x U_k + \bar{u} \bar{v} \p_y U_k + \bar{u}_{yyy}Q_k + 2 (\bar{u}_{yy} - \p_x p_E(x))U_k  - \p_y^2 u^{(k)} = G_k  \\
    & u^{(k)}|_{y = 0} = 0, \qquad \lim_{y \rightarrow \infty} u^{(k)} = 0,   \\
    & u^{(k)}|_{x = 1} = u_{IN, k}(y)   \\
    &G_k :=  F_{comm, k} - \eps \p_x^k \mathcal{Q}[u, v]
        \end{aligned}\right| (x, y) \in (1, \infty) \times \mathbb{R}_+.
\end{align}

\subsection{Nonlinear Good Variables \& Algebraic Structure}

We start with the identity 
\begin{align} \n
\eps \p_x^k \mathcal{Q}[u, v] = & \eps \p_x^k \{ u \p_x u + \p_y u v \} \\ \n
= & \eps u \p_x u^{(k)} + \eps u_y v^{(k)} +  \eps \sum_{k' = 0}^{k-1} \binom{k}{k'} u^{(k-k')} u^{(k'+1)} + \p_y u^{(k-k')} v^{(k')} \\
= & \eps u \p_x u^{(k)} + \eps u_y v^{(k)} + \eps  \mathcal{Q}_{k, \text{lo}}[u, v],
\end{align}
where we defined 
\begin{align} \label{def:qklo}
\mathcal{Q}_{k, \text{lo}}[u, v] := \sum_{k' = 0}^{k-1} \binom{k}{k'} u^{(k-k')} u^{(k'+1)} + \p_y u^{(k-k')} v^{(k')}.
\end{align}
Motivated by the forthcoming identity, we denote our nonlinear corrections follows 
\begin{align}
\mu := \bar{u} + \eps u, \qquad \nu := \bar{v} + \eps v. 
\end{align}
We now rewrite our linear operator to include these ``quasilinear" terms as follows: 
\begin{align} \n
\mathcal{L}[u^{(k)}, v^{(k)}] + \eps \p_x^k \mathcal{Q}[u, v] = & (\overline{u} + \eps u) \p_x u^{(k)} + v^{(k)} ( \p_y \overline{u} + \eps u_y)  + u^{(k)} \p_x \overline{u} + \overline{v} \p_y u^{(k)} +  \eps  \mathcal{Q}_{k, \text{lo}}[u, v] \\
= & \mu \p_x u^{(k)}  + v^{(k)}  \p_y \mu  + u^{(k)} \p_x \overline{u} + \overline{v} \p_y u^{(k)} +  \eps  \mathcal{Q}_{k, \text{lo}}[u, v]
\end{align}
Motivated by the first two ``Rayleigh" terms appearing on the right-hand side above, we define 
\begin{align} \label{nl:defns}
\mathcal{Q}_k := \frac{\psi^{(k)}}{\mu}, \qquad \mathcal{U}_k := \p_y \mathcal{Q}_k, \qquad \mathcal{V}_k := - \p_x \mathcal{Q}_k.
\end{align}
From these identities, we have 
\begin{align}
u^{(k)} =& \p_y \psi^{(k)} = \p_y \{ \mu \mathcal{Q}_k \} = \mu \mathcal{U}_k + \p_y \mu \mathcal{Q}_k, \\
v^{(k)} =& - \p_x \psi^{(k)} = - \p_x \{ \mu \mathcal{Q}_k \} = \mu \mathcal{V}_k - \p_x \mu \mathcal{Q}_k 
\end{align}
after which we have 
\begin{align} \n
 \mu \p_x u^{(k)}  + v^{(k)}  \p_y \mu  = & \mu \p_x \{ \mu \mathcal{U}_k + \p_y \mu \mathcal{Q}_k \} + \p_y \mu \{   \mu \mathcal{V}_k - \p_x \mu \mathcal{Q}_k \} \\
 = & \mu^2 \p_x \mathcal{U}_k + \mu \mu_x \mathcal{U}_k + \mu \mu_{xy} \mathcal{Q}_k + \mu \p_y \mu \p_x \mathcal{Q}_k + \mu_{y} \mu \mathcal{V}_k - \mu_y \mu_x \mathcal{Q}_k \\
 = & \mu^2 \p_x \mathcal{U}_k + \mu \mu_x \mathcal{U}_k + (\mu \mu_{xy} - \mu_x \mu_y ) \mathcal{Q}_k 
\end{align}
Using these identities, we can write 
\begin{align}
\mathcal{L}[u^{(k)}, v^{(k)}] + \eps \p_x^k \mathcal{Q}[u, v] = &  \mu^2 \p_x \mathcal{U}_k + ( \mu \mu_x + \bar{u}_x \mu + 2 \bar{v} \mu_y ) \mathcal{U}_k \\
&+ (\mu \mu_{xy} + \bar{v} \mu_{yy} -( \mu_x - \bar{u}_x) \mu_y ) \mathcal{Q}_k + \bar{v} \mu \p_y \mathcal{U}_k \\
&+ \eps  \mathcal{Q}_{k, \text{lo}}[u, v]
\end{align}
We give a name to the coefficients appearing above: 
\begin{align}
\bold{a} := &\mu \mu_x + \bar{u}_x \mu + 2 \bar{v} \mu_y, \\
\bold{b} := &\mu \mu_{xy} + \bar{v} \mu_{yy} -( \mu_x - \bar{u}_x) \mu_y, \\
\bold{c} := & \bar{v} \mu. 
\end{align}
With these definitions in hand, we have our ``quasilinearized" system 
\begin{align} \label{ql:sys}
\mu^2 \p_x \mathcal{U}_k + \bold{a} \mathcal{U}_k + \bold{b} \mathcal{Q}_k + \bold{c} \p_y \mathcal{U}_k - \p_y^2 u^{(k)} = F_{comm,k} - \eps \mathcal{Q}_{k,\text{lo}}[u, v] := H_k. 
\end{align}

\subsection{Cauchy Data: Iteration Scheme \& Compatibility Conditions} \label{sec:data}

We set the notation
\begin{align*}
u_{IN; k}(y) & := u^{(k)}(1, y), \qquad \psi_{IN; k}(y) := \psi^{(k)}(1, y), \\
U_{IN; k}(y) &:= U_k|_{x = 0}(y), \qquad Q_{IN; k}(y) := Q_k|_{x = 0}(y). 
\end{align*}
First, we are prescribed 
\begin{align}
u_{IN}(y), \qquad \psi_{IN}(y) := \int_0^y u_{IN}(y') \ud y',
\end{align}
and from here we may immediately compute 
\begin{align}
U_{IN}(y) := \frac{1}{\bar{u}}(u_{IN} - \frac{\bar{u}_y}{\bar{u}} \psi_{IN}).
\end{align}
For this quantity to be well-defined we require 
\begin{align} \label{comp:cond:0}
\textbf{CC}_0: u_{IN}(0) = 0. 
\end{align}
Our interest is in obtaining the quantities $U_{IN;k}(y)$. To do so in the most streamlined way, we define the operator 
\begin{align}
\text{Ray}[U] := \bar{u} U + \bar{u}_y \int_0^y U, 
\end{align}
which allows us to convert $U_{IN;k}$ to $u_{IN;k}$ according to the relation \eqref{inv:form}. Second, motivated by the equation \eqref{Pr:lin2}, 
\begin{align}
\text{Jump}_k[U] := \frac{1}{\bar{u}^2}(G_k + \p_y^2 \text{Ray}[U] - \bar{u} \bar{v} \p_y U - \bar{u}_{yyy} \int_0^y U \ud y' - 2(\bar{u}_{yy} - \p_x p_E(x)) U ). 
\end{align}
The $\text{Jump}_k[\cdot]$ operator allows us to compute $\p_x U_k|_{x = 1}$ from $U_{IN;k}$ and $u_{IN;k} = \text{Ray}[U_{IN;k}]$. Due to the discrepancy between $\p_x U_k$ and $U_{k+1}$ (a computation yields $U_{k+1} =  \p_x  U_k + \frac{\bar{u}_x}{\bar{u}} U_k + \p_y \{ \frac{\bar{u}_x}{\bar{u}} \} Q_k$), we will need to modify the expression for $\p_x U_k|_{x = 1}$ by the two extra quantities. Therefore, we also define 
\begin{align}
\text{Shift}[U] := \frac{\bar{u}_x}{\bar{u}} U + \p_y \{ \frac{\bar{u}_x}{\bar{u}} \} \int_0^y U \ud y'. 
\end{align}
We then have our map 
\begin{align}
U_{IN; k+1} = \text{Jump}_k[U_{IN; k}] + \text{Shift}[U_{IN; k}]. 
\end{align}
Of course, inductively we can write 
\begin{align} \label{defTkdata}
U_{IN; k} = (\text{Jump}_{k-1} + \text{Shift}) \circ \dots  \circ (\text{Jump}_0 + \text{Shift})[ U_{IN;0}] =: \mathcal{T}_k[U_{IN}], k \ge 1,
\end{align}
(and we put $\mathcal{T}_0[U] = U$ to be the identity map). Therefore, we may state the $k'$th order compatibility conditions as follows: 
\begin{align*}
&\Big[G_{k-1} + \p_y^2 \text{Ray}[U_{IN; k-1}] - \bar{u} \bar{v} \p_y U_{IN; k-1} - \bar{u}_{yyy} \int_0^y U_{IN; k-1} \ud y' \\
& \qquad \qquad- 2(\bar{u}_{yy} - \p_x p_E(x)) U_{IN;k-1}\Big]|_{y = 0} = 0 \\
&\p_y|_{y = 0} \Big[G_{k-1} + \p_y^2 \text{Ray}[U_{IN; k-1}] - \bar{u} \bar{v} \p_y U_{IN; k-1} - \bar{u}_{yyy} \int_0^y U_{IN; k-1} \ud y' \\
& \qquad \qquad - 2(\bar{u}_{yy} - \p_x p_E(x)) U_{IN;k-1}\Big] = 0,
\end{align*}
which of course can be rewritten as 
\begin{align} \n
\bold{CC}_{k; 1}: &\Big[G_{k-1} + \p_y^2 \text{Ray}[\mathcal{T}_{k-1}[U_{IN}]] - \bar{u} \bar{v} \p_y \mathcal{T}_{k-1}[U_{IN}] - \bar{u}_{yyy} \int_0^y \mathcal{T}_{k-1}[U_{IN}] \ud y' \\ \label{comp:cond:1}
& \qquad \qquad - 2(\bar{u}_{yy} - \p_x p_E(x)) \mathcal{T}_{k-1}[U_{IN}]\Big]|_{y = 0} = 0, \\ \n
\bold{CC}_{k; 2}:  &\p_y|_{y = 0} \Big[G_{k-1} + \p_y^2 \text{Ray}[\mathcal{T}_{k-1}[U_{IN}]] - \bar{u} \bar{v} \p_y \mathcal{T}_{k-1}[U_{IN}] - \bar{u}_{yyy} \int_0^y \mathcal{T}_{k-1}[U_{IN}] \ud y' \\ \label{comp:cond:2}
& \qquad \qquad - 2(\bar{u}_{yy} - \p_x p_E(x)) \mathcal{T}_{k-1}[U_{IN}]\Big] = 0.
\end{align}
\section{Norms \& Embedding Theorems}\label{sec:3}

\subsection{Energy-CK-Dissipation Functionals}

We will define the following ``integer-level" energy-CK-dissipation functionals 
\begin{align}
\mathcal{E}_{k,n}(x) := &\int_{\mathbb{R}_+} \bar{u}^2 U_k^2 x^{2k - \frac{1}{100}} \langle \eta \rangle^{2n} \ud y, \\
\mathcal{CK}_{k,n}(x) := &\frac{1}{100} \int_{\mathbb{R}_+} \bar{u}^2 U_k^2 x^{2k -1 - \frac{1}{100}} \langle \eta \rangle^{2n} \ud y, \\
\mathcal{CK}_{k,n}^{(P)}(x) := & \int_{\mathbb{R}_+} (-\p_x p_E(x))U_k^2 x^{2k- \frac{1}{100}} \langle \eta \rangle^{2n} \ud y  \\
\mathcal{D}_{k,n}(x) := & \int_{\mathbb{R}_+} \bar{u} |\p_y U_k|^2 x^{2k- \frac{1}{100}} \langle \eta \rangle^{2n} \ud y, \\
\mathcal{B}_{k}(x) := &  \bar{u}_y |U_k(x, 0)|^2 x^{2k- \frac{1}{100}}.
\end{align}
Above, we have introduced the $x^{-\frac{1}{100}}$ weight in order to generate the quantity $\mathcal{CK}_{k,n}$. Formally, this term is of the same type as $\mathcal{CK}_{k,n}^{(P)}$, the CK-term generated by the favorable pressure gradient. However, the term $CK_{k,n}^{(P)}$ has a prefactor of $m$, and we would like to get estimates uniform as $m \downarrow 0$. Therefore, we use $\mathcal{CK}_{k,n}(x)$ in order to absorb the majority of error terms. 

We also need the following ``half-level" energy-CK-dissipation functionals: 
\begin{align}
\mathcal{E}^{(Y)}_{k+\frac12,n}(x) := & \int_{\mathbb{R}_+} \bar{u} |\p_y U_k|^2 x^{1 - \frac{1}{100}+ 2k} \langle \eta \rangle^{2n} \ud y, \\
\mathcal{D}^{(Y)}_{k + \frac12,n}(x) := & \int_{\mathbb{R}_+} \bar{u}^2 |\p_x U_k|^2 x^{1 - \frac{1}{100}+ 2k} \langle \eta \rangle^{2n} \ud y,
\end{align} 
and 
\begin{align}
\mathcal{E}^{(Z)}_{k+\frac12,n}(x) := & \int_{\mathbb{R}_+} \bar{u}^2 |\p_y U_k|^2 x^{2k + 1 - m- \frac{1}{100}}\langle \eta \rangle^{2n} \ud y, \\
\mathcal{D}^{(Z)}_{k + \frac12,n}(x) := &\int_{\mathbb{R}_+} \bar{u} |\p_y^2 U_k|^2 x^{1 + 2k-m- \frac{1}{100}} \langle \eta \rangle^{2n} \ud y, \\
\mathcal{B}^{(Z)}_{k + \frac12}(x) := &  \bar{u}_y |\p_y U_k(x, 0)|^2 x^{2k + 1 -m- \frac{1}{100}}
\end{align} 
Our total energy functional will be as follows: 
\begin{align} \label{dfr:1}
\mathcal{E}(x) := & \sum_{k = 0}^4  \sigma_{k} \mathcal{E}_{k,10-k}(x) + \sum_{k = 0}^4  (\sigma_{k + \frac12}^{(Y)} \mathcal{E}^{(Y)}_{k+ \frac12,9-k}(x)  + \sigma_{k + \frac12}^{(Z)} \mathcal{E}^{(Z)}_{k+ \frac12,9-k}(x)) +  \sigma_{5} \mathcal{E}_{5,0}(x), \\ \label{dfr:2}
\mathcal{D}(x) := &\sum_{k = 0}^4 \sigma_{k} \mathcal{D}_{k,10-k}(x) + \sum_{k = 0}^4  (\sigma_{k + \frac12}^{(Y)}  \mathcal{D}^{(Y)}_{k+ \frac12,9-k}(x) + \sigma_{k + \frac12}^{(Z)} \mathcal{D}^{(Z)}_{k+ \frac12,9-k}(x) ) + \sigma_{5} \mathcal{D}_{5,0}(x), \\ \label{dfr:3}
\mathcal{CK}(x) := &  \sum_{k = 0}^4 \sigma_{k} \mathcal{CK}_{k,10-k}(x) + \sigma_{5} \mathcal{CK}_{5,0}(x), \\ \label{dfr:4}
\mathcal{CK}^{(P)}(x) := &  \sum_{k = 0}^4 \sigma_{k} \mathcal{CK}^{(P)}_{k,10-k}(x) + \sigma_{5} \mathcal{CK}^{(P)}_{5,0}(x), \\ \label{dfr:5}
\mathcal{B}(x) := &\sum_{k = 0}^4 \sigma_{k} \mathcal{B}_{k}(x) + \sum_{k = 0}^4 \sigma_{k + \frac12}^{(Z)} \mathcal{B}^{(Z)}_{k+ \frac12}(x) + \sigma_{5} \mathcal{B}_{5}(x).
\end{align}

\begin{remark}The quantities $\sigma_{k}, \sigma_{k + \frac12}^{(Y)}, \sigma_{k + \frac12}^{(Z)}$ are scalar weights that will be chosen, and will satisfy the ordering: 
\begin{align} \n
\sigma_{k, 0} >& \sigma_{k + \frac12, 0}^{(Y)} = \sigma_{k + \frac12, 0}^{(Z)} > \sigma_{k+1, 0} \dots 
\end{align}
The purpose of these scalars is technical in nature and can be understood as follows: if we have a lower order norm, say $\| \cdot \|_{A}$, that is closed: $\| \cdot \|_A \lesssim \text{Data}$, and then a higher order norm estimated in terms of the lower order norm $\| \cdot \|_{B} \le C \| \cdot \|_A$, then we form our total norm by taking a weighted linear combination  $\| \cdot \|_A + \sigma \| \cdot \|_B \le \text{Data} + \sigma C \| \cdot \|_A$, which, if $\sigma C < 1$ closes an estimate for both $A$ and $B$. Our energy scheme will essentially consist of a stand-alone estimate for $\mathcal{E}_0$, and then each successive $\mathcal{E}_k$ will be estimated (almost) in terms of the previous order. Hence, including these weights $\sigma_k$ accounts for this ``iterative" structure. 
\end{remark}

We now introduce our quasilinear energy: 
\begin{align}
\overline{\mathcal{E}}_k(x) :=& \int_{\mathbb{R}_+} \mu |\mathcal{U}_k|^2 x^{2k- \frac{1}{100}} \ud y, \\
\overline{\mathcal{D}}_k(x) := & \int_{\mathbb{R}_+} \mu |\p_y \mathcal{U}_k|^2 x^{2k- \frac{1}{100}} \ud y, \\
\overline{\mathcal{CK}}_k(x) :=&\frac{1}{100} \int_{\mathbb{R}_+} \bar{u}^2 |\mathcal{U}_k|^2 x^{2k-1- \frac{1}{100}} \ud y, \\
\overline{\mathcal{CK}}^{(P)}_k(x) :=& \int_{\mathbb{R}_+} (-\p_xp_E(x)) |\mathcal{U}_k|^2 x^{2k- \frac{1}{100}} \ud y, \\
\overline{\mathcal{B}}_{k}(x) := &  \bar{\mu}_y |\mathcal{U}_k(x, 0)|^2 x^{2k- \frac{1}{100}}.
\end{align}
Our total \textit{quasilinear} energy functional will be as follows: 
\begin{align} \label{mellow:1}
\mathcal{E}_{\text{Quasi}}(x) := & \sum_{k = 0}^4 \sigma_{k} \mathcal{E}_{k,10-k}(x) + \sum_{k = 0}^4 ( \sigma^{(Y)}_{k + \frac12} \mathcal{E}^{(Y)}_{k+ \frac12,9-k}(x)  + \sigma^{(Z)}_{k + \frac12} \mathcal{E}^{(Z)}_{k+ \frac12,9-k}(x))+ \sigma_{5} \overline{\mathcal{E}}_5(x),  \\ \label{mellow:2}
\mathcal{D}_{\text{Quasi}}(x) := &\sum_{k = 0}^4  \sigma_{k} \mathcal{D}_{k,10-k}(x) + \sum_{k = 0}^4 (\sigma^{(Y)}_{k + \frac12} \mathcal{D}^{(Y)}_{k+ \frac12,9-k}(x) +\sigma^{(Z)}_{k + \frac12} \mathcal{D}^{(Z)}_{k+ \frac12,9-k}(x) ) + \sigma_{5} \overline{\mathcal{D}}_5(x),  \\ \label{mellow:3}
\mathcal{CK}_{\text{Quasi}}(x) := &  \sum_{k = 0}^4  \sigma_{k}  \mathcal{CK}_{k,10-k}(x) + \sigma_{5} \overline{\mathcal{CK}}_5(x), \\ \label{mellow:4}
\mathcal{CK}_{\text{Quasi}}^{(P)}(x) := &  \sum_{k = 0}^4  \sigma_{k}  \mathcal{CK}^{(P)}_{k,10-k}(x) + \sigma_{5} \overline{\mathcal{CK}}^{(P)}_5(x), \\ 
\mathcal{B}_{\text{Quasi}}(x) := &\sum_{k = 0}^4 \sigma_{k}  \mathcal{B}_{k}(x) + \sum_{k = 0}^4 \sigma^{(Z)}_{k + \frac12}  \mathcal{B}^{(Z)}_{k+ \frac12}(x) + \sigma_{5} \overline{\mathcal{B}}_{5}(x).
\end{align}
We will introduce some auxiliary notation to help simplify various expressions in the forthcoming bounds: let 
\begin{align} \n
\mathcal{I}_{\le k}(x) := & \sum_{k' = 0}^k (  \mathcal{D}_{k,0}(x) +   \mathcal{CK}_{k,0}(x) +   \mathcal{B}_{k}(x)   )+  \bold{1}_{k \ge 1} \sum_{k' = 0}^{k-1} ( \mathcal{D}^{(Y)}_{k + \frac12,0}(x) + \mathcal{D}^{(Z)}_{k + \frac12,0}(x) \\
& + \mathcal{B}^{(Z)}_{k + \frac12}(x)  ) , \\
\mathcal{I}_{\le k + \frac12}(x) := & \sum_{k' = 0}^k (  \mathcal{D}_{k,0}(x) +   \mathcal{CK}_{k,0}(x) +   \mathcal{B}_{k}(x)   )+   \sum_{k' = 0}^{k} ( \mathcal{D}^{(Y)}_{k + \frac12,0}(x) + \mathcal{D}^{(Z)}_{k + \frac12,0}(x) + \mathcal{B}^{(Z)}_{k + \frac12}(x)  ).
\end{align}
Because eventually our choice of weight (indexed by $n$) is tied to regularity (indexed by $k$), we also will have a need for the notation: 
\begin{align} \n
\widehat{\mathcal{I}}_{ k}(x) := & \sum_{k' = 0}^k  (  \mathcal{D}_{k,10-k}(x) +   \mathcal{CK}_{k,10-k}(x) +   \mathcal{B}_{k}(x)   )+  \bold{1}_{k \ge 1} \sum_{k' = 0}^{k-1}  ( \mathcal{D}^{(Y)}_{k + \frac12,9-k}(x) + \mathcal{D}^{(Z)}_{k + \frac12,9-k}(x) \\ \label{defn:hat:IK}
& + \mathcal{B}^{(Z)}_{k + \frac12}(x)  ) , \\ \n
\widehat{\mathcal{I}}_{k + \frac12}(x) := & \sum_{k' = 0}^k  (  \mathcal{D}_{k,10-k}(x) +   \mathcal{CK}_{k,10-k}(x) +   \mathcal{B}_{k}(x)   )+   \sum_{k' = 0}^{k}  ( \mathcal{D}^{(Y)}_{k + \frac12,9-k}(x) + \mathcal{D}^{(Z)}_{k + \frac12,9-k}(x) \\
&+ \mathcal{B}^{(Z)}_{k + \frac12}(x)  ).
\end{align}

Our bootstrap hypothesis will be the following:
\begin{align} \label{boots:1}
\sup_{0 \le x \le X_\ast} \mathcal{E}_{\text{Quasi}}(x) +  \int_0^{X_\ast} (\frac{1}{10}\mathcal{D}(x) +\frac{1}{10} \mathcal{CK}(x) +  \mathcal{CK}^{(P)}(x) +\frac{1}{10} \mathcal{B}(x) ) dx \le 10 \mathcal{E}_{\text{Init}}, 
\end{align}
\begin{remark} The proof will proceed by completing an energy estimate on the quasilinear energy/dissipation/CK functionals, \eqref{mellow:1} -- \eqref{mellow:4}. We then prove equivalence of the $L^1_t$ quantities $\mathcal{D}_{\text{Quasi}}, \mathcal{CK}_{\text{Quasi}}, \mathcal{B}_{\text{Quasi}}$ to their linear counterparts, which are bootstrapped in \eqref{boots:1}. We do not attempt to prove that $\mathcal{E}_{\text{Quasi}}$ is ``equivalent" to its linear counterpart $\mathcal{E}(x)$.
\end{remark}

\subsection{$L^2$ Bounds on Good Variables}

\begin{lemma} Let $f = f(y) \in L^2(\mathbb{R}_+)$. Let $0 < \gamma << 1$. Then 
\begin{align}\label{lam:1}
x^{2m-1}\|  f \|_{L^2_y}^2 \lesssim \frac{1}{\lambda^2} \| \bar{u} f x^{-\frac12} \|_{L^2_y}^2 + \lambda \| \sqrt{\bar{u}} f_y \|_{L^2_y}^2
\end{align} 

\end{lemma}
\begin{proof} We decompose the left-hand side into 
\begin{align} \label{hgu:1:2}
x^{2m-1}\|  f \|_{L^2_y}^2 \le& \int_{\mathbb{R}_+} x^{2m-1} f^2 \chi(\eta \le \lambda) \ud y + \int_{\mathbb{R}_+} x^{2m-1} f^2 \chi(\eta > \lambda) \ud y :=  I_{Near} + I_{Far}.
\end{align}
For the ``near" term, we use a Hardy type inequality as follows: 
\begin{align} \n
  I_{Near} = & \int_{\mathbb{R}_+} \p_y \{ y \} x^{2m-1} f^2 \chi(\eta \le \lambda) \ud y \\ \n
  = & - \int_{\mathbb{R}_+} y x^{2m-1} 2 f f_y \chi(\eta \le \lambda) \ud y - \frac{1}{\lambda} \int_{\mathbb{R}_+} y x^{2m-1} f^2 \frac{1}{x^{\frac12(1-m)}} \chi'(\eta \le \lambda) \ud y \\
  = &  I_{Near}^{(1)} +  I_{Near}^{(2)}.  
\end{align}
For $I_{Near}^{(1)}$ we proceed by using Cauchy-Schwartz as follows: 
\begin{align*}
|I_{Near}^{(1)}| \lesssim & \| f x^{m - \frac12} \|_{L^2_y} \| \frac{y}{x^{\frac12(1-m)}} x^{\frac{m}{2}} f_y \chi(\eta \le \lambda) \|_{L^2_y} \\
\le & \frac{1}{10} \| f x^{m - \frac12} \|_{L^2_y}^2 + \| \frac{y}{x^{\frac12(1-m)}} x^{\frac{m}{2}} f_y \chi(\eta \le \lambda) \|_{L^2_y}^2 \\
\le & \frac{1}{10} \| f x^{m - \frac12} \|_{L^2_y}^2 + \lambda \| \sqrt{ \frac{y}{x^{\frac12(1-m)}}} x^{\frac{m}{2}} f_y \chi(\eta \le \lambda) \|_{L^2_y}^2 \\
\le & \frac{1}{10} \| f x^{m - \frac12} \|_{L^2_y}^2 + \lambda \| \sqrt{\bar{u}} f_y \chi(\eta \le \lambda) \|_{L^2_y}^2,
\end{align*}
where we use the factor of $1/10$ to the left-hand side of \eqref{hgu:1:2}.

For the ``far" term, we have
\begin{align}
\frac{\bar{u}x^{-m}}{\lambda} \chi(\eta > \lambda) \ge  \chi(\eta > \lambda)
\end{align}
after which 
\begin{align}
I_{Far} \le & \int_{\mathbb{R}_+} x^{2m-1} f^2 \Big(\frac{\bar{u}x^{-m}}{\lambda}\Big)^2 \chi(\eta > \lambda) \ud y \le  \int_{\mathbb{R}_+} x^{2m-1}\bar{u}^2 f^2 x^{-1} \ud y.
\end{align}
\end{proof}

We next have the following interpolation inequalities: 
\begin{lemma} Given any $0 < \lambda << 1$, the following bounds are valid: 
\begin{align}  \label{train:1}
\| \p_y U_k x^{k + \frac{m}{2} - \frac{1}{100}} \langle \eta \rangle^n\|_{L^2}^2 \le & C_{\lambda} \mathcal{D}_{k,n}(x) + \lambda \mathcal{D}_{k + \frac12,0}^{(Z)}(x), \qquad k \ge 0, \\ \label{train:2}
\| \p_x U_k x^{k + \frac12 + m - \frac{1}{100}} \langle \eta \rangle^n\|_{L^2}^2 \le & C_\lambda \mathcal{D}_{k + \frac12, n}^{(Y)}(x) + \lambda \mathcal{D}_{k + 1, 0}(x) + \lambda ( \mathcal{D}_{k,0}(x) + \mathcal{CK}_{k,0}(x)), \qquad k \ge 0, \\ \label{train:3}
\| U_{k} x^{k - \frac12 +m - \frac{1}{100}} \langle \eta \rangle^n\|_{L^2}^2 \le &C_\lambda (\mathcal{D}_{(k-1) + \frac12, n}^{(Y)}(x) + \mathcal{CK}_{k-1} + \mathcal{D}_{k-1,0})+ \lambda \mathcal{D}_{k, 0}(x), \qquad k \ge 1.
\end{align}
\end{lemma}
\begin{proof}[Proof of \eqref{train:1}] First of all, we have 
\begin{align} \n
\| \p_y U_k x^{k + \frac{m}{2}- \frac{1}{100}} \langle \eta \rangle^n\|_{L^2}^2 \le & \| \p_y U_k x^{k + \frac{m}{2}- \frac{1}{100}} \langle \eta \rangle^n \chi(\eta \ge 1)\|_{L^2}^2 + \| \p_y U_k x^{k + \frac{m}{2}- \frac{1}{100}} \langle \eta \rangle^n \chi(\eta \le 1)\|_{L^2}^2 \\ \label{yas}
\le & \| \p_y U_k x^{k + \frac{m}{2}- \frac{1}{100}} \langle \eta \rangle^n \chi(\eta \ge 1)\|_{L^2}^2 + \| \p_y U_k x^{k + \frac{m}{2}- \frac{1}{100}}\|_{L^2}^2.
\end{align}
For the first term on the right-hand side, we may use the fact that $\bar{u} \chi(\eta \ge 1) \gtrsim x^m \chi(\eta \ge 1)$, and therefore 
\begin{align*}
 \| \p_y U_k x^{k + \frac{m}{2}- \frac{1}{100}} \langle \eta \rangle^n \chi(\eta \ge 1)\|_{L^2}^2 \lesssim  \| \sqrt{\bar{u}} \p_y U_k x^{k- \frac{1}{100}} \langle \eta \rangle^n \chi(\eta \ge 1)\|_{L^2}^2 \lesssim \mathcal{D}_{k,n}(x).
\end{align*}
It therefore remains for us to estimate the second term on the right-hand side of \eqref{yas}, which follows immediately from \eqref{lam:1}. 
\end{proof}
\begin{proof}[Proof of \eqref{train:2}] As before, we localize based on $\eta$ as follows:
\begin{align*}
\| \p_x U_k x^{k + \frac12 + m- \frac{1}{100}} \langle \eta \rangle^n\|_{L^2}^2 \le & \| \p_x U_k x^{k + \frac12 + m- \frac{1}{100}} \langle \eta \rangle^n \chi(\eta \ge 1)\|_{L^2}^2 + \| \p_x U_k x^{k + \frac12 + m- \frac{1}{100}} \langle \eta \rangle^n \chi(\eta \le 1)\|_{L^2}^2 \\
\lesssim & \| \bar{u} \p_x U_k x^{k + \frac12- \frac{1}{100}} \langle \eta \rangle^n\|_{L^2}^2 +  \| \p_x U_k x^{k + \frac12 + m- \frac{1}{100}} \langle \eta \rangle^n \chi(\eta \le 1)\|_{L^2}^2 \\
\lesssim & \mathcal{D}_{k + \frac12, n}^{(Y)} +  \| \p_x U_k x^{k + \frac12 + m- \frac{1}{100}} \|_{L^2}^2.
\end{align*}
Therefore, it suffices to estimate the latter contribution above. To do so, we first invoke \eqref{lam:1} to bound 
\begin{align*}
 \| \p_x U_k x^{k + \frac12 + m- \frac{1}{100}} \|_{L^2}^2 \le& \lambda  \| \sqrt{\bar{u}} \p_x \p_y U_k x^{k + 1- \frac{1}{100}} \|_{L^2}^2 + C_\lambda \| \bar{u} \p_x U_k x^{k + \frac12- \frac{1}{100}} \|_{L^2}^2 \\
 \le & \lambda  \| \sqrt{\bar{u}} \p_x \p_y U_k x^{k + 1- \frac{1}{100}} \|_{L^2}^2 + C_\lambda  \mathcal{D}_{k + \frac12, n}^{(Y)}.
\end{align*}
Therefore, it now suffices to estimate the first quantity appearing on the right-hand side of the previous line. To do so, we first record the identity 
\begin{align}
\bar{u} Q_{k+1} = \psi^{(k+1)} = \p_x \psi^{(k)} = \p_x \{ \bar{u} Q_k \} = \bar{u}_x Q_k + \bar{u} \p_x Q_k,
\end{align}
from which the following identities follow upon further differentiating in $y$: 
\begin{align} \label{union:1}
Q_{k+1} = & \p_x Q_k + \frac{\bar{u}_x}{\bar{u}} Q_k, \\ \label{union:2}
U_{k+1} = & \p_x  U_k + \frac{\bar{u}_x}{\bar{u}} U_k + \p_y \{ \frac{\bar{u}_x}{\bar{u}} \} Q_k, \\ \label{union:3}
\p_y U_{k+1} = & \p_{x} \p_y U_k + \frac{\bar{u}_x}{\bar{u}} \p_y U_k + 2 \p_y \{\frac{\bar{u}_x}{\bar{u}} \} U_k + \p_y^2 \{ \frac{\bar{u}_x}{\bar{u}} \} Q_k.
\end{align}
From here, it follows that 
\begin{align*}
\| \sqrt{\bar{u}} \p_{x} \p_y U_k x^{k+1- \frac{1}{100}} \|_{L^2_y}^2 \lesssim & \mathcal{D}_{k+1}(x) + \| \frac{\bar{u}_x}{\bar{u}} \sqrt{\bar{u}} \p_y U_k x^{k+1- \frac{1}{100}} \|_{L^2_y} + \| x^{\frac{m}{2}} \p_y \{ \frac{\bar{u}_x}{\bar{u}} \} U_k x^{k+1- \frac{1}{100}} \|_{L^2_y}^2 \\
& + \| \p_y^2 \{ \frac{\bar{u}_x}{\bar{u}} \} Q_k x^{\frac{m}{2}- \frac{1}{100}} \|_{L^2_y}^2 \\
\lesssim &  \mathcal{D}_{k+1}(x) + \Big\| \frac{\bar{u}_x}{\bar{u}} x \Big\|_{L^\infty}^2 \| \sqrt{\bar{u}} \p_y U_k x^{k- \frac{1}{100}} \|_{L^2}^2 \\
&+ \Big\| \p_y \{ \frac{\bar{u}_x}{\bar{u}} \} x^{\frac{3}{2} - \frac{m}{2}} \Big\|_{L^\infty}^2 \| U_k x^{k + m - \frac12- \frac{1}{100}} \|_{L^2_y}^2 \\
& +  \Big\| y \p_y^2 \{ \frac{\bar{u}_x}{\bar{u}} \} x^{\frac{3}{2} - \frac{m}{2}} \Big\|_{L^\infty}^2 \| \frac{Q_k}{y} x^{k + m - \frac12- \frac{1}{100}} \|_{L^2_y}^2 \\
\lesssim &  \mathcal{D}_{k+1}(x) +  \mathcal{D}_{k}(x) + \mathcal{CK}_k(x). 
\end{align*}
This now completes the proof of this lemma. 
\end{proof}
\begin{proof}[Proof of \eqref{train:3}] Here we rely upon \eqref{union:2} applied with $k-1$ instead of $k$ which reads upon rearrangement: 
\begin{align}
U_k = \p_x U_{k-1} + \frac{\bar{u}_x}{\bar{u}} U_{k-1} + \p_y \{ \frac{\bar{u}_x}{\bar{u}} \} Q_{k-1}.
\end{align}
We apply \eqref{train:2} to the first term on the right-hand side. For the remaining two terms, we have 
\begin{align*}
\|  \frac{\bar{u}_x}{\bar{u}} U_{k-1} x^{k -\frac12 + m- \frac{1}{100}} \langle \eta \rangle^n \|_{L^2_y} \lesssim & \| \frac{\bar{u}_x}{\bar{u}} x \|_{L^\infty} \|  U_{k-1} x^{(k-1) -\frac12 + m- \frac{1}{100}} \|_{L^2_y} \lesssim \sqrt{\mathcal{CK}_{k-1,n}}, \\
\|   \p_y \{\frac{\bar{u}_x}{\bar{u}}\} Q_{k-1} x^{k -\frac12 + m- \frac{1}{100}}\langle \eta \rangle^n \|_{L^2_y} \lesssim & \| \langle \eta \rangle^n y \p_y \{ \frac{\bar{u}_x}{\bar{u}}\} x \|_{L^\infty} \|  U_{k-1} x^{(k-1) -\frac12 + m- \frac{1}{100}} \|_{L^2_y} \lesssim \sqrt{\mathcal{CK}_{k-1,0}}.
\end{align*}
\end{proof}

\subsection{$L^p_x L^q_y$ Bounds on Original Variables}

We define the following functions of $x$: 
\begin{align}
d_{k,n}(x) := & \int_{\mathbb{R}_+} |u^{(k)}_y|^2 x^{-m+2k- \frac{1}{100}} \langle \eta \rangle^{2n} \ud y, \\
c_{k,n}(x) := & \int_{\mathbb{R}_+} |u^{(k)}|^2 x^{2k-1- \frac{1}{100}} \langle \eta \rangle^{2n} \ud y, \\
d_{k + \frac12, n}^{(Y)}(x) := & \int_{\mathbb{R}_+} |u^{(k)}_x|^2 x^{2k+1- \frac{1}{100}} \langle \eta \rangle^{2n} \ud y, \\
d_{k + \frac12, n}^{(Z)}(x) := &  \int_{\mathbb{R}_+} |u^{(k)}_{yy}|^2 x^{2k+1-2m- \frac{1}{100}} \langle \eta \rangle^{2n} \ud y,
\end{align}
The purpose of these functions is that they approximately ``mirror" the functionals $\mathcal{D}_{k,n}$, $\mathcal{CK}_{k,n}$, $\mathcal{D}_{k + \frac12,n}^{(Y)}$, and  $\mathcal{D}_{k + \frac12,n}^{(Z)}$ respectively in the original $u^{(k)}$ unknown, as opposed to the good-unknown, $U_k$. It is convenient to have a means to measure $u^{(k)}$ itself for instance in Lemmas \ref{lem:61}, \ref{lem:62}, \ref{lem:63}, as well as for the Sobolev embeddings which appear below (\eqref{jurassic:1} -- \eqref{trader:joes:1}). This ``equivalence" is quantified in the following lemma.   
\begin{lemma} The functions $d_{k,n}(x)$, $c_{k,n}(x)$, $d_{k + \frac12, n}^{(Y)}(x)$, and $d_{k + \frac12, n}^{(Z)}(x)$ satisfy the following estimates: 
\begin{align} \label{G:League:na:1}
d_{k,n}(x) \lesssim & \mathcal{D}_{k,n}(x) + \mathcal{CK}_{k,n}(x), \\ \label{G:League:na:2}
c_{k,n}(x) \lesssim &\mathcal{CK}_{k,n}(x), \\ \label{G:League:na:3}
d_{k + \frac12, n}^{(Y)}(x) \lesssim & \lambda D_{k+1,0}(x) + C_{\lambda} (D^{(Y)}_{k + \frac12,n}(x) + \mathcal{I}_{\le k, 0}(x) ), \\ \label{G:League:na:4}
d_{k + \frac12, n}^{(Z)}(x) \lesssim &  \mathcal{D}_{k + \frac12, n}^{(Z)}(x)  + \mathcal{I}_{\le k, 0}(x).
\end{align}
\end{lemma}
\begin{proof}[Proof of \eqref{G:League:na:1}] We use \eqref{inv:form} to obtain the identity 
\begin{align} \label{mama:1}
\p_y u^{(k)} = \bar{u} \p_y U_k + 2 \bar{u}_y U_k + \bar{u}_{yy}Q_k, 
\end{align}
after which we compute 
\begin{align*}
d_{k,n}(x) \lesssim & \int \bar{u}^2 |\p_y U_k|^2 x^{-m + 2k- \frac{1}{100}} \langle \eta \rangle^{2n} \ud y + \int \bar{u}_y^2 |U_k|^2 x^{-m + 2k- \frac{1}{100}} \langle \eta \rangle^{2n} \ud y \\
&+ \int \bar{u}_{yy}^2 |Q_k|^2 x^{-m + 2k- \frac{1}{100}} \langle \eta \rangle^{2n} \ud y \\
\lesssim & \mathcal{D}_{k,n}(x) + ( \| \bar{u}_y x^{\frac12 - \frac{3m}{2}} \langle \eta \rangle^{2n} \|_{L^\infty}^2 + \| y \bar{u}_{yy} x^{\frac12 - \frac{3m}{2}} \langle \eta \rangle^{2n} \|_{L^\infty}^2 \Big) \| U_k x^{m + k - \frac12- \frac{1}{100}} \|_{L^2_y}^2 \\
\lesssim & \mathcal{D}_{k,n}(x) + \mathcal{CK}_{k,0}(x).
\end{align*}
\end{proof}
\begin{proof}[Proof of \eqref{G:League:na:2}] We directly use \eqref{inv:form} to estimate 
\begin{align*}
c_{k,n}(x) \lesssim & \int \bar{u}^2 U_k^2 x^{2k-1} \langle \eta \rangle^{2n} \ud y + \int \bar{u}_y^2 Q_k^2 x^{2k-1- \frac{1}{100}}\langle \eta \rangle^{2n} \ud y \\
\lesssim & ( \| \bar{u} x^{-m} \|_{L^\infty}^2 + \| y \p_y \bar{u} x^{-m}\langle \eta \rangle^{n} \|_{L^\infty}^2  ) \| U_k x^{k + m - \frac12- \frac{1}{100}} \langle \eta \rangle^{n} \|_{L^2_y}^2 \\
\lesssim & \mathcal{CK}_{k,n}(x). 
\end{align*}
\end{proof}
\begin{proof}[Proof of \eqref{G:League:na:3}] We apply \eqref{inv:form} with index $k + 1$ to generate the estimate 
\begin{align*}
d_{k + \frac12, n}^{(Y)}(x) \lesssim & \int \bar{u}^2 |U_{k+1}|^2 x^{2k+1- \frac{1}{100}} \langle \eta \rangle^{2n} + \int \bar{u}_y^2 Q_{k+1}^2 x^{2k+1- \frac{1}{100}} \langle \eta \rangle^{2n} \\
\lesssim & (\| \bar{u} x^{-m} \|_{L^\infty}^2 + \| y \bar{u}_y x^{-m} \langle \eta \rangle^n \|_{L^\infty}^2) \| U_{k+1} x^{(k+1) - \frac12 + m- \frac{1}{100}} \|_{L^2_y}^2 \\
\lesssim & \lambda D_{k+1,0}(x) + C_{\lambda} (D^{(Y)}_{k + \frac12,n}(x) + \mathcal{CK}_{k}(x) + \mathcal{D}_{k,0}(x) ),
\end{align*}
where we have invoked the bound \eqref{train:3} in the final step. 

\end{proof}
\begin{proof}[Proof of \eqref{G:League:na:4}] We differentiate further the identity \eqref{mama:1} to obtain 
\begin{align} \label{mama:1}
\p_y^2 u^{(k)} = \bar{u} \p_y^2 U_k + 2 \bar{u}_y \p_y U_k + 2 \bar{u}_{yy}U_k + \bar{u}_{yyy} Q_k.
\end{align}
From this, we have 
\begin{align*}
d_{k + \frac12, n}^{(Z)}(x) \lesssim & \int \bar{u}^2 |\p_y^2 U_k|^2 x^{2k+1-2m- \frac{1}{100}} \langle \eta \rangle^{2n} +  \int \bar{u}_y^2 |\p_y U_k|^2 x^{2k+1-2m- \frac{1}{100}} \langle \eta \rangle^{2n} \\
& + \int \bar{u}_{yy}^2 |U_k|^2 x^{2k+1-2m- \frac{1}{100}} \langle \eta \rangle^{2n} +  \int \bar{u}_{yyy}^2 |Q_k|^2 x^{2k+1-2m- \frac{1}{100}} \langle \eta \rangle^{2n} \\
\lesssim & \| \bar{u} x^{-m} \|_{L^\infty} \mathcal{D}_{k + \frac12, n}^{(Z)} + \| \bar{u}_y x^{\frac12 - \frac{3m}{2}} \langle \eta \rangle^{2n} \|_{L^\infty}^2 \| \p_y U_k x^{k + \frac{m}{2}- \frac{1}{100}} \|_{L^2_y}^2 \\
& +( \| \bar{u}_{yy} x^{1 - 2m} \langle \eta \rangle^{2n} \|_{L^\infty}^2 + \| y \bar{u}_{yyy} x^{1 - 2m} \langle \eta \rangle^{2n} \|_{L^\infty}^2) \| U_k x^{k + m - \frac12- \frac{1}{100}} \|_{L^2_y}^2 \\
\lesssim & \mathcal{D}_{k + \frac12, n}^{(Z)}  + \mathcal{I}_{\le k, 0}. 
\end{align*}
\end{proof}

We will need the following $L^\infty$ bounds to close our nonlinear bootstrap. 

\begin{lemma} Assume the bootstrap \eqref{boots:1}. Let $0 \le k \le 4$. Then the following bounds hold:
\begin{align} \label{jurassic:1}
\| u^{(k)} x^{\frac14 + k - \frac{m}{4}- \frac{1}{200}} \|_{L^\infty} \le & \eps^{-\frac12} 
\end{align}
Let $0 \le k \le 3$. Then the following bounds hold:
\begin{align} \label{jurassic:2}
\| u^{(k)} x^{\frac14 + k - \frac{m}{4}- \frac{1}{200}} \langle \eta \rangle^3 \|_{L^\infty} \le &  \eps^{-\frac12} , \\  \label{jurassic:4}
\| u^{(k)}_y x^{\frac34 + k - \frac{3m}{4}- \frac{1}{200}} \langle \eta \rangle^3 \|_{L^\infty} \le &  \eps^{-\frac12}, \\  \label{jurassic:3}
\| \frac{1}{\bar{u}} u^{(k)} x^{\frac14 + k + \frac{3m}{4}- \frac{1}{200}} \langle \eta \rangle^3 \|_{L^\infty} \le &   \eps^{-\frac12}.
\end{align}
Let $0 \le k \le 2$. Then the following bound holds:
\begin{align} \label{jurassic:5}
\| v^{(k)} x^{\frac34 + k + \frac{m}{4}- \frac{1}{200}} \|_{L^\infty} \le &  \eps^{-\frac12}\\ \label{jurassic:6}
\| \frac{1}{\bar{u}} v^{(k)} x^{\frac34 + k + \frac{5m}{4}- \frac{1}{200}} \|_{L^\infty} \le &  \eps^{-\frac12}.
\end{align}
Let $0 \le k \le 1$. Then 
\begin{align} \label{trader:joes:1}
\| u^{(k)}_y \langle \eta \rangle^8 x^{\frac34 + k - \frac{3m}{4}- \frac{1}{200}} \|_{L^\infty} \le  \eps^{-\frac12}. 
\end{align}
\end{lemma}
\begin{proof}[Proof of \eqref{jurassic:1}] We have 
\begin{align*}
&|u^{(k)}|^2 \langle x \rangle^{\frac12 + 2k - \frac{m}{2}- \frac{1}{100}} \\
= & \int_y^\infty u^{(k)} \p_y u^{(k)} \langle x \rangle^{\frac12 + 2k - \frac{m}{2}- \frac{1}{100}} \\
= & \int_y^\infty u^{(k)} \p_y u^{(k)} \langle x \rangle^{\frac12 + 2k - \frac{m}{2}- \frac{1}{100}} \Big|_{x = 0}+  \int_y^\infty\int_0^x u_x^{(k)} \p_y u^{(k)} \langle x \rangle^{\frac12 + 2k - \frac{m}{2}- \frac{1}{100}} \\
& + \int_y^\infty\int_0^x u^{(k)} \p_y u^{(k+1)} \langle x \rangle^{\frac12 + 2k - \frac{m}{2}- \frac{1}{100}} +C_{k,m}  \int_y^\infty\int_0^x u^{(k)} \p_y u^{(k)} \langle x \rangle^{-\frac12 + 2k - \frac{m}{2}- \frac{1}{100}} \\
\lesssim & C(u_{IN,k}) + \| \sqrt{ d_{k,0}} \|_{L^2_x} \| \sqrt{ d^{(Y)}_{k + \frac12,0}} \|_{L^2_x} + \| \sqrt{ c_{k,0}} \|_{L^2_x} \| \sqrt{ d_{k + 1,0}} \|_{L^2_x} \\
& +  \| \sqrt{ c_{k,0}} \|_{L^2_x} \| \sqrt{ d_{k,0}} \|_{L^2_x}
\end{align*}
\end{proof}
\begin{proof}[Proof of \eqref{jurassic:2}] We have 
\begin{align} \n
&|u^{(k)}|^2 \langle x \rangle^{\frac12 + 2k - \frac{m}{2}- \frac{1}{100}} \langle \eta \rangle^6 \\ \n
= & \int_y^\infty u^{(k)} \p_y u^{(k)} \langle x \rangle^{\frac12 + 2k - \frac{m}{2}- \frac{1}{100}} \langle \eta \rangle^6 +6 \int_y^\infty |u^{(k)}|^2 \langle \eta \rangle^5 x^{2k- \frac{1}{100}} \\ \n
= & \int_y^\infty u^{(k)} \p_y u^{(k)} \langle x \rangle^{\frac12 + 2k - \frac{m}{2}- \frac{1}{100}} \langle \eta \rangle^6 \Big|_{x = 0}+  \int_y^\infty\int_0^x u_x^{(k)} \p_y u^{(k)} \langle x \rangle^{\frac12 + 2k - \frac{m}{2}- \frac{1}{100}}\langle \eta \rangle^6 \\ \n
& + \int_y^\infty\int_0^x u^{(k)} \p_y u^{(k+1)} \langle x \rangle^{\frac12 + 2k - \frac{m}{2}- \frac{1}{100}}\langle \eta \rangle^6 +C_{k,m}  \int_y^\infty\int_0^x u^{(k)} \p_y u^{(k)} \langle x \rangle^{-\frac12 + 2k - \frac{m}{2}- \frac{1}{100}} \langle \eta \rangle^6\\ \n
& +\int_y^\infty |u^{(k)}|^2 \langle \eta \rangle^5 x^{2k- \frac{1}{100}}\Big|_{x = 0} + \int_y^\infty \int_0^x u^{(k)} u^{(k)}_x \langle \eta \rangle^5 x^{2k- \frac{1}{100}} + C_k \int_y^\infty \int_0^x |u^{(k)}|^2 \langle \eta \rangle^5 x^{2k-1- \frac{1}{100}} \\ \n
\lesssim  & C(u_{IN,k}) + \| \sqrt{ d_{k,3}} \|_{L^2_x} \| \sqrt{ d^{(Y)}_{k + \frac12,3}} \|_{L^2_x} + \| \sqrt{ c_{k,3}} \|_{L^2_x} \| \sqrt{ d_{k + 1,3}} \|_{L^2_x} \\ \label{greer:1}
& +  \| \sqrt{ c_{k,3}} \|_{L^2_x} \| \sqrt{ d_{k,3}} \|_{L^2_x} +  \| \sqrt{ c_{k,3}} \|_{L^2_x} \| \sqrt{ d^{(Y)}_{k + \frac12,3}} \|_{L^2_x} +  \| \sqrt{ c_{k,3}} \|_{L^2_x}^2.
\end{align}
\end{proof}
\begin{proof}[Proof of \eqref{jurassic:4}] We have 
\begin{align} \n
&|u^{(k)}_y|^2 \langle x \rangle^{\frac32 + 2k - \frac{3m}{2}- \frac{1}{100}} \langle \eta \rangle^6 \\ \n
= & \int_y^\infty 2 u^{(k)}_y u^{(k)}_{yy} \langle x \rangle^{\frac32 + 2k - \frac{3m}{2}- \frac{1}{100}} \langle \eta \rangle^6 \ud y +  6\int_y^\infty 2 |u^{(k)}_y|^2 \langle x \rangle^{1 + 2k - m- \frac{1}{100}} \langle \eta \rangle^5 \ud y \\ \n
= & C(u_{IN,k}) +  \int_y^\infty \int_0^x  u^{(k+1)}_y u^{(k)}_{yy} \langle x \rangle^{\frac32 + 2k - \frac{3m}{2}- \frac{1}{100}} \langle \eta \rangle^6 \ud y +  \int_y^\infty  \int_0^x u^{(k)}_y u^{(k+1)}_{yy} \langle x \rangle^{\frac32 + 2k - \frac{3m}{2}- \frac{1}{100}} \langle \eta \rangle^6 \ud y \\ \n
& + C_1  \int_y^\infty  \int_0^x u^{(k)}_y u^{(k)}_{yy} \langle x \rangle^{1 + 2k - \frac{3m}{2}- \frac{1}{100}} \langle \eta \rangle^6 \ud y + C_2\int_y^\infty \int_0^x  u^{(k)}_y u^{(k+1)}_y \langle x \rangle^{1 + 2k - m- \frac{1}{100}} \langle \eta \rangle^5 \ud y \\ \n
& + C_3 \int_y^\infty \int_0^x |u^{(k)}_y|^2 \langle x \rangle^{ 2k - m- \frac{1}{100}} \langle \eta \rangle^5 \ud y \\ \n
\lesssim & C(u_{IN,  k}) + \| \sqrt{d_{k+1,3}} \|_{L^2_x} \| \sqrt{d_{k + \frac12,3}^{(Z)}} \|_{L^2_x} + \| \sqrt{d_{k,3}} \|_{L^2_x} \| \sqrt{d_{(k+1) + \frac12,3}^{(Z)}} \|_{L^2_x} +  \| \sqrt{d_{k,3}} \|_{L^2_x} \| \sqrt{d_{k + \frac12,3}^{(Z)}} \|_{L^2_x} \\ \label{reprep:1}
& +  \| \sqrt{d_{k,3}} \|_{L^2_x} \| \sqrt{d_{k+1,3}} \|_{L^2_x} +  \| \sqrt{d_{k,3}} \|_{L^2_x}^2.
\end{align}
\end{proof}
\begin{proof}[Proof of \eqref{jurassic:3}] This bound follows trivially from \eqref{jurassic:2} and \eqref{jurassic:4}. 
\end{proof}
\begin{proof}[Proof of \eqref{jurassic:5}] We have 
\begin{align*}
v^{(k)} x^{\frac34 + k + \frac{m}{4}- \frac{1}{100}}  = &  x^{\frac34 + k + \frac{m}{4}- \frac{1}{100}} \int_0^y u^{(k)}_x(x, y') \ud y' =  x^{\frac34 + k + \frac{m}{4}- \frac{1}{100}} \int_0^y u^{(k)}_x(x, y') \langle \eta \rangle^3 \frac{1}{\langle \eta \rangle^3} \ud y' \\
\lesssim & x^{\frac34 + k + \frac{m}{4}- \frac{1}{100}} \| u^{(k+1)} \langle \eta \rangle^3 \|_{L^\infty_y} \int_0^\infty \frac{1}{\langle \eta \rangle^3} \ud y \\
\lesssim & x^{\frac34 + k + \frac{m}{4}- \frac{1}{100}} \| u^{(k+1)} \langle \eta \rangle^3 \|_{L^\infty_y} x^{\frac12(1-m)},
\end{align*}
after which the result follows upon invoking \eqref{jurassic:2} with index $k + 1$ (which is admissible due to the restriction of $k \le 2$). 
\end{proof}
\begin{proof}[Proof of \eqref{jurassic:6}] First of all, we notice that when $\eta \ge 1$, $\bar{u} \gtrsim x^{m}$, and so the result trivially follows from \eqref{jurassic:5}. Therefore, we assume $\eta \le 1$, which is equivalent to $y \lesssim x^{\frac12(1-m)}$. In this case, $\bar{u} \gtrsim x^{m} \eta$, and we have \begin{align*}
v^{(k)} x^{\frac34 + k + \frac{5m}{4}- \frac{1}{100}} \frac{1}{\bar{u}} \chi(\eta \le 1) \lesssim &\frac{ v^{(k)} x^{\frac54 + k - \frac{m}{4}- \frac{1}{100}}}{y} \\
  = &  x^{\frac54 + k - \frac{m}{4}- \frac{1}{100}} \frac{1}{y} \int_0^y u^{(k)}_x(x, y') \ud y' =  x^{\frac54 + k - \frac{m}{4}- \frac{1}{100}} \| u^{(k)}_x\|_{L^\infty_y},
\end{align*}
after which the result follows upon invoking \eqref{jurassic:2} with index $k + 1$ (which is again admissible due to the restriction of $k \le 2$). 
\end{proof}
\begin{proof}[Proof of \eqref{trader:joes:1}] We follow exactly the proof of \eqref{reprep:1} with a change of weight of $\langle \eta \rangle^8$ instead of $\langle \eta \rangle^3$ which gives 
\begin{align} \n
&|u^{(k)}_y|^2 \langle x \rangle^{\frac32 + 2k - \frac{3m}{2}- \frac{1}{100}} \langle \eta \rangle^{16} \\ \n
\lesssim & C(u_{IN,  k}) + \| \sqrt{d_{k+1,8}} \|_{L^2_x} \| \sqrt{d_{k + \frac12,8}^{(Z)}} \|_{L^2_x} + \| \sqrt{d_{k,9}} \|_{L^2_x} \| \sqrt{d_{(k+1) + \frac12,7}^{(Z)}} \|_{L^2_x} +  \| \sqrt{d_{k,8}} \|_{L^2_x} \| \sqrt{d_{k + \frac12,8}^{(Z)}} \|_{L^2_x} \\ \label{reprep:2}
& +  \| \sqrt{d_{k,8}} \|_{L^2_x} \| \sqrt{d_{k+1,8}} \|_{L^2_x} +  \| \sqrt{d_{k,8}} \|_{L^2_x}^2.
\end{align}
\end{proof}

\subsection{Quasilinearized Estimates}

In addition to our main estimates (given in the previous two lemmas), we will will need the following auxiliary estimate. Define 
\begin{align}
\alpha(x) := & \int_{\mathbb{R}_+}  |\p_y \{ \frac{u}{\bar{u}} \}|^2 \langle \eta \rangle^4  x^{1+m- \frac{1}{100}} \ud y, \\
\gamma(x) := & \int_{\mathbb{R}_+} \bar{u}^2 |\p_y^2 \{ \frac{u}{\bar{u}} \}|^2 \langle \eta \rangle^4  x^{2-2m- \frac{1}{100}} \ud y.  
\end{align}

\begin{lemma} The quantities $\alpha(x), \gamma(x)$ obey the following bounds:
\begin{align} \label{ehg:alpha}
\sup_x \alpha(x) \le& \eps^{-1}, \\ \label{ehg:gamma}
\sup_x  \gamma(x) \le& \eps^{-1}.
\end{align}
\end{lemma}
\begin{proof} The bounds on both $\alpha(x)$ and $\gamma(x)$ are analogous. Therefore, we prove the higher order bound, \eqref{ehg:gamma}, which is a bit more involved. According to our bootstrap assumption, \eqref{boots:1}, it suffices to establish: 
\begin{align}
\sup_x  \gamma(x)   \lesssim \| \mathcal{D}(x) \|_{L^1_x} + \| \mathcal{CK}(x) \|_{L^1_x} + \| \mathcal{B}(x) \|_{L^1_x}.
\end{align}
Moreover, through $1D$ Sobolev interpolation, it suffices to establish the following two bounds: 
\begin{align} \label{ffu:1}
\| \bar{u} \p_y^2 \{ \frac{u}{\bar{u}} \} x^{\frac12 -m- \frac{1}{100}} \langle \eta \rangle^2  \|_{L^2_{y}}^2 \lesssim &  \mathcal{D}(x) + \mathcal{CK}(x) + \mathcal{B}(x), \\ \label{ffu:2}
\| \p_x \{ \bar{u} \p_y^2 \{ \frac{u}{\bar{u}} \} x^{1 -m- \frac{1}{100}} \langle \eta \rangle^2 \} x^{\frac12}  \|_{L^2_{y}}^2 \lesssim & \mathcal{D}(x) + \mathcal{CK}(x) + \mathcal{B}(x). 
\end{align}
Both of these bounds work in largely the same manner, so for simplicity we prove \eqref{ffu:1}. To proceed, we localize as follows: 
\begin{align*}
\| \bar{u} \p_y^2 \{ \frac{u}{\bar{u}} \} x^{\frac12 -m- \frac{1}{100}} \langle \eta \rangle^2  \|_{L^2_{y}}^2 \lesssim & \| \bar{u} \p_y^2 \{ \frac{u}{\bar{u}} \} x^{\frac12 -m- \frac{1}{100}} \langle \eta \rangle^2 \chi(\eta \ge 1)  \|_{L^2_{y}}^2 + \| \bar{u} \p_y^2 \{ \frac{u}{\bar{u}} \} x^{\frac12 -m- \frac{1}{100}} \chi(\eta \le 1)  \|_{L^2_{y}}^2 \\
=: & E_{\text{Far}} + E_{\text{Near}}.
\end{align*}
The estimation of $E_{\text{Far}}$ is the simpler of the two. Indeed, an application of the product rule coupled with the fact that $\bar{u} \chi(\eta \ge 1) \gtrsim x^m \chi(\eta \ge 1)$ gives 
\begin{align*}
E_{\text{Far}} \lesssim & \| u_{yy} x^{\frac12 -m- \frac{1}{100}} \langle \eta \rangle^2 \chi(\eta \ge 1)  \|_{L^2_{y}}^2 + \| u_{y} \frac{1}{\bar{u}} \bar{u}_y \langle \eta \rangle^2 x^{\frac12 -m- \frac{1}{100}} \langle \eta \rangle^2 \chi(\eta \ge 1)  \|_{L^2_{y}}^2 \\
& + \| u \frac{\bar{u}_y^2}{\bar{u}^2} \langle \eta \rangle^2 x^{\frac12-m- \frac{1}{100}} \|_{L^2_y}^2 + \| u \frac{\bar{u}_{yy}}{\bar{u}} \langle \eta \rangle^2 x^{\frac12 -m- \frac{1}{100}} \|_{L^2_y}^2 \\
\lesssim & d^{(Z)}_{0 + \frac12, 2}(x) + \| u_y x^{-\frac{m}{2}} \|_{L^2_y}^2  \\
\lesssim & d^{(Z)}_{0 + \frac12, 2}(x) + d_0(x),
\end{align*} 
upon which we invoke \eqref{G:League:na:1} and \eqref{G:League:na:4}.

At this stage, we introduce the following technique to estimate $E_{\text{Near}}$. Define $f(\eta)$ to be a smooth, rapidly decaying function which satisfies $f(0) = 0$ and $f'(0) = 1$. Subsequently decompose
\begin{align*}
 E_{\text{Near}} =& \| \bar{u} \p_y^2 \{ \frac{u - x^{\frac12(1-m)}f(\eta) u_y(x, 0)}{\bar{u}} \} x^{\frac12 -m- \frac{1}{100}} \chi(\eta \le 1)  \|_{L^2_{y}}^2 \\
 &+ \| \bar{u} \p_y^2 \{ \frac{ x^{\frac12(1-m)}f(\eta) u_y(x, 0)}{\bar{u}} \} x^{\frac12 -m- \frac{1}{100}} \chi(\eta \le 1)  \|_{L^2_{y}}^2 =:  E_{\text{Near}, 1} +  E_{\text{Near}, 2}.
\end{align*}
To further clarify matters, define $\hat{u} := u - x^{\frac12(1-m)}f(\eta) u_y(x, 0)$, and notice that $\hat{u}|_{y = 0} = \p_y \hat{u}|_{y = 0} = 0$. In this case, we estimate $E_{\text{Near}, 1}$ as follows 
\begin{align*}
E_{\text{Near}, 1} \lesssim & \| \hat{u}_{yy} x^{\frac12 -m- \frac{1}{100}} \|_{L^2_y}^2 + \| \bar{u}_y \frac{\hat{u}_y}{\bar{u}} \chi(\eta \le 1) x^{\frac12 -m- \frac{1}{100}}\|_{L^2_y}^2 + \| \bar{u}_{yy} \frac{\hat{u}}{\bar{u}}  \chi(\eta \le 1) x^{\frac12 -m- \frac{1}{100}}\|_{L^2_y}^2 \\
&+ \| \bar{u}_y^2 \frac{\hat{u}}{\bar{u}^2}  \chi(\eta \le 1)x^{\frac12 -m- \frac{1}{100}}\|_{L^2_y}^2 \\
\lesssim & \| \hat{u}_{yy} x^{\frac12 -m- \frac{1}{100}} \|_{L^2_y}^2 + \| \bar{u}_y x^{\frac12(1-m) -m} \frac{\hat{u}_y}{y} \chi(\eta \le 1) x^{\frac12 -m- \frac{1}{100}}\|_{L^2_y}^2 \\
& + \| y x^{\frac12(1-m) - m} \bar{u}_{yy} \frac{\hat{u}}{y^2}  \chi(\eta \le 1) x^{\frac12 -m- \frac{1}{100}}\|_{L^2_y}^2 + \| \bar{u}_y^2 x^{(1-m) - 2m} \frac{\hat{u}}{y^2}  \chi(\eta \le 1)x^{\frac12 -m- \frac{1}{100}}\|_{L^2_y}^2 \\
\lesssim & \| \hat{u}_{yy} x^{\frac12 -m- \frac{1}{100}} \|_{L^2_y}^2 + \| \bar{u}_y x^{\frac12(1-m) -m} \frac{\hat{u}_y}{y} x^{\frac12 -m- \frac{1}{100}}\|_{L^2_y}^2 \\
& + \| y x^{\frac12(1-m) - m} \bar{u}_{yy} \frac{\hat{u}}{y^2} x^{\frac12 -m- \frac{1}{100}}\|_{L^2_y}^2 + \| \bar{u}_y^2 x^{(1-m) - 2m} \frac{\hat{u}}{y^2} x^{\frac12 -m- \frac{1}{100}}\|_{L^2_y}^2 \\
\lesssim & \| \hat{u}_{yy} x^{\frac12-m- \frac{1}{100}} \|_{L^2_y}^2 \\
\lesssim & \| u_{yy} x^{\frac12-m- \frac{1}{100}} \|_{L^2_y}^2 + \| \p_y^2 \{ f(\eta) x^{\frac12(1-m)} u_y(x, 0)  \} x^{\frac12-m- \frac{1}{100}} \|_{L^2_y}^2 \\
\lesssim & d^{(Z)}_{0 + \frac12,0}(x) + \| f''(\eta) x^{-(1-m)} x^{\frac12(1-m)} u_y(x, 0)   x^{\frac12-m- \frac{1}{100}} \|_{L^2_y}^2 \\
\lesssim & d^{(Z)}_{0 + \frac12,0}(x) + \| u_y(x, 0)   x^{\frac14(1-m) - \frac{m}{2}- \frac{1}{100}} \|_{L^x_y}^2 \\
\lesssim & d^{(Z)}_{0 + \frac12,0}(x) + \mathcal{B}_0(x),
\end{align*}
upon which we invoke \eqref{G:League:na:4}. The estimate of $E_{\text{Near},2}$ is essentially identical to that of $E_{\text{Far}}$. 
\end{proof}

Given the stated bootstraps, we are now able to establish the following equivalence of norms. 
\begin{lemma} Assuming the bootstrap \eqref{boots:1}, the following bounds are valid. 
\begin{align} \label{norm:equi:1}
|\mathcal{D}_5(x) - \overline{\mathcal{D}}_5(x)| \le & \eps^{\frac14} \mathcal{I}_{5}(x), \\ \label{norm:equi:2}
|\mathcal{CK}_5(x) - \overline{\mathcal{CK}}_5(x)| \le & \eps^{\frac14} \mathcal{I}_{5}(x), \\ \label{norm:equi:3}
|\mathcal{B}_5(x) - \overline{\mathcal{B}}_5(x)| \le &  \eps^{\frac14} \mathcal{I}_{5}(x), 
\end{align}
\end{lemma}
\begin{proof}[Proof of \eqref{norm:equi:1}] We compute $\overline{\mathcal{D}}_5$ using the following identities 
\begin{align}
\mathcal{Q}_k = & \frac{\bar{u}}{\mu} Q_k, \\ \label{red64:red64}
\mathcal{U}_k =& \frac{\bar{u}}{\mu} U_k + \p_y \{ \frac{\bar{u}}{\mu} \} Q_k,  \\
\p_y \mathcal{U}_k = & \frac{\bar{u}}{\mu} \p_y U_k + 2 \p_y \{ \frac{\bar{u}}{\mu} \} U_k +  \p_y^2 \{ \frac{\bar{u}}{\mu} \} Q_k,
\end{align}
of which we choose to linearize the final identity as follows:
\begin{align}
\sqrt{\mu} \p_y \mathcal{U}_k = &\sqrt{\bar{u}} \p_y U_k + \Big( \frac{\bar{u}}{\mu} \frac{\sqrt{\mu}}{\sqrt{\bar{u}}} - 1 \Big)\sqrt{\bar{u}} \p_y U_k + 2 \sqrt{\mu} \p_y \{ \frac{\bar{u}}{\mu} \} U_k + \sqrt{\mu} \p_y^2 \{ \frac{\bar{u}}{\mu} \} Q_k,
\end{align}
from which we obtain 
\begin{align} \n
|\overline{\mathcal{D}}_{5,0}(x) - \mathcal{D}_{5,0}(x)| \lesssim &  \|  \frac{\bar{u}}{\mu} \frac{ \sqrt{\mu}}{\sqrt{\bar{u}}} - 1 \Big\|_{L^\infty}^2  \| \sqrt{\bar{u}} \p_y U_5 x^{5- \frac{1}{100}}\|_{L^2_y}^2  + \| \sqrt{\bar{u}} \p_y \{ \frac{\bar{u}}{\mu} \} U_5 x^{5- \frac{1}{100}} \|_{L^2_y}^2 \\ \n
 &+ \| \sqrt{\bar{u}} \p_y^2 \{ \frac{\bar{u}}{\mu} \} Q_5 x^{5- \frac{1}{100}} \|_{L^2_y}^2  \\ \n
 =: & \sum_{i = 1}^3 \text{Err}_i. 
\end{align}
To estimate $\text{Err}_1$, we notice that 
\begin{align*}
\Big\|  \frac{\bar{u}}{\mu} \frac{ \sqrt{\mu}}{\sqrt{\bar{u}}} - 1 \Big\|_{L^\infty} = & \Big\|   \frac{ \sqrt{\bar{u}}}{\sqrt{\mu}} - 1 \Big\|_{L^\infty} = \|  \sqrt{1 - \frac{\eps u}{\mu}} - 1 \Big\|_{L^\infty} \lesssim \eps \| \frac{u}{\bar{u}}\|_{L^\infty} \\
\lesssim & \eps \| \frac{u}{\bar{u}} \chi(\eta \le 1)\|_{L^\infty} + \eps \| \frac{u}{\bar{u}} \chi(\eta > 1)\|_{L^\infty} \\
\lesssim & \eps x^{-m} \| u \|_{L^\infty} + \eps \| \frac{\int_0^y u_y}{x^m y} x^{\frac12(1-m)} \chi(\eta \le 1) \|_{L^\infty} \\
\lesssim & \eps x^{-m} \| u \|_{L^\infty} +  \eps \| u_y \|_{L^\infty} x^{\frac12 - \frac{3m}{2}} \\
\lesssim& \eps^{\frac12} x^{-\frac14 - \frac{3m}{4}+ \frac{1}{100}},
\end{align*}
where we have invoked \eqref{jurassic:1} and \eqref{jurassic:4}. To estimate $\text{Err}_2$, we proceed as follows: 
\begin{align*}
|\text{Err}_2| \lesssim & \| \sqrt{\bar{u}} \p_y \{ \frac{\bar{u}}{\mu} \} \|_{L^\infty_y}^2 x^{1 - 2m} \| U_5 x^{4.5+m- \frac{1}{100}} \|_{L^2_y}^2 \\
\lesssim & (\|  \p_y \{ \frac{\bar{u}}{\mu} \} \|_{L^2_y}^2 x^{-\frac12 + \frac{3m}{2}} + x^{\frac12 - \frac{3m}{2}} \| \bar{u} \p_y^2 \{ \frac{\bar{u}}{\mu} \} \|_{L^2_y}^2)\| U_5 x^{4.5+m- \frac{1}{100}} \|_{L^2_y}^2 x^{1 - 2m} \\
\lesssim & \eps^2 x^{-\frac12 - \frac{3m}{2} + \frac{2}{100}} (\alpha(x) + \gamma(x)) \mathcal{I}_5(x),
\end{align*}
where we have invoked the bounds \eqref{ehg:alpha} -- \eqref{ehg:gamma}. To estimate $\text{Err}_3$, we proceed as follows. First notice that 
\begin{align*}
|Q_5| \lesssim \sqrt{y} \| U_5 \|_{L^2_y}.
\end{align*}
Inserting this, we obtain
\begin{align*}
|\text{Err}_3| \lesssim & \| \sqrt{\bar{u}} \sqrt{y} \p_y^2 \{ \frac{\bar{u}}{\mu} \} \|_{L^2_y}^2 x^{1-2m} \| U_5 x^{4.5+m- \frac{1}{100}} \|_{L^2_y}^2 \\
\lesssim & \| \sqrt{\bar{u}} \sqrt{\eta} x^{\frac14(1-m)} \p_y^2 \{ \frac{\bar{u}}{\mu} \} \|_{L^2_y}^2 x^{1-2m} \| U_5 x^{4.5+m- \frac{1}{100}} \|_{L^2_y}^2 \\
\lesssim & \| \bar{u} x^{-\frac{m}{2}} \langle \eta \rangle x^{\frac14(1-m)} \p_y^2 \{ \frac{\bar{u}}{\mu} \} \|_{L^2_y}^2 x^{1-2m} \| U_5 x^{4.5+m- \frac{1}{100}} \|_{L^2_y}^2 \\
\lesssim & \eps^2 x^{-\frac12 - \frac{3m}{2} + \frac{2}{100}} \gamma(x) \mathcal{I}_5(x).
\end{align*}
where we have used the inequality $\sqrt{\eta} \lesssim x^{-m} \sqrt{\bar{u}} \langle \eta \rangle$. 
\end{proof}
\begin{proof}[Proof of \eqref{norm:equi:2}] For this estimate, we choose to linearize \eqref{red64:red64} in the following manner: 
\begin{align} \label{blue:68}
\mathcal{U}_k =&U_k +  \Big(\frac{\bar{u}}{\mu} -1\Big) U_k + \p_y \{ \frac{\bar{u}}{\mu} \} Q_k
\end{align}
From here, we get 
\begin{align*}
|\mathcal{CK}_5(x) - \overline{\mathcal{CK}}_5(x)| \lesssim & \| x^{4.5 + m- \frac{1}{100}}\Big(\frac{\bar{u}}{\mu} -1\Big) U_k \|_{L^2_y}^2 + \| \p_y \{ \frac{\bar{u}}{\mu} \} Q_k  x^{4.5 + m- \frac{1}{100}} \|_{L^2_y}^2 \\
\lesssim & \eps^2 \| u \|_{L^\infty}^2 \mathcal{CK}_5(x) + \Big\| y \p_y \{ \frac{\bar{u}}{\mu} \} \Big\|_{L^\infty_y}^2 \| \frac{Q_k}{y} x^{4.5+m- \frac{1}{100}} \|_{L^2_y}^2 \\
\lesssim & \eps \mathcal{CK}_5(x) + \eps^2 x^{1-2m} \| \sqrt{\bar{u}} \p_y \{ \frac{\bar{u}}{\mu} \} \|_{L^\infty_y}^2 \mathcal{CK}_5(x) \\
\lesssim & \eps \mathcal{CK}_5(x) + \eps x^{-\frac12 - \frac{3m}{2} + \frac{2}{100}} \mathcal{CK}_5(x)
\end{align*}

\end{proof}
\begin{proof}[Proof of \eqref{norm:equi:3}] From \eqref{blue:68}, we have 
\begin{align}
\sqrt{\mu_y} \mathcal{U}_k =& \sqrt{\bar{u}_y} U_k + ( \sqrt{\bar{u}_y}  -  \sqrt{\mu_y}  ) U_k + \sqrt{\mu_y}  \Big(\frac{\bar{u}}{\mu} -1\Big) U_k + \sqrt{\mu_y} \p_y \{ \frac{\bar{u}}{\mu} \} Q_k.
\end{align}
Evaluation at $y = 0$ and using thatt $Q_k|_{y = 0}$ results in 
\begin{align}
|\mathcal{B}_5(x) - \overline{\mathcal{B}_5}(x)| \lesssim \bold{a}(x)^2 | \sqrt{\bar{u}_y} U_5(x,0)x^{5- \frac{1}{100}}|^2 \lesssim \bold{a}(x)^2 \mathcal{B}_5(x)
\end{align}
where
\begin{align}
\bold{a}(x) := x^{\frac14- \frac{3m}{4}}( ( \sqrt{\bar{u}_y}  -  \sqrt{\mu_y}  )  +  \sqrt{\mu_y}  \Big(\frac{\bar{u}}{\mu} -1\Big)) =: \bold{a}_1(x) + \bold{a}_2(x). 
\end{align}
To estimate $\bold{a}_1(x)$, we have 
\begin{align*}
\bold{a}_1(x) = & x^{\frac14- \frac{3m}{4}} \sqrt{\bar{u}_y} \Big(1 - \sqrt{\frac{\mu_y}{\bar{u}_y}} \Big) =  x^{\frac14- \frac{3m}{4}} \sqrt{\bar{u}_y} \Big(1 - \sqrt{\frac{\bar{u}_y + \eps u_y}{\bar{u}_y}} \Big) \\
\lesssim & \eps \| \frac{u_y}{\bar{u}_y}|_{y = 0} \|_{L^\infty}  \\
\lesssim & \eps^{\frac12} x^{-\frac14 - \frac{3m}{4}+ \frac{1}{100}},
\end{align*}
where we have used \eqref{jurassic:4}. To estimate $\bold{a}_2(x)$, we have 
\begin{align*}
\bold{a}_2(x) \lesssim & x^{\frac14- \frac{3m}{4}} x^{-\frac14 + \frac{3m}{4}} \eps \| \frac{u}{\bar{u}} \|_{L^\infty} \lesssim  \eps^{\frac12} x^{-\frac14 - \frac{3m}{4} + \frac{1}{100}},
\end{align*}
where we have used \eqref{jurassic:3}.
\end{proof}

We can immediately deduce a ``quasilinearized" version of \eqref{train:3}. 
\begin{corollary} Given any $0 < \lambda << 1$, the following bounds are valid: 
\begin{align} \label{ohm:1}
\| \mathcal{U}_{5} x^{4.5 +m- \frac{1}{100}} \|_{L^2}^2 \le &C_\lambda \mathcal{I}_{4 + \frac12, 0}(x) + \lambda \mathcal{D}_{5,0}(x) + \eps^{\frac14} \mathcal{I}_5(x), \\ \label{ohm:2}
\| \mathcal{U}_{5} x^{4.5 +m- \frac{1}{100}} \|_{L^2}^2 \le &C_\lambda \mathcal{I}_{4 + \frac12, 0}(x) + \lambda \overline{\mathcal{D}}_{5,0}(x) + \eps^{\frac14} \mathcal{I}_5(x).
\end{align}
\end{corollary}

\section{Linear Estimates} \label{sec:4}

\begin{lemma} \label{X:est:1} For any $k, n \ge 0$, the following energy inequality is valid: 
\begin{align} \n
&\frac{\p_x}{2} \mathcal{E}_{k,n}(x) + \mathcal{CK}_{k,n}(x) + \mathcal{CK}^{(P)}_{k,n}(x) +  \mathcal{B}_{k}(x) + \mathcal{D}_{k,n}(x) \\ \label{cal:1}
& \qquad \lesssim  \delta_{n > 0} \mathcal{CK}_{k,n-1}(x) +  \delta_{k > 0} \mathcal{CK}_{k-1,n}(x) +C |\bold{S}_{k,n}(x)|, 
\end{align}
where 
\begin{align} \label{Skn:def:1}
\bold{S}_{k,n}(x) := \int_{\mathbb{R}_+} G_k U_k x^{2k - \frac{1}{100}} \langle \eta \rangle^{2n} \ud y. 
\end{align}
\end{lemma}
\begin{proof} We multiply equation \eqref{Pr:lin2} by $U_k \langle \eta \rangle^{2n} x^{2k- \frac{1}{100}}$. We split the treatment of the resulting terms into three steps: Transport, Diffusion, Source, as follows: 
\begin{align}
\bold{T}_{k,n}(x) := &\int_{\mathbb{R}_+} (\bar{u}^2 \p_x U_k + \bar{u} \bar{v} \p_y U_k + \bar{u}_{yyy}Q_k + 2 (\bar{u}_{yy} - \p_x p_E(x))U_k) U_k \langle \eta \rangle^{2n} x^{2k- \frac{1}{100}} \ud y, \\
\bold{D}_{k,n}(x) := & - \int_{\mathbb{R}_+}  \p_y^2 u^{(k)} U_k \langle \eta \rangle^{2n} x^{2k- \frac{1}{100}} \ud y \\
\bold{S}_{k,n}(x) := &  \int_{\mathbb{R}_+} G_k U_k \langle \eta \rangle^{2n} x^{2k- \frac{1}{100}} \ud y,
\end{align}
after which we have
\begin{align} \label{mj:1}
\bold{T}_{k,n}(x) + \bold{D}_{k,n}(x) = \bold{S}_{k,n}(x).
\end{align}

\vspace{2 mm}

\noindent \textit{Transport Terms:} We have 
\begin{align} \n
\bold{T}_{k,n}(x)  = &\int_{\mathbb{R}_+} (\bar{u}^2 \p_x U_k + \bar{u} \bar{v} \p_y U_k + \bar{u}_{yyy}Q_k + 2 (\bar{u}_{yy} - \p_x p_E(x))U_k) U_k \langle \eta \rangle^{2n} x^{2k - \frac{1}{100}} \ud y \\ \n
= & \frac{\p_x}{2} \int_{\mathbb{R}_+} \bar{u}^2 |U_k|^2 \langle \eta \rangle^{2n} x^{2k- \frac{1}{100}} \ud y - \int_{\mathbb{R}_+} \bar{u} \bar{u}_x |U_k|^2 \langle \eta \rangle^{2n} x^{2k- \frac{1}{100}} \ud y - k \int_{\mathbb{R}_+} \bar{u}^2 |U_k|^2 \langle \eta \rangle^{2n} x^{2k-1- \frac{1}{100}} \\ \n
& + \frac{n}{2}(1-m) \int_{\mathbb{R}_+} \bar{u}^2 |U_k|^2 \langle \eta \rangle^{2n-1} \eta x^{2k-1- \frac{1}{100}} \ud y - \frac12 \int_{\mathbb{R}_+} \p_y \{ \bar{u} \bar{v} \} |U_k|^2 x^{2k- \frac{1}{100}} \langle \eta \rangle^{2n} \ud y  \\ \n
& - n \int_{\mathbb{R}_+} \bar{u} \bar{v} |U_k|^2 \langle \eta \rangle^{2n-1} \frac{1}{x^{\frac12(1-m)}} x^{2k- \frac{1}{100}} \ud y - \frac12 \int_{\mathbb{R}_+} \bar{u}_{yyyy} |Q_k|^2 \langle \eta \rangle^{2n} x^{2k- \frac{1}{100}} \ud y \\ \n
& -  n \int_{\mathbb{R}_+} \bar{u}_{yyy}|Q_k|^2 \langle \eta \rangle^{2n-1} x^{2k- \frac{1}{100}} \frac{1}{x^{\frac12(1-m)}} \ud y  + \int_{\mathbb{R}_+} 2(\bar{u}_{yy} - \p_x p_E(x)) |U_k|^2 \langle \eta \rangle^{2n} x^{2k- \frac{1}{100}} \ud y \\ \label{as:in}
& + \frac{1}{100} \int_{\mathbb{R}_+} \bar{u}^2 |U_k|^2 x^{2k - 1- \frac{1}{100}} \\
= & \bold{T}^{(Main)}_{k,n}(x) + \bold{T}^{(Comm)}_{k,n}(x).   
\end{align}
We define above 
\begin{align} \n
\bold{T}^{(Main)}_{k,n}(x) := & \frac{\p_x}{2} \int_{\mathbb{R}_+} \bar{u}^2 |U_k|^2 \langle \eta \rangle^{2n} x^{2k - \frac{1}{100}} \ud y- \int_{\mathbb{R}_+} \bar{u} \bar{u}_x |U_k|^2 \langle \eta \rangle^{2n} x^{2k - \frac{1}{100}} \ud y \\ \n
&- \frac12 \int_{\mathbb{R}_+} \p_y \{ \bar{u} \bar{v} \} |U_k|^2 x^{2k - \frac{1}{100}} \langle \eta \rangle^{2n} \ud y- \frac12 \int_{\mathbb{R}_+} \bar{u}_{yyyy} |Q_k|^2 \langle \eta \rangle^{2n} x^{2k - \frac{1}{100}} \ud y \\
&+  \int_{\mathbb{R}_+} 2(\bar{u}_{yy} - \p_x p_E(x)) |U_k|^2 \langle \eta \rangle^{2n} x^{2k - \frac{1}{100}} \ud y + \frac{1}{100} \int_{\mathbb{R}_+} \bar{u}^2 |U_k|^2 x^{2k - 1- \frac{1}{100}}, 
\end{align}
and
\begin{align} \n
\bold{T}^{(Comm)}_{k,n}(x) := &- k \int_{\mathbb{R}_+} \bar{u}^2 |U_k|^2 \langle \eta \rangle^{2n} x^{2k-1 - \frac{1}{100}}+ \frac{n}{2}(1-m) \int_{\mathbb{R}_+} \bar{u}^2 |U_k|^2 \langle \eta \rangle^{2n-1} \eta x^{2k-1 - \frac{1}{100}} \ud y \\
&  - n \int_{\mathbb{R}_+} \bar{u} \bar{v} |U_k|^2 \langle \eta \rangle^{2n-1} \frac{1}{x^{\frac12(1-m)}} x^{2k - \frac{1}{100}} \ud y -  n \int_{\mathbb{R}_+} \bar{u}_{yyy}|Q_k|^2 \langle \eta \rangle^{2n-1} x^{2k - \frac{1}{100}} \frac{1}{x^{\frac12(1-m)}} \ud y 
\end{align}
For the $\bold{T}_{Main}$ contributions, we have: 
\begin{align}\n
\bold{T}_{Main}(x) = & \frac{\p_x}{2} \mathcal{E}_{k,n}(x) + \int_{\mathbb{R}_+} \bold{c}_U |U_k|^2 \langle \eta \rangle^2x^{2k - \frac{1}{100}} \ud y + \int_{\mathbb{R}_+} \bold{c}_Q |Q_k|^2 \langle \eta \rangle^2x^{2k - \frac{1}{100}} \ud y \\  \label{TMain:dec}
& + \frac{1}{100} \int_{\mathbb{R}_+} \bar{u}^2 |U_k|^2 x^{2k- 1 - \frac{1}{100}} \ud y, 
\end{align}
where we define the coefficients 
\begin{align}
\bold{c}_U :=  &\frac32(\bar{u}_{yy} - \p_x p_E(x)) \\
\bold{c}_Q := & - \frac12 \bar{u}_{yyyy}
\end{align}

\vspace{2 mm}

\noindent \textit{Transport Commutator Bounds:} We now perform bounds on the terms arising from $\bold{T}_{k,n}^{(Comm)}(x)$. We first of all use the trivial identity $\eta = \langle \eta \rangle - 1$ as well as the decomposition \eqref{v:deco} to rewrite the commutator as follows 
\begin{align} \n
\bold{T}^{(Comm)}_{k,n}(x) = & - k \int_{\mathbb{R}_+} \bar{u}^2 |U_k|^2 \langle \eta \rangle^{2n} x^{2k-1 - \frac{1}{100}}+ \frac{n}{2}(1-m) \int_{\mathbb{R}_+} \bar{u}^2 |U_k|^2 \langle \eta \rangle^{2n}  x^{2k-1- \frac{1}{100}} \ud y \\ \n
&  - \frac{n}{2}(1-m) \int_{\mathbb{R}_+} \bar{u}^2 |U_k|^2 \langle \eta \rangle^{2n-1}  x^{2k-1- \frac{1}{100}} \ud y + nm \int_{\mathbb{R}_+} \bar{u} |U_k|^2 \langle \eta \rangle^{2n} \frac{1}{x^{(1-m)}} x^{2k- \frac{1}{100}} \ud y\\ \n
& - n \int_{\mathbb{R}_+} \bar{u} \bar{v}_{\ast} |U_k|^2 \langle \eta \rangle^{2n-1} \frac{1}{x^{\frac12(1-m)}} x^{2k- \frac{1}{100}} \ud y -  n \int_{\mathbb{R}_+} \bar{u}_{yyy}|Q_k|^2 \langle \eta \rangle^{2n-1} x^{2k- \frac{1}{100}} \frac{1}{x^{\frac12(1-m)}} \ud y \\ \label{beebop:1}
=: & \sum_{i = 1}^6 \bold{T}^{(Comm, i)}_{k,n}(x) 
\end{align}
First of all, we notice that 
\begin{align} \label{beebop:2}
\bold{T}^{(Comm, 2)}_{k,n}(x) + \bold{T}^{(Comm, 4)}_{k,n}(x) \ge 0. 
\end{align}
To estimate $\bold{T}^{(Comm, 1)}_{k,n}(x)$, we will need the following identity which follows from \eqref{QU} and \eqref{inv:form},
\begin{align} \label{id:k:kp1}
U_k = \p_x U_{k-1} + \frac{\bar{u}_x}{\bar{u}} U_{k-1} + \p_x \Big( \frac{\bar{u}_y}{\bar{u}} \Big) Q_{k-1}.
\end{align} 
We therefore have 
\begin{align*}
|\bold{T}^{(Comm, 1)}_{k,n}(x)| \le & k \int_{\mathbb{R}_+} \bar{u}^2 |\p_x U_{k-1}|^2 \langle \eta \rangle^{2n} x^{2k-1- \frac{1}{100}} +  k \int_{\mathbb{R}_+} \bar{u}^2 |\frac{\bar{u}_x}{\bar{u}} U_{k-1}|^2 \langle \eta \rangle^{2n} x^{2k-1- \frac{1}{100}} \\
&+ k  \int_{\mathbb{R}_+} \bar{u}^2 |\p_x \Big( \frac{\bar{u}_y}{\bar{u}} \Big) Q_{k-1}|^2 \langle \eta \rangle^{2n} x^{2k-1- \frac{1}{100}} =: \sum_{i =1}^3 |\bold{T}^{(Comm, 1, i)}_{k,n}(x)|
\end{align*}
We estimate successively,  
\begin{align*}
|\bold{T}^{(Comm, 1, 1)}_{k,n}(x)| \lesssim & \delta_{k > 0} \mathcal{E}_{(k-1) + \frac12}(x), \\
|\bold{T}^{(Comm, 1, 2)}_{k,n}(x)| \lesssim & \delta_{k > 0} \| U_{k-1} x^{(k-1) - \frac12 + m- \frac{1}{200}} \|_{L^2}^2 \lesssim  \delta_{k > 0} \mathcal{CK}_{k-1,0}(x) \\
|\bold{T}^{(Comm, 1, 2)}_{k,n}(x)| \lesssim & \delta_{k > 0} \| \frac{ Q_{k-1}}{y} x^{(k-1) - \frac12 + m- \frac{1}{200}} \|_{L^2}^2 \lesssim  \delta_{k > 0} \mathcal{CK}_{k-1,0}(x).
\end{align*}
For $\bold{T}^{(Comm,3)}_{k,n}(x)$, we estimate using the previous $n$: 
\begin{align*}
|\bold{T}^{(Comm,3)}_{k,n}(x)| \lesssim & \delta_{n > 0} \|  U_{k} x^{k-\frac12 + m- \frac{1}{200}} \langle \eta \rangle^n \|_{L^2}  \|  U_{k} x^{k-\frac12 + m- \frac{1}{200}} \langle \eta \rangle^{n-1} \|_{L^2} \\
\lesssim &  \delta_{n > 0} \mathcal{CK}_{k,n}(x)^{\frac12} \mathcal{CK}_{k,n-1}(x)^{\frac12},
\end{align*}
and similarly for 
\begin{align*}
|\bold{T}^{(Comm,5)}_{k,n}(x)| \lesssim & \delta_{n > 0} \| \bar{u} x^{-m} \|_{\infty} \| \bar{v}_{\ast} x^{\frac12(1-m)} \|_\infty \|  U_{k} x^{k-\frac12 + m- \frac{1}{200}} \langle \eta \rangle^n \|_{L^2} \|  U_{k} x^{k-\frac12 + m- \frac{1}{200}} \langle \eta \rangle^{n-1} \|_{L^2} \\
\lesssim &  \delta_{n > 0} \mathcal{CK}_{k,n}(x)^{\frac12} \mathcal{CK}_{k,n-1}(x)^{\frac12}.
\end{align*}
Finally, we have 
\begin{align*}
|\bold{T}^{(Comm,6)}_{k,n}(x)| \lesssim & \delta_{n > 0} \| \bar{u}_{yyy} y^2 x^{\frac12(1-m)} x^{-m} \langle \eta \rangle^{2n} \|_{L^\infty} \|  \frac{Q_{k}}{y} x^{k-\frac12 + m- \frac{1}{200}}  \|_{L^2}^2 \\
\lesssim & \delta_{n > 0}\|  U_k x^{k-\frac12 + m- \frac{1}{200}}  \|_{L^2}^2 \lesssim  \delta_{n > 0} \mathcal{CK}_{k,n-1}(x). 
\end{align*}

\vspace{2 mm}

\noindent \textit{Diffusive Terms:} We now treat the diffusive term, for which we have a long sequence of integrations by parts to produce the following 
\begin{align} \n
\bold{D}_{k,n}(x) := & - \int_{\mathbb{R}_+}  \p_y^2 u^{(k)} U_k \langle \eta \rangle^{2n} x^{2k- \frac{1}{100}} \ud y \\ \n
= & \int_{\mathbb{R}_+} \p_y u^{(k)} \p_y U_k \langle \eta \rangle^{2n} x^{2k- \frac{1}{100}} \ud y + 2n \int_{\mathbb{R}_+} \p_y u^{(k)} U_k \langle \eta \rangle^{2n-1} \frac{1}{x^{\frac12(1-m)}} x^{2k- \frac{1}{100}} \ud y \\ \label{hju:1}
& + \p_y u^{(k)} U_k x^{2k}|_{y = 0} \\ \n
= &  \int_{\mathbb{R}_+} \p_y \{\bar{u} U_k + \bar{u}_y Q_k \} \p_y U_k \langle \eta \rangle^{2n} x^{2k- \frac{1}{100}} \ud y + 2n \int_{\mathbb{R}_+} \p_y\{ \bar{u} U_k + \bar{u}_y Q_k \} U_k  \frac{\langle \eta \rangle^{2n-1}}{x^{\frac12(1-m)}} x^{2k- \frac{1}{100}} \ud y \\
& + \p_y \{ \bar{u} U_k + \bar{u}_y Q_k \} U_k x^{2k- \frac{1}{100}}|_{y = 0} \\  \n
= &  \int_{\mathbb{R}_+} \bar{u}|\p_y U_k|^2  \langle \eta \rangle^{2n} x^{2k- \frac{1}{100}} \ud y + \int_{\mathbb{R}_+} 2 \bar{u}_y U_k \p_y U_k\langle \eta \rangle^{2n} x^{2k- \frac{1}{100}} \ud y + \int_{\mathbb{R}_+} \bar{u}_{yy} Q_k \p_y U_k \langle \eta \rangle^{2n} x^{2k- \frac{1}{100}} \ud y \\ \n
& + 2n \int_{\mathbb{R}_+} \bar{u} \p_y U_k U_k \frac{\langle \eta \rangle^{2n-1}}{x^{\frac12(1-m)}} x^{2k- \frac{1}{100}} \ud y +  4n \int_{\mathbb{R}_+} \bar{u}_y |U_k|^2 \frac{\langle \eta \rangle^{2n-1}}{x^{\frac12(1-m)}} x^{2k- \frac{1}{100}} \ud y \\
&+ 2n \int_{\mathbb{R}_+} \bar{u}_{yy} Q_k U_k \frac{\langle \eta \rangle^{2n-1}}{x^{\frac12(1-m)}} x^{2k- \frac{1}{100}} \ud y + 2 \bar{u}_y |U_k|^2 x^{2k- \frac{1}{100}}|_{y = 0} \\ \n
= &  \int_{\mathbb{R}_+} \bar{u}|\p_y U_k|^2  \langle \eta \rangle^{2n} x^{2k- \frac{1}{100}} \ud y - \int_{\mathbb{R}_+}  \bar{u}_{yy} |U_k|^2 \langle \eta \rangle^{2n} x^{2k- \frac{1}{100}} \ud y + \int_{\mathbb{R}_+} \bar{u}_{yy} Q_k \p_y U_k \langle \eta \rangle^{2n} x^{2k- \frac{1}{100}} \ud y \\ \n
& + 2n \int_{\mathbb{R}_+} \bar{u} \p_y U_k U_k \frac{\langle \eta \rangle^{2n-1}}{x^{\frac12(1-m)}} x^{2k- \frac{1}{100}} \ud y +  4n \int_{\mathbb{R}_+} \bar{u}_y |U_k|^2 \frac{\langle \eta \rangle^{2n-1}}{x^{\frac12(1-m)}} x^{2k- \frac{1}{100}} \ud y \\ \n
&+ 2n \int_{\mathbb{R}_+} \bar{u}_{yy} Q_k U_k \frac{\langle \eta \rangle^{2n-1}}{x^{\frac12(1-m)}} x^{2k- \frac{1}{100}} \ud y + 2 \bar{u}_y |U_k|^2 x^{2k- \frac{1}{100}}|_{y = 0} -\bar{u}_y |U_k|^2 x^{2k}|_{y = 0} \\
& -n \int_{\mathbb{R}_+}  \bar{u}_{y} |U_k|^2  \frac{\langle \eta \rangle^{2n-1}}{x^{\frac12(1-m)}} x^{2k- \frac{1}{100}} \ud y \\ \n
= &  \int_{\mathbb{R}_+} \bar{u}|\p_y U_k|^2  \langle \eta \rangle^{2n} x^{2k- \frac{1}{100}} \ud y - \int_{\mathbb{R}_+}  \bar{u}_{yy} |U_k|^2 \langle \eta \rangle^{2n} x^{2k- \frac{1}{100}} \ud y - \int_{\mathbb{R}_+} \bar{u}_{yy} | U_k|^2 \langle \eta \rangle^{2n} x^{2k- \frac{1}{100}} \ud y \\ \n
& + 2n \int_{\mathbb{R}_+} \bar{u} \p_y U_k U_k \frac{\langle \eta \rangle^{2n-1}}{x^{\frac12(1-m)}} x^{2k- \frac{1}{100}} \ud y +  4n \int_{\mathbb{R}_+} \bar{u}_y |U_k|^2 \frac{\langle \eta \rangle^{2n-1}}{x^{\frac12(1-m)}} x^{2k- \frac{1}{100}} \ud y \\ \n
&+ 2n \int_{\mathbb{R}_+} \bar{u}_{yy} Q_k U_k \frac{\langle \eta \rangle^{2n-1}}{x^{\frac12(1-m)}} x^{2k- \frac{1}{100}} \ud y + 2 \bar{u}_y |U_k|^2 x^{2k- \frac{1}{100}}|_{y = 0} -\bar{u}_y |U_k|^2 x^{2k}|_{y = 0} \\ \n
& -n \int_{\mathbb{R}_+}  \bar{u}_{y} |U_k|^2  \frac{\langle \eta \rangle^{2n-1}}{x^{\frac12(1-m)}} x^{2k- \frac{1}{100}} \ud y  - \int_{\mathbb{R}_+} \bar{u}_{yyy} Q_k U_k \langle \eta \rangle^{2n} x^{2k- \frac{1}{100}} \ud y \\
&-2n \int_{\mathbb{R}_+} \bar{u}_{yy} Q_k U_k \frac{\langle \eta \rangle^{2n-1}}{x^{\frac12(1-m)}} x^{2k- \frac{1}{100}} \ud y \\ \n
= &  \int_{\mathbb{R}_+} \bar{u}|\p_y U_k|^2  \langle \eta \rangle^{2n} x^{2k- \frac{1}{100}} \ud y - \int_{\mathbb{R}_+}  \bar{u}_{yy} |U_k|^2 \langle \eta \rangle^{2n} x^{2k- \frac{1}{100}} \ud y - \int_{\mathbb{R}_+} \bar{u}_{yy} | U_k|^2 \langle \eta \rangle^{2n} x^{2k- \frac{1}{100}} \ud y \\ \n
& + 2n \int_{\mathbb{R}_+} \bar{u} \p_y U_k U_k \frac{\langle \eta \rangle^{2n-1}}{x^{\frac12(1-m)}} x^{2k- \frac{1}{100}} \ud y +  4n \int_{\mathbb{R}_+} \bar{u}_y |U_k|^2 \frac{\langle \eta \rangle^{2n-1}}{x^{\frac12(1-m)}} x^{2k- \frac{1}{100}} \ud y \\ \n
&+ 2n \int_{\mathbb{R}_+} \bar{u}_{yy} Q_k U_k \frac{\langle \eta \rangle^{2n-1}}{x^{\frac12(1-m)}} x^{2k- \frac{1}{100}} \ud y + 2 \bar{u}_y |U_k|^2 x^{2k- \frac{1}{100}}|_{y = 0} -\bar{u}_y |U_k|^2 x^{2k}|_{y = 0} \\ \n
& -n \int_{\mathbb{R}_+}  \bar{u}_{y} |U_k|^2  \frac{\langle \eta \rangle^{2n-1}}{x^{\frac12(1-m)}} x^{2k- \frac{1}{100}} \ud y  +\frac12 \int_{\mathbb{R}_+} \bar{u}_{yyyy} |Q_k|^2 \langle \eta \rangle^{2n} x^{2k- \frac{1}{100}} \ud y \\
&-2n \int_{\mathbb{R}_+} \bar{u}_{yy} Q_k U_k \frac{\langle \eta \rangle^{2n-1}}{x^{\frac12(1-m)}} x^{2k- \frac{1}{100}} \ud y + n \int_{\mathbb{R}_+} \bar{u}_{yyy} |Q_k|^2 \frac{\langle \eta \rangle^{2n-1}}{x^{\frac12(1-m)}} x^{2k- \frac{1}{100}} \ud y \\
= & \bold{D}^{(Main)}_{k,n}(x) +  \bold{D}^{(Comm)}_{k,n}(x),
\end{align}
where we define 
\begin{align} \n
 \bold{D}^{(Main)}_{k,n}(x) := &  \mathcal{D}_{k,n}(x) +\bar{u}_y |U_k|^2 x^{2k- \frac{1}{100}}|_{y = 0} + \int_{\mathbb{R}_+} \bold{d}_U |U_k|^2 \langle \eta \rangle^{2n} x^{2k- \frac{1}{100}} \ud y \\ \label{Dmain:dec}
 &+ \int_{\mathbb{R}_+} \bold{d}_Q |Q_k|^2 \langle \eta \rangle^{2n} x^{2k- \frac{1}{100}},
\end{align}
and
\begin{align*}
 \bold{D}^{(Comm)}_{k,n}(x) := &  2n \int_{\mathbb{R}_+} \bar{u} \p_y U_k U_k \frac{\langle \eta \rangle^{2n-1}}{x^{\frac12(1-m)}} x^{2k- \frac{1}{100}} \ud y +  3n \int_{\mathbb{R}_+} \bar{u}_y |U_k|^2 \frac{\langle \eta \rangle^{2n-1}}{x^{\frac12(1-m)}} x^{2k- \frac{1}{100}} \ud y \\
 & + n \int_{\mathbb{R}_+} \bar{u}_{yyy} |Q_k|^2 \frac{\langle \eta \rangle^{2n-1}}{x^{\frac12(1-m)}} x^{2k- \frac{1}{100}} \ud y \\
 =: & \sum_{i = 1}^3  \bold{D}^{(Comm, i)}_{k,n}(x)
\end{align*}
with coefficients 
\begin{align}
\bold{d}_U := &-2 \bar{u}_{yy}, \\
\bold{d}_Q := & \frac12 \bar{u}_{yyyy}.
\end{align}

\vspace{2 mm}

\noindent \textit{Diffusive Commutator Bounds:} We estimate successively the diffusive commutator terms appearing above as follows: 
\begin{align*}
| \bold{D}^{(Comm, 1)}_{k,n}(x)| \lesssim & \delta_{n > 0} \| \sqrt{\bar{u}} \p_y U_k x^k \langle \eta \rangle^n x^{- \frac{1}{200}} \|_{L^2} \| U_{k} \langle \eta \rangle^{n-1} x^{k + m - \frac12- \frac{1}{200}} \|_{L^2} \lesssim \delta_{n > 0} \mathcal{D}_{k,n}^{\frac12} \mathcal{CK}_{m,n-1}^{\frac12}   \\
| \bold{D}^{(Comm, 2)}_{k,n}(x)| \lesssim & \delta_{n > 0} \| U_{k} \langle \eta \rangle^{n-1} x^{k + m - \frac12- \frac{1}{200}} \|_{L^2} \| U_{k} \langle \eta \rangle^{n} x^{k + m - \frac12- \frac{1}{200}} \|_{L^2} \lesssim  \delta_{n > 0} \mathcal{CK}_{k,n-1}^{\frac12} \mathcal{CK}_{k,n}^{\frac12}\\
| \bold{D}^{(Comm, 3)}_{k,n}(x)| \lesssim & \delta_{n > 0} \| \frac{Q_{k}}{y} \langle \eta \rangle^{n-1} x^{k + m - \frac12- \frac{1}{200}} \|_{L^2} \| \frac{Q_{k}}{y} \langle \eta \rangle^{n} x^{k + m - \frac12- \frac{1}{200}} \|_{L^2}\lesssim  \delta_{n > 0} \mathcal{CK}_{k,n-1}^{\frac12} \mathcal{CK}_{k,n}^{\frac12}.
\end{align*}

\vspace{2 mm}

\noindent \textit{Complete Energy Bound:} We now bring together the various bounds. Indeed, starting with \eqref{mj:1}, we now have 
\begin{align} \label{mj:2}
\bold{T}_{k,n}^{(Main)}(x) + \bold{D}^{(Main)}_{k,n}(x) = \bold{S}_{k,n}(x) - \bold{T}_{k,n}^{(Comm)}(x) + \bold{D}^{(Comm)}_{k,n}(x).
\end{align}
Using the identities \eqref{TMain:dec} and \eqref{Dmain:dec}, we have 
\begin{align} \n
&\bold{T}_{k,n}^{(Main)}(x) + \bold{D}^{(Main)}_{k,n}(x) \\ \n
 = & \frac{\p_x}{2} \mathcal{E}_{k,n}(x) + \mathcal{D}_{k,n}(x) + \bar{u}_y |U_k|^2 x^{2k- \frac{1}{100}}|_{y = 0} \\ \n
&+ \int_{\mathbb{R}_+} (\bold{c}_U + \bold{d}_U) |U_k|^2 \langle \eta \rangle^{2m}x^{2k- \frac{1}{100}} \ud y + \int_{\mathbb{R}_+} (\bold{c}_Q + \bold{d}_Q) |Q_k|^2 \langle \eta \rangle^{2m}x^{2k- \frac{1}{100}} \ud y \\ \n
& + \frac{1}{100} \int_{\mathbb{R}_+} \bar{u}^2 U_k^2 x^{2k - 1 - \frac{1}{100}} \ud y \\ \n
\ge & \frac{\p_x}{2} \mathcal{E}_{k,n}(x) + \mathcal{D}_{k,n}(x) + \bar{u}_y |U_k|^2 x^{2k- \frac{1}{100}}|_{y = 0} + \frac{3}{2} m \int_{\mathbb{R}_+}  x^{2m-1- \frac{1}{100}} |U_k|^2 \langle \eta \rangle^{2m}x^{2k- \frac{1}{100}} \ud y \\
& + \mathcal{CK}_{k,n}(x).
\end{align}
This completes the proof of the lemma. 
\end{proof}

\begin{lemma} \label{lem:trav:1} For any $0 < \delta << 1$, the following energy inequality is valid: 
\begin{align} \n
\frac{\p_x}{2} \mathcal{E}_{k+\frac12,n}^{(Y)}(x) + \mathcal{D}^{(Y)}_{k+\frac12,n}(x) \le &  C_\delta  \mathcal{D}_{k, n+1}(x) + \delta (  \mathcal{D}_{k+1, 0}(x) + \mathcal{CK}_{k+1,0}(x)   + \mathcal{D}^{(Z)}_{k + \frac12, 0}(x)) \\ \label{cal:2}
&+C |\bold{S}^{(Y)}_{k,n}(x)|,
\end{align}
where the source term is defined as 
\begin{align} \label{def:skY}
\bold{S}^{(Y)}_{k,n}(x) :=  \int_{\mathbb{R}_+} G_k\p_x U_k \langle \eta \rangle^{2n} x^{2k+1 - \frac{1}{100}} \ud y
\end{align}
\end{lemma}
\begin{proof} We multiply equation \eqref{Pr:lin2} by $\p_x U_k \langle \eta \rangle^{2n} x^{2k+1}$. This produces the identity 
\begin{align}
\bold{T}_{k,n}(x) + \bold{D}_{k,n}(x)  = \bold{S}^{(Y)}_{k,n}(x),
\end{align}
where 
\begin{align}
\bold{T}_{k,n}(x) := & \int_{\mathbb{R}_+} (\bar{u}^2 \p_x U_k + \bar{u} \bar{v} \p_y U_k + \bar{u}_{yyy}Q_k + 2 (\bar{u}_{yy} - \p_x p_E(x))U_k)\p_x U_k \langle \eta \rangle^{2n} x^{2k+1} \ud y   \\
\bold{D}_{k,n}(x) := & - \int_{\mathbb{R}_+}\p_y^2 u^{(k)} \p_x U_k \langle \eta \rangle^{2n} x^{2k+1- \frac{1}{100}} \ud y \\
\bold{S}^{(Y)}_{k,n}(x) := &   \int_{\mathbb{R}_+} G_k\p_x U_k \langle \eta \rangle^{2n} x^{2k+1- \frac{1}{100}} \ud y
\end{align}

\vspace{2 mm}

\noindent \textit{Transport Terms:} For the transport term, we obtain 
\begin{align*}
\bold{T}_{k,n}(x) \ge \mathcal{D}^{(Y)}_{k + \frac12, n}(x)  - \sum_{i = 1}^3 \bold{T}^{(Err, i)}_{k,n}(x),
\end{align*}
where we define the transport error terms as follows
\begin{align*}
\bold{T}^{(Err, 1)}_{k,n}(x) := & \int_{\mathbb{R}_+}  \bar{u} \bar{v} \p_y U_k \p_x U_k \langle \eta \rangle^{2n} x^{2k+1- \frac{1}{100}} \ud y  \\
\bold{T}^{(Err, 2)}_{k,n}(x) := &  \int_{\mathbb{R}_+}  \bar{u}_{yyy}Q_k \p_x U_k \langle \eta \rangle^{2n} x^{2k+1- \frac{1}{100}} \ud y \\
\bold{T}^{(Err, 3)}_{k,n}(x) := &  \int_{\mathbb{R}_+}  2 (\bar{u}_{yy} - \p_x p_E(x))U_k\p_x U_k \langle \eta \rangle^{2n} x^{2k+1- \frac{1}{100}} \ud y 
\end{align*}
Of these, the most dangerous to control is $\bold{T}^{(Err, 1)}_{k,n}(x)$ because the growth of $v$ as $\eta \rightarrow \infty$ creates a loss of weight which we need to absorb into $\p_y U_k$. Indeed, we bound these terms as follows 
\begin{align} \n
|\bold{T}^{(Err, 1)}_{k,n}(x)| \lesssim & \Big\| \frac{\bar{v}}{\sqrt{\bar{u}}} \langle \eta \rangle^{-1} x^{\frac12} \Big\|_{L^\infty} \| \bar{u} \p_x U_k x^{k + \frac12- \frac{1}{200}} \langle \eta \rangle^n \|_{L^2_y} \| \sqrt{\bar{u}} \p_y U_k \langle \eta \rangle^{n+1} x^{k- \frac{1}{200}} \|_{L^2} \\
\lesssim &  \mathcal{D}^{(Y)}_{k + \frac12, n}(x)^{\frac12}  \mathcal{D}_{k, n+1}(x)^{\frac12}, \\ \n
|\bold{T}^{(Err, 2)}_{k,n}(x)| \lesssim & \| \frac{ \bar{u}_{yyy} }{\bar{u}} \langle y \rangle x^{1-m} \langle \eta \rangle^{2n} \|_{L^\infty} \| \frac{Q_k}{y} x^{k + m - \frac12- \frac{1}{200}} \|_{L^2} \| \bar{u} \p_x U_k x^{k + \frac12- \frac{1}{200}} \|_{L^2} \\
\lesssim & \| U_k x^{k + m - \frac12} \|_{L^2} \mathcal{D}_{k + \frac12, 0}^{(Y)}(x)^{\frac12} \lesssim \mathcal{CK}_{k,0}(x)^{\frac12} \mathcal{D}_{k + \frac12, 0}^{(Y)}(x)^{\frac12} \\ \n
|\bold{T}^{(Err, 3)}_{k,n}(x)| \lesssim & \| \frac{\bar{u}_{yy} - \p_x p_E(x)}{\bar{u}} x^{1-m} \langle \eta \rangle^{2n} \|_{L^\infty} \| U_k x^{k + m - \frac12- \frac{1}{200}} \|_{L^2} \| \bar{u} \p_x U_k x^{k + \frac12- \frac{1}{200}} \|_{L^2} \\
\lesssim & \mathcal{CK}_{k,0}(x)^{\frac12} \mathcal{D}^{(Y)}_{k + \frac12, 0}(x)^{\frac12}.
\end{align}

\vspace{2 mm}

\noindent \textit{Diffusive Terms:} For the contributions from $\bold{D}_{k,n}(x)$, we proceed as follows 
\begin{align*}
\bold{D}_{k,n}(x) = &-\int_{\mathbb{R}_+} \p_y^2 u^{(k)}  \p_x U_k x^{2k+1- \frac{1}{100}} \langle \eta \rangle^{2n} \ud y \\
= & -\int_{\mathbb{R}_+} \p_y^2 \{ \bar{u} U_k + \bar{u}_y Q_k   \}  \p_x U_k x^{2k+1- \frac{1}{100}} \langle \eta \rangle^{2n} \ud y \\
= & -\int_{\mathbb{R}_+} \{ \bar{u} \p_y^2 U_k + 3 \bar{u}_y \p_y U_k + 3 \bar{u}_{yy} U_k + \bar{u}_{yyy} Q_k \} \p_x U_k x^{2k+1- \frac{1}{100}} \langle \eta \rangle^{2n} \ud y \\
= & \int_{\mathbb{R}_+} \bar{u} \p_y U_k \p_{xy} U_k x^{2k+1- \frac{1}{100}} \langle \eta \rangle^{2n} \ud y - 2 \int \bar{u}_y \p_y U_k \p_x U_k x^{2k+1- \frac{1}{100}} \langle \eta \rangle^{2n} \ud y \\
& - \int_{\mathbb{R}_+} 3 \bar{u}_{yy} U_k \p_x U_k x^{2k+1- \frac{1}{100}} \langle \eta \rangle^{2n} \ud y  - \int_{\mathbb{R}_+}  \bar{u}_{yyy} Q_k \p_x U_k x^{2k+1- \frac{1}{100}} \langle \eta \rangle^{2n} \ud y \\
& +2n \int_{\mathbb{R}_+} \bar{u} \p_y U_k \p_x U_k x^{2k+1- \frac{1}{100}} \frac{\langle \eta \rangle^{2n-1}}{x^{\frac12(1-m)}} \\
= & \frac{\p_x}{2} \int_{\mathbb{R}_+} \bar{u} |\p_y U_k|^2 x^{2k+1- \frac{1}{100}} \langle \eta \rangle^{2n} \ud y - 2 \int \bar{u}_y \p_y U_k \p_x U_k x^{2k+1- \frac{1}{100}} \langle \eta \rangle^{2n} \ud y \\
& - \int_{\mathbb{R}_+} 3 \bar{u}_{yy} U_k \p_x U_k x^{2k+1- \frac{1}{100}} \langle \eta \rangle^{2n} \ud y  - \int_{\mathbb{R}_+}  \bar{u}_{yyy} Q_k \p_x U_k x^{2k+1- \frac{1}{100}} \langle \eta \rangle^{2n} \ud y \\
& +2n \int_{\mathbb{R}_+} \bar{u} \p_y U_k \p_x U_k x^{2k+1- \frac{1}{100}} \frac{\langle \eta \rangle^{2n-1}}{x^{\frac12(1-m)}} \ud y -  \int_{\mathbb{R}_+} \frac{ \bar{u}_x}{2} |\p_y U_k|^2 x^{2k+1- \frac{1}{100}} \langle \eta \rangle^{2n} \ud y \\
& -  \frac{2k+1}{2} \int_{\mathbb{R}_+} \bar{u} |\p_y U_k|^2 x^{2k- \frac{1}{100}} \langle \eta \rangle^{2n} \ud y + \frac{n(1-m)}{2} \int_{\mathbb{R}_+} \bar{u} |\p_y U_k|^2 x^{2k- \frac{1}{100}} \langle \eta \rangle^{2n-1} \eta \ud y \\
= &\frac{\p_x}{2} \mathcal{E}^{(Y)}_{k,n}(x) + \sum_{i = 1}^7 \bold{D}^{(i)}_{k,n}(x).
\end{align*} 
We now estimate these error terms. First, we have 
\begin{align*}
|\bold{D}^{(1)}_{k,n}(x)| \lesssim & \| \bar{u}_y x^{\frac12(1-m)} x^{-m} \langle \eta \rangle^{2n} \|_{L^\infty} \| \p_y U_k x^{\frac{m}{2} + k- \frac{1}{200}} \|_{L^2} \| \p_x U_k x^{k + \frac12 + m- \frac{1}{200}} \|_{L^2} \\
\le & C (\delta_1 \mathcal{D}_{k+\frac12,0}^{(Z)}(x)^{\frac12} + C_{\delta_1} \mathcal{D}_{k, 0}(x)^{\frac12} ) ( \delta_2 \mathcal{D}_{k + 1, 0}(x)^{\frac12} + C_{\delta_2} \mathcal{D}^{(Y)}_{k + \frac12, 0}(x)^{\frac12}  ) \\
\le & C_\delta \mathcal{D}_{k, 0}(x) + \delta  \mathcal{D}_{k+\frac12,0}^{(Z)}(x) + \delta \mathcal{D}_{k + 1, 0}(x) + \delta \mathcal{D}^{(Y)}_{k + \frac12, 0}(x),
\end{align*}
where we have invoked the interpolation bounds \eqref{train:1} and \eqref{train:2}. 

For estimate the error term $\bold{D}^{(2)}_{k,n}(x)$, we will need to collect the following preliminary estimate. First, we note that 
\begin{align} \label{id:k:kp2}
U_{k+1} = \p_x U_{k} + \frac{\bar{u}_x}{\bar{u}} U_{k} + \p_x \Big( \frac{\bar{u}_y}{\bar{u}} \Big) Q_{k}.
\end{align} 
Therefore, 
\begin{align*}
\| \p_x U_k x^{k+\frac12 + m} \|_{L^2} \lesssim  &\| U_{k+1} x^{k+\frac12 + m- \frac{1}{200}} \|_{L^2} + \| \frac{\bar{u}_x}{\bar{u}} U_{k} x^{k+\frac12 + m- \frac{1}{200}} \|_{L^2}+  \|  \p_x \Big( \frac{\bar{u}_y}{\bar{u}} \Big) Q_{k} x^{k+\frac12 + m- \frac{1}{200}} \|_{L^2} \\
\lesssim &\| U_{k+1} x^{k+\frac12 + m- \frac{1}{200}} \|_{L^2} + \| \frac{\bar{u}_x}{\bar{u}} x \|_{L^\infty} \|U_{k} x^{k-\frac12 + m- \frac{1}{200}} \|_{L^2} \\
& + \| \p_x ( \frac{\bar{u}_y}{\bar{u}} ) x y \|_{L^\infty} \|\frac{Q_{k}}{y} x^{k-\frac12 + m- \frac{1}{200}} \|_{L^2}   \\
\lesssim & \mathcal{CK}_{k+1, 0}(x)^{\frac12} + \mathcal{CK}_{k,0}(x)^{\frac12}
\end{align*}
Using this bound, we have 
\begin{align*}
|\bold{D}^{(2)}_{k,n}(x)| \lesssim & \| \bar{u}_{yy} \langle \eta \rangle^{2n}  x^{-m} x^{(1-m)} \|_{L^\infty} \| U_k x^{k-\frac12 + m- \frac{1}{200}} \|_{L^2}  \| \p_x U_k x^{k+\frac12 + m- \frac{1}{200}} \|_{L^2} \\
\lesssim & \mathcal{CK}_{k,0}(x)^{\frac12}(\mathcal{CK}_{k+1, 0}(x)^{\frac12} + \mathcal{CK}_{k,0}(x)^{\frac12}),
\end{align*}
and very similarly,  
\begin{align*}
|\bold{D}^{(3)}_{k,n}(x)| \lesssim  & \| y \bar{u}_{yyy} \langle \eta \rangle^{2n}  x^{-m} x^{(1-m)} \|_{L^\infty} \| U_k x^{k-\frac12 + m- \frac{1}{200}} \|_{L^2}  \| \p_x U_k x^{k+\frac12 + m- \frac{1}{200}} \|_{L^2} \\
\lesssim & \mathcal{CK}_{k,0}(x)^{\frac12}(\mathcal{CK}_{k+1, 0}(x)^{\frac12} + \mathcal{CK}_{k,0}(x)^{\frac12}).
\end{align*}
To estimate $\bold{D}^{(4)}_{k,n}(x)$, it is convenient to split into the case when $\eta > 1$ and $\eta \le 1$ as follows 
\begin{align*}
\bold{D}^{(4)}_{k,n}(x) = & 2n \int_{\mathbb{R}_+} \bar{u} \p_y U_k \p_x U_k x^{2k+1- \frac{1}{100}} \frac{\langle \eta \rangle^{2n-1}}{x^{\frac12(1-m)}} \chi(\eta > 1) \ud y \\
& + 2n \int_{\mathbb{R}_+} \bar{u} \p_y U_k \p_x U_k x^{2k+1- \frac{1}{100}} \frac{\langle \eta \rangle^{2n-1}}{x^{\frac12(1-m)}} \chi(\eta \le 1) \ud y := \bold{D}^{(4, >)}_{k,n}(x) + \bold{D}^{(4, \le)}_{k,n}(x).
\end{align*}
For the far-field term, we estimate as follows
\begin{align*}
| \bold{D}^{(4, >)}_{k,n}(x)| \lesssim  &\| \frac{x^{\frac{m}{2}}}{\sqrt{\bar{u}}} \chi_{\eta > 1} \|_{L^\infty} \| \sqrt{\bar{u}} \p_y U_x x^{k- \frac{1}{200}} \langle \eta \rangle^n \|_{L^2} \| \bar{u} \p_x U_k x^{k + \frac12- \frac{1}{200}} \langle \eta \rangle^n \|_{L^2} \\
\lesssim &  \mathcal{D}_{k,n}(x)^{\frac12} \mathcal{D}^{(Y)}_{k+\frac12,n}(x)^{\frac12}.
\end{align*}
For the near-field term, we notice that weights of $\langle \eta \rangle$ are bounded by a constant. Therefore, we can estimate 
\begin{align*}
| \bold{D}^{(4, \le)}_{k,n}(x)| \lesssim  & \| x^{\frac{m}{2}} \p_y U_k x^{k- \frac{1}{200}} \|_{L^2} \| x^m \p_x U_k x^{k + \frac12- \frac{1}{200}} \|_{L^2} \\
\le & C (\delta_1 \mathcal{D}_{k+\frac12,0}^{(Z)}(x)^{\frac12} + C_{\delta_1} \mathcal{D}_{k, 0}(x)^{\frac12} ) ( \delta_2 \mathcal{D}_{k + 1, 0}(x)^{\frac12} + C_{\delta_2} \mathcal{D}^{(Y)}_{k + \frac12, 0}(x)^{\frac12}  ).
\end{align*}
The terms $\bold{D}^{(5)}_{k,n}(x), \bold{D}^{(6)}_{k,n}(x)$ and $\bold{D}^{(7)}_{k,n}(x)$ are easily seen to be bounded by the previous dissipation term: 
\begin{align*}
|\bold{D}^{(5)}_{k,n}(x)| \lesssim &  \| \frac{\bar{u}_x}{\bar{u}} x \|_{L^\infty} \| \sqrt{\bar{u}} \p_y U_k x^{k- \frac{1}{200}} \langle \eta \rangle^n \|_{L^2}^2 \lesssim \mathcal{D}_{k,n}(x), \\
|\bold{D}^{(6)}_{k,n}(x)| \lesssim &   \| \sqrt{\bar{u}} \p_y U_k x^{k- \frac{1}{200}} \langle \eta \rangle^n \|_{L^2}^2 \lesssim \mathcal{D}_{k,n}(x), \\
|\bold{D}^{(7)}_{k,n}(x)| \lesssim &\| \sqrt{\bar{u}} \p_y U_k x^{k- \frac{1}{200}} \langle \eta \rangle^n \|_{L^2}^2 \lesssim \mathcal{D}_{k,n}(x).
\end{align*}
This completes the proof of the lemma. 
\end{proof}

\begin{lemma} \label{lem:Tioga:1}For any $0 < \delta << 1$, the following energy inequality is valid:
\begin{align} \n
\frac{\p_x}{2} \mathcal{E}_{k+\frac12,n}^{(Z)}(x)+  \mathcal{B}^{(Z)}_{k+\frac12}(x)  + \mathcal{D}^{(Z)}_{k+\frac12,n}(x) \le& C_{\delta} ( \mathcal{D}_{k,n+1}(x) +  \mathcal{CK}_{k,n}(x) + \mathcal{B}_k(x) )+  \delta \mathcal{D}^{(Y)}_{k + \frac12,0}(x) \\ \label{cal:3}
& + C |\bold{S}^{(Z)}_{k,n}(x)|,
\end{align}
where the source term 
\begin{align} \label{szA}
\bold{S}^{(Z)}_{k,n}(x) := &  \int_{\mathbb{R}_+} \p_y G_k \p_y U_k \langle \eta \rangle^{2n} x^{2k + 1 - m - \frac{1}{100}} \ud y.
\end{align}
\end{lemma}
\begin{proof} We multiply equation \eqref{Pr:lin2} by the quantity $\p_y U_k \langle \eta \rangle^{2n} x^{2k + 1 - m- \frac{1}{100}}$ to the vorticity formulation of the equation. This results in the identity 
\begin{align*}
\widehat{\bold{T}}_{k,n}(x) + \widehat{\bold{D}}_{k,n}(x) = \bold{S}^{(Z)}_{k,n}(x),
\end{align*}
where we define the terms 
\begin{align}
\widehat{\bold{T}}_{k,n}(x) := &\int_{\mathbb{R}_+} \p_y (\bar{u}^2 \p_x U_k + \bar{u} \bar{v} \p_y U_k + \bar{u}_{yyy}Q_k + 2 (\bar{u}_{yy} - \p_x p_E(x))U_k) \p_y U_k \langle \eta \rangle^{2n} x^{2k + 1 - m- \frac{1}{100}}\ud y,  \\
\widehat{\bold{D}}_{k,n}(x) := & - \int_{\mathbb{R}_+}\p_y^3 u^{(k)} \p_y U_k \langle \eta \rangle^{2n} x^{2k + 1 - m- \frac{1}{100}} \ud y, \\
\bold{S}^{(Z)}_{k,n}(x) := &  \int_{\mathbb{R}_+} \p_y G_k \p_y U_k \langle \eta \rangle^{2n} x^{2k + 1 - m- \frac{1}{100}} \ud y.
\end{align}

\vspace{2 mm}

\noindent \textit{Transport Terms:} We distribute the $\p_y$ to obtain 
\begin{align*}
\widehat{\bold{T}}_{k,n}(x) =& \int_{\mathbb{R}_+} \bar{u}^2 \p_{xy} U_k \p_y U_k \langle \eta \rangle^{2n} x^{2k+1-m- \frac{1}{100}} \ud y + \int_{\mathbb{R}_+} 2 \bar{u} \bar{u}_y \p_x U_k \p_y U_k \langle \eta \rangle^{2n} x^{2k+1-m- \frac{1}{100}} \ud y \\
& + \int_{\mathbb{R}_+} \p_y (\bar{u} \bar{v})| \p_y U_k|^2 \langle \eta \rangle^{2n} x^{2k+1 - m- \frac{1}{100}} \ud y + \int_{\mathbb{R}_+} \bar{u} \bar{v} \p_y^2 U_k \p_y U_k \langle \eta \rangle^{2n} x^{2k+1-m- \frac{1}{100}} \ud y \\
& + \int_{\mathbb{R}_+} \bar{u}_{yyyy} Q_k \p_y U_k \langle \eta \rangle^{2n} x^{2k+1-m- \frac{1}{100}} \ud y + \int_{\mathbb{R}_+} 3 \bar{u}_{yyy} U_k \p_y U_k \langle \eta \rangle^{2n} x^{2k+1-m- \frac{1}{100}} \ud y \\
& + \int_{\mathbb{R}_+} 2 (\bar{u}_{yy} - \p_x p_E) |\p_y U_k|^2 \langle \eta \rangle^{2n} x^{2k+1-m- \frac{1}{100}} \ud y \\
= & \frac{\p_x}{2} \int_{\mathbb{R}_+} \bar{u}^2 |\p_y U_k|^2 \langle \eta \rangle^{2n} x^{2k+1-m- \frac{1}{100}} \ud y + \int_{\mathbb{R}_+} 3 \bar{v} \bar{u}_y | \p_y U_k|^2 \langle \eta \rangle^{2n} x^{2k+1 - m- \frac{1}{100}} \ud y \\
& - \frac{2k+1-m- \frac{1}{100}}{2} \int_{\mathbb{R}_+} \bar{u}^2 |\p_y U_k|^2 \langle \eta \rangle^{2n} x^{2k-m- \frac{1}{100}} \ud y \\
&+ \frac{n(1-m)}{2} \int_{\mathbb{R}_+} \bar{u}^2 |\p_y U_k|^2 x^{2k-m- \frac{1}{100}} \langle \eta \rangle^{2n-1}\eta \ud y \\
&+ \int_{\mathbb{R}_+} 2 \bar{u} \bar{u}_y \p_x U_k \p_y U_k \langle \eta \rangle^{2n} x^{2k+1-m- \frac{1}{100}} \ud y +  \int_{\mathbb{R}_+} \bar{u} \bar{v} \p_y^2 U_k \p_y U_k \langle \eta \rangle^{2n} x^{2k+1-m- \frac{1}{100}} \ud y \\
& + \int_{\mathbb{R}_+} \bar{u}_{yyyy} Q_k \p_y U_k \langle \eta \rangle^{2n} x^{2k+1-m- \frac{1}{100}} \ud y + \int_{\mathbb{R}_+} 3 \bar{u}_{yyy} U_k \p_y U_k \langle \eta \rangle^{2n} x^{2k+1-m- \frac{1}{100}} \ud y \\
= & \frac{\p_x}{2} \mathcal{E}^{(Z)}_{k + \frac12,n}(x) + \sum_{i = 1}^7 \widehat{\bold{T}}^{(i)}_{k,n}(x). 
\end{align*}
We estimate the transport error terms appearing above in succession: 
\begin{align*}
|\widehat{\bold{T}}^{(1)}_{k,n}(x)| \lesssim & \| \frac{\bar{v}}{\langle \eta \rangle} x^{\frac12(1-m)} \|_{L^\infty} \| \bar{u}_y x^{-m} x^{\frac12(1-m)} \langle \eta \rangle^{2n+1} \|_{L^\infty} \| \p_y U_k x^{k + \frac{m}{2}- \frac{1}{200}} \|_{L^2}^2 \\
\le & \delta \mathcal{D}^{(Z)}_{k + \frac12, 0}(x) + C_\delta \mathcal{D}_{k, 0}(x), \\
|\widehat{\bold{T}}^{(2)}_{k,n}(x)| + |\widehat{\bold{T}}^{(3)}_{k,n}(x)| \lesssim & \| \frac{\bar{u}}{x^m} \|_{L^\infty} \| \sqrt{\bar{u}} \p_y U_k x^{k- \frac{1}{200}} \langle \eta \rangle^{n} \|_{L^2}^2 \lesssim \mathcal{D}_{k, n}(x) \\
|\widehat{\bold{T}}^{(4)}_{k,n}(x)| \lesssim & \| \bar{u}_y \langle \eta \rangle^{2n} x^{\frac12(1-m) - m} \|_{L^\infty} \| \bar{u} \p_x U_k x^{k  - \frac{1}{200}} \|_{L^2} \| \p_y U_k x^{k + \frac{m}{2}- \frac{1}{200}} \|_{L^2} \\
\lesssim & \mathcal{D}^{(Y)}_{k + \frac12,0}(x)^{\frac12}(C_\delta \mathcal{D}_{k, n}(x)^{\frac12} + \delta \mathcal{D}^{(Z)}_{k+\frac12, n}(x)^{\frac12} ), \\
|\widehat{\bold{T}}^{(5)}_{k,n}(x)| \lesssim & \| \frac{\bar{v}}{\langle \eta \rangle} x^{\frac12(1-m)} \|_{L^\infty} \| \sqrt{\bar{u}} \p_y U_k x^{k- \frac{1}{200}} \langle \eta \rangle^{n+1} \|_{L^2} \| \sqrt{\bar{u}} \p_y^2 U_k \langle \eta \rangle^n x^{k + \frac12 - \frac{m}{2}- \frac{1}{200} } \|_{L^2} \\
\lesssim & \mathcal{D}_{k, n+1}(x)^{\frac12} \mathcal{D}^{(Z)}_{k,n}(x)^{\frac12}, \\
|\widehat{\bold{T}}^{(6)}_{k,n}(x)| + |\widehat{\bold{T}}^{(7)}_{k,n}(x)| \lesssim & ( \| y \bar{u}_{yyyy}  \langle \eta \rangle^{2n}  x^{\frac{3}{2}(1-m) - m} \|_{L^\infty} + \| \bar{u}_{yyy} \langle \eta \rangle^{2n} x^{\frac{3}{2}(1-m) - m}\|_{L^\infty}) \\
& \times \| U_k x^{k - \frac12 + m- \frac{1}{200}} \|_{L^2} \| \p_y U_k x^{k + \frac{m}{2}- \frac{1}{200}} \|_{L^2} \\
\lesssim & \mathcal{CK}_{k,0}(x)^{\frac12}(C_\delta \mathcal{D}_{k, n}(x)^{\frac12} + \delta \mathcal{D}^{(Z)}_{k+\frac12, n}(x)^{\frac12} ).
\end{align*}

\vspace{2 mm}

\noindent \textit{Diffusive Terms:} For the diffusive terms, we have 
\begin{align*}
\widehat{\bold{D}}_{k,n}(x) = & \int_{\mathbb{R}_+} \p_y^2 u^{(k)} \p_y^2 U_k x^{2k + 1 -m- \frac{1}{100}} \langle \eta \rangle^{2n} \ud y + 2n  \int_{\mathbb{R}_+} \p_y^2 u^{(k)} \p_y U_k x^{2k + 1 -m- \frac{1}{100}} \frac{\langle \eta \rangle^{2n-1}}{x^{\frac12(1-m)}} \ud y \\
& + \p_y^2 u^{(k)} \p_y U_k x^{2k+1-m- \frac{1}{100}}|_{y = 0} \\ 
= &  \int_{\mathbb{R}_+} \p_y^2\{ \bar{u} U_k + \bar{u}_y Q_k \} \p_y^2 U_k x^{2k + 1 -m- \frac{1}{100}} \langle \eta \rangle^{2n} \ud y \\
&+ 2n  \int_{\mathbb{R}_+} \p_y^2 \{ \bar{u} U_k + \bar{u}_y Q_k \}\p_y U_k x^{2k + 1 -m- \frac{1}{100}} \frac{\langle \eta \rangle^{2n-1}}{x^{\frac12(1-m)}} \ud y \\
& + \p_y^2 \{ \bar{u} U_k + \bar{u}_y Q_k \} \p_y U_k x^{2k+1-m- \frac{1}{100}}|_{y = 0} \\
= & \int_{\mathbb{R}_+} \bar{u} |\p_y^2 U_k|^2 x^{2k + 1 -m- \frac{1}{100}} \langle \eta \rangle^{2n} \ud y + \int_{\mathbb{R}_+} 3 \bar{u}_y \p_y U_k \p_y^2 U_k x^{2k+1-m- \frac{1}{100}} \langle \eta \rangle^{2n} \ud y \\
& +  \int_{\mathbb{R}_+} 3 \bar{u}_{yy} U_k \p_y^2 U_k x^{2k+1-m- \frac{1}{100}} \langle \eta \rangle^{2n} \ud y + \int_{\mathbb{R}_+}  \bar{u}_{yyy} Q_k  \p_y^2 U_k x^{2k + 1 -m- \frac{1}{100}} \langle \eta \rangle^{2n} \ud y \\
& + 2n \int_{\mathbb{R}_+} \bar{u} \p_y^2 U_k \p_y U_k x^{2k+1 - m- \frac{1}{100}} \frac{\langle \eta \rangle^{2n-1}}{x^{\frac12(1-m)}} \ud y +  6n \int_{\mathbb{R}_+} \bar{u}_y | \p_y U_k|^2 x^{2k+1 - m- \frac{1}{100}} \frac{\langle \eta \rangle^{2n-1}}{x^{\frac12(1-m)}} \ud y \\
& + 6n \int_{\mathbb{R}_+} \bar{u}_{yy} U_k \p_y U_k x^{2k+1-m- \frac{1}{100}} \frac{\langle \eta \rangle^{2n-1}}{x^{\frac12(1-m)}} \ud y + 2n \int_{\mathbb{R}_+} \bar{u}_{yyy} Q_k \p_y U_k x^{2k+1-m- \frac{1}{100}} \frac{\langle \eta \rangle^{2n-1}}{x^{\frac12(1-m)}} \ud y \\
& + 3 \bar{u}_y |\p_y U_k|^2 x^{2k+1 - m- \frac{1}{100}}|_{y = 0} + \bar{u}_{yy} U_k \p_y U_k x^{2k+1-m- \frac{1}{100}}|_{y = 0} \\
= & \int_{\mathbb{R}_+} \bar{u} |\p_y^2 U_k|^2 x^{2k + 1 -m- \frac{1}{100}} \langle \eta \rangle^{2n} \ud y + \frac{3}{2} \bar{u}_y |\p_y U_k|^2 x^{2k+1 - m- \frac{1}{100}}|_{y = 0}  \\
& - \frac92 \int_{\mathbb{R}_+}  \bar{u}_{yy} |\p_y U_k|^2  x^{2k+1-m- \frac{1}{100}} \langle \eta \rangle^{2n} \ud y    - \int_{\mathbb{R}_+}  \bar{u}_{yyyy} Q_k  \p_y U_k x^{2k + 1 -m- \frac{1}{100}} \langle \eta \rangle^{2n} \ud y \\
&-4 \int_{\mathbb{R}_+}  \bar{u}_{yyy} U_k  \p_y U_k x^{2k + 1 -m- \frac{1}{100}} \langle \eta \rangle^{2n} \ud y  -2n \int_{\mathbb{R}_+}  \bar{u}_{yyy} Q_k  \p_y U_k x^{2k + \frac12 -\frac{m}{2}- \frac{1}{100}} \langle \eta \rangle^{2n-1} \ud y \\
&+ 2n \int_{\mathbb{R}_+} \bar{u} \p_y^2 U_k \p_y U_k x^{2k+1 - m- \frac{1}{100}} \frac{\langle \eta \rangle^{2n-1}}{x^{\frac12(1-m)}} \ud y +  6n \int_{\mathbb{R}_+} \bar{u}_y | \p_y U_k|^2 x^{2k+1 - m- \frac{1}{100}} \frac{\langle \eta \rangle^{2n-1}}{x^{\frac12(1-m)}} \ud y  \\
&+ 2n \int_{\mathbb{R}_+} \bar{u}_{yyy} Q_k \p_y U_k x^{2k+1-m- \frac{1}{100}} \frac{\langle \eta \rangle^{2n-1}}{x^{\frac12(1-m)}} \ud y  + 6n \int_{\mathbb{R}_+} \bar{u}_{yy} U_k \p_y U_k x^{2k+1-m- \frac{1}{100}} \frac{\langle \eta \rangle^{2n-1}}{x^{\frac12(1-m)}} \ud y  \\
&-2 \bar{u}_{yy} U_k \p_y U_k x^{2k+1-m- \frac{1}{100}}|_{y = 0} \\
= & \mathcal{D}^{(Z)}_{k + \frac12, n}(x) + \frac{3}{2} \bar{u}_y |\p_y U_k|^2 x^{2k+1 - m- \frac{1}{100}}|_{y = 0} + \sum_{i = 1}^{8} \widehat{\bold{D}}^{(i)}_{k,n}(x)  +  \widehat{\bold{B}}_{k,n}(x).
\end{align*}
Each of these terms will be estimated in succession. First, the bulk terms: 
\begin{align*}
|\widehat{\bold{D}}^{(1)}_{k,n}(x)| \lesssim & \| \bar{u}_{yy} x^{1 - 2m} \langle \eta \rangle^{2n}\|_{L^\infty} \| \p_y U_k x^{k + \frac{m}{2} - \frac{1}{200}} \|_{L^2_y}^2 \le \lambda \mathcal{D}^{(Z)}_{k + \frac12, 0} + C_\lambda \mathcal{D}_{k, 0}(x) \\
\sum_{i = 2}^3|\widehat{\bold{D}}^{(i)}_{k,n}(x)| \lesssim &(\| \bar{u}_{yyy} x^{\frac32 - \frac{5m}{2}} \langle \eta \rangle^{2n}\|_{L^\infty} + \| y\bar{u}_{yyyy} x^{\frac32 - \frac{5m}{2}} \langle \eta \rangle^{2n}\|_{L^\infty}) \\ 
& \times \| U_k x^{k + m - \frac12- \frac{1}{200}} \|_{L^2_y} \| \p_y U_k x^{k + \frac{m}{2}- \frac{1}{200}} \|_{L^2_y} \\
\le & C_\lambda (\mathcal{CK}_{k,0}(x) + \mathcal{D}_{k,0}(x)) + \lambda \mathcal{D}^{(Z)}_{k,0}(x), \\
|\widehat{\bold{D}}^{(4)}_{k,n}(x)| + |\widehat{\bold{D}}^{(7)}_{k,n}(x)| \lesssim & \| y \bar{u}_{yyy} x^{1 - 2m} \langle \eta \rangle^{2n-1} \|_{L^\infty} \| U_k x^{k - \frac12 + m- \frac{1}{200}} \|_{L^2_y} \| \p_y U_k x^{k + \frac{m}{2}- \frac{1}{200}} \|_{L^2_y} \\
\le & C_\lambda (\mathcal{CK}_{k,0}(x) + \mathcal{D}_{k,0}(x)) + \lambda \mathcal{D}^{(Z)}_{k,0}(x), \\
|\widehat{\bold{D}}^{(5)}_{k,n}(x)| \lesssim & \| \sqrt{\bar{u}} \p_y^2 U_k x^{k + \frac12 - \frac{m}{2}- \frac{1}{200}} \langle \eta \rangle^n \|_{L^2_y} \| \sqrt{\bar{u}} \p_y U_k x^{k- \frac{1}{200}} \langle \eta \rangle^{n-1} \|_{L^2_y} \\
\lesssim & \mathcal{D}^{(Z)}_{k + \frac12, n}(x)^{\frac12} \mathcal{D}_{k,n-1}(x)^{\frac12}, \\
|\widehat{\bold{D}}^{(6)}_{k,n}(x)| \lesssim & \| \bar{u}_y x^{\frac12 - \frac{3m}{2}} \langle \eta \rangle^{2n-1} \|_{L^\infty} \| \p_y U_k x^{k + \frac{m}{2}- \frac{1}{200}} \|_{L^2_y}^2 \le C_\lambda \mathcal{D}_{k, 0}(x) + \lambda \mathcal{D}_{k + \frac12,0}^{(Z)}(x) \\
|\widehat{\bold{D}}^{(8)}_{k,n}(x)| \lesssim & \| \bar{u}_{yy} x^{1 - 2m} \langle \eta \rangle^{2n-1} \|_{L^\infty} \| U_k x^{k + m - \frac12- \frac{1}{200}} \|_{L^2_y} \| \p_y U_k x^{k + \frac{m}{2}- \frac{1}{200}} \|_{L^2_y} \\
\le &  C_\lambda (\mathcal{CK}_{k,0}(x) + \mathcal{D}_{k,0}(x)) + \lambda \mathcal{D}^{(Z)}_{k,0}(x).
\end{align*} 
Next, the boundary term: 
\begin{align*}
 |\widehat{\bold{B}}_{k,n}(x)| \lesssim & \| \bar{u}_{yy} x^{1 - 2m} \|_{L^\infty}  (x^{-\frac14 + \frac{3m}{4}} |U_k| x^{k  - \frac{1}{200}})(x^{\frac14 + \frac{m}{4}} |\p_y U_k| x^{k- \frac{1}{200}}) \\
 \lesssim & ( \mathcal{B}_k(x))^{\frac12} (\mathcal{B}_{k + \frac12}^{(Z)}(x))^{\frac12}.
\end{align*}
This concludes the proof of the lemma. 
\end{proof}

\section{Quasilinear Energy Estimates} \label{sec:5}

\begin{lemma} Assume the bootstrap assumption \eqref{boots:1}. Then the following quasilinear energy bound is valid: 
\begin{align} \n
&\frac{\p_x}{2} \overline{\mathcal{E}}_5(x) + \overline{\mathcal{CK}}_5(x) + \overline{\mathcal{B}}_5(x) + \overline{\mathcal{D}}_5(x) \\ \label{cal:4}
\lesssim & C_\lambda \mathcal{I}_{4 + \frac12, 0}(x) + \lambda \mathcal{D}_{5,0} +  \eps^{\frac14} \mathcal{I}_5(x) + C| \bold{S}_{k} |.
\end{align}
\end{lemma}
\begin{proof} We multiply equation \eqref{ql:sys} by $\mathcal{U}_{k} x^{2k - \frac{1}{100}}$ which produces the identity 
\begin{align}
\bold{T}_{k} + \bold{D}_{k} = \bold{S}_k,
\end{align}
where we define these contributions as follows: 
\begin{align}
\bold{T}_{k} = & \int_{\mathbb{R}_+} ( \mu^2 \p_x \mathcal{U}_k + \bold{a} \mathcal{U}_k + \bold{b} \mathcal{Q}_k + \bold{c} \p_y \mathcal{U}_k ) \mathcal{U}_k x^{2k - \frac{1}{100}} \ud y \\
\bold{D}_{k} = & - \int_{\mathbb{R}_+} \p_y^2 u^{(k)} \mathcal{U}_k x^{2k - \frac{1}{100}} \ud y \\
\bold{S}_{k} = & \int_{\mathbb{R}_+} H_k \mathcal{U}_k x^{2k - \frac{1}{100}} \ud y, 
\end{align}
where we define $H_k$ in \eqref{ql:sys}. 

\vspace{2 mm}

\noindent \textit{Transport Terms:} We have 
\begin{align*}
\bold{T}_{k} = & \frac{\p_x}{2} \int_{\mathbb{R}_+} \mu^2 \mathcal{U}_k^2 x^{2k - \frac{1}{100}} \ud y + \int_{\mathbb{R}_+}(\bold{a} - \mu \mu_x) |\mathcal{U}_k|^2 x^{2k - \frac{1}{100}} \ud y - k \int_{\mathbb{R}_+} \mu^2 \mathcal{U}_k^2 x^{2k-1 - \frac{1}{100}} \ud y \\
& +\int_{\mathbb{R}_+}  \bold{b} \mathcal{Q}_k \mathcal{U}_k x^{2k - \frac{1}{100}}\ud y + \int_{\mathbb{R}_+} \bold{c} \p_y \mathcal{U}_k \mathcal{U}_k x^{2k - \frac{1}{100}} \ud y \\
= & \frac{\p_x}{2} \int_{\mathbb{R}_+} \mu^2 \mathcal{U}_k^2 x^{2k - \frac{1}{100}} \ud y + \sum_{i = 1}^4 \bold{T}^{(Err, i)}_{k}.
\end{align*}
We estimate these terms individually. First, we have 
\begin{align*}
|\bold{T}^{(Err, 1)}_{k}| \lesssim & \| (\bar{u}_x \mu + 2 \bar{v} \mu_y) x^{1-2m} \|_{L^\infty} \| x^{k + m-\frac12 - \frac{1}{200}} \mathcal{U}_k \|_{L^2_y}^2 \\
\lesssim &\Big( \| \bar{u}_x x^{1-m} \|_{L^\infty} (\| \bar{u} x^{-m} \|_{L^\infty} + \eps \| u x^{-m} \|_{L^\infty}) + \| \frac{\bar{v}}{\langle \eta \rangle} x^{\frac12(1-m)} \|_{L^\infty}(\| \bar{u}_y x^{\frac12 - \frac{3m}{2}} \langle \eta \rangle \|_{L^\infty} \\
&+ \eps \| u_y x^{\frac12 - \frac{3m}{2}}  \langle \eta \rangle\|_{L^\infty} ) \Big) \| x^{k + m-\frac12 - \frac{1}{200}} \mathcal{U}_k \|_{L^2_y}^2 \\
\lesssim & \overline{\mathcal{CK}}_{k,0}(x) \\
\le & C_\lambda \mathcal{I}_{4 + \frac12, 0}(x) + \lambda \mathcal{D}_{5,0} + \eps^{\frac14} \mathcal{I}_5(x),
\end{align*}
where we have used the bound \eqref{jurassic:1} and \eqref{ohm:1}. Second, we have 
\begin{align*}
|\bold{T}^{(Err, 2)}_{k}| \lesssim & \| \mu x^{-m} \|_{L^\infty}^2 \| x^{m+k-\frac12 - \frac{1}{200}} \mathcal{U}_k \|_{L^2_y}^2 \\
\lesssim & (\| \bar{u} x^{-m} \|_{L^\infty}^2 + \eps^2 \| u x^{-m} \|_{L^\infty}^2) \| x^{m+k-\frac12- \frac{1}{200}} \mathcal{U}_k \|_{L^2_y}^2 \\
\le & C_\lambda \mathcal{I}_{4 + \frac12, 0}(x) + \lambda \mathcal{D}_{5,0} + \eps^{\frac14} \mathcal{I}_5(x),
\end{align*}
where we have again used the bound \eqref{ohm:1}. Third, we have 
\begin{align*}
|\bold{T}^{(Err, 3)}_{k}| \lesssim & \| \bold{b} y x^{1-2m} \|_{L^\infty} \| x^{m+k-\frac12- \frac{1}{200}} \frac{\mathcal{Q}_k}{y} \|_{L^2_y} \| x^{m+k-\frac12- \frac{1}{200}} \mathcal{U}_k \|_{L^2_y} \\
\lesssim & \| \bold{b} y x^{1-2m} \|_{L^\infty} \| x^{m+k-\frac12- \frac{1}{200}} \mathcal{U}_k \|_{L^2_y}^2 \\
\le & C_\lambda \mathcal{I}_{4 + \frac12, 0}(x) + \lambda \mathcal{D}_{5,0} + \eps^{\frac14} \mathcal{I}_5(x)\\
\le & C_\lambda \mathcal{I}_{4 + \frac12, 0}(x) + \lambda \mathcal{D}_{5,0} + \eps^{\frac14} \mathcal{I}_5(x),
\end{align*}
Fourth, we have 
\begin{align*}
|\bold{T}^{(Err, 4)}_{k}| \lesssim & \| \frac{\bold{c}}{\sqrt{\bar{u}}} \frac{1}{\langle \eta \rangle} x^{\frac12 - m} \|_{L^\infty} \| \mathcal{U}_k \langle \eta \rangle x^{k + m - \frac12- \frac{1}{200}} \|_{L^2_y} \| \sqrt{\bar{u}} \p_y \mathcal{U}_k x^{k - \frac{1}{200}} \|_{L^2_y} \\
\le & C_\lambda \mathcal{I}_{4 + \frac12, 0}(x) + \lambda \mathcal{D}_{5,0} + \lambda \overline{\mathcal{D}}_{5,0} + \eps^{\frac14} \mathcal{I}_5(x).
\end{align*}
\vspace{2 mm}

\noindent \textit{Diffusion Terms:} Upon integrating by parts once in $y$, we have 
\begin{align*}
\bold{D}_{k} = & \int_{\mathbb{R}_+} \p_y u^{(k)} \p_y \mathcal{U}_k x^{2k - \frac{1}{100}}\ud y  +x^{2k- \frac{1}{100}} \p_y u^{(k)} \mathcal{U}_k|_{y = 0} \\
= & \int_{\mathbb{R}_+} \p_y \{ \mu \mathcal{U}_k + \mu_y \mathcal{Q}_k \} \p_y \mathcal{U}_k x^{2k - \frac{1}{100}} \ud y + x^{2k - \frac{1}{100}} \p_y \{ \mu \mathcal{U}_k + \mu_y \mathcal{Q}_k \} \mathcal{U}_k|_{y = 0} \\
= & \int_{\mathbb{R}_+} \mu |\p_y \mathcal{U}_k|^2x^{2k - \frac{1}{100}} \ud y + \int_{\mathbb{R}_+}2 \mu_y \mathcal{U}_k \p_y \mathcal{U}_k x^{2k - \frac{1}{100}} \ud y + \int_{\mathbb{R}_+} \mu_{yy} \mathcal{Q}_k \p_y \mathcal{U}_k x^{2k - \frac{1}{100}} \ud y \\
&+ 2x^{2k - \frac{1}{100}} \mu_y |\mathcal{U}_k|^2|_{y = 0} \\
= & \int_{\mathbb{R}_+} \mu |\p_y \mathcal{U}_k|^2x^{2k - \frac{1}{100}} \ud y -  \int_{\mathbb{R}_+}2 \mu_{yy} |\mathcal{U}_k|^2  x^{2k - \frac{1}{100}} \ud y + \int_{\mathbb{R}_+} \mu_{yy} \mathcal{Q}_k \p_y \mathcal{U}_k x^{2k - \frac{1}{100}} \ud y \\
& + x^{2k - \frac{1}{100}} \mu_y |\mathcal{U}_k|^2|_{y = 0} \\
= &x^{2k - \frac{1}{100}} \mu_y |\mathcal{U}_k|^2|_{y = 0} + \int_{\mathbb{R}_+} \mu |\p_y \mathcal{U}_k|^2x^{2k - \frac{1}{100}} \ud y -  \int_{\mathbb{R}_+}2 \bar{u}_{yy} |\mathcal{U}_k|^2  x^{2k - \frac{1}{100}} \ud y \\
 & + \int_{\mathbb{R}_+} \bar{u}_{yy} \mathcal{Q}_k \p_y \mathcal{U}_k x^{2k - \frac{1}{100}} \ud y + \int_{\mathbb{R}_+}\eps u_{yy} \mathcal{Q}_k \p_y \mathcal{U}_k x^{2k - \frac{1}{100}} \ud y - \int_{\mathbb{R}_+}2 \eps u_{yy} |\mathcal{U}_k|^2  x^{2k - \frac{1}{100}} \ud y \\
= &x^{2k - \frac{1}{100}} \mu_y |\mathcal{U}_k|^2|_{y = 0} + \int_{\mathbb{R}_+} \mu |\p_y \mathcal{U}_k|^2x^{2k - \frac{1}{100}} \ud y -  \int_{\mathbb{R}_+}2 \bar{u}_{yy} |\mathcal{U}_k|^2  x^{2k - \frac{1}{100}} \ud y + \sum_{i = 1}^3 \bold{D}_{k}^{(Err, i)}
\end{align*}
We now estimate the error terms appearing above. First, we have
\begin{align} \n
|\bold{D}_k^{(Err, 1)}| \le & \Big\| \frac{\bar{u}_{yy}}{\sqrt{\bar{u}}} x^{\frac12- m} y \Big\|_{L^\infty} \Big\| \frac{\mathcal{Q}_k}{y} x^{k - \frac12 + m - \frac{1}{200}} \Big\|_{L^2_y} \| \sqrt{\bar{u}} \p_y \mathcal{U}_k x^{k- \frac{1}{200}} \|_{L^2_y} \\ \label{est:hu:1}
\lesssim & \Big\| \frac{\bar{u}_{yy}}{\sqrt{\bar{u}}} x^{\frac12- m} y \Big\|_{L^\infty} \| \mathcal{U}_k x^{k - \frac12 + m- \frac{1}{200}} \|_{L^2_y} \| \sqrt{\bar{u}} \p_y \mathcal{U}_k x^{k- \frac{1}{200}} \|_{L^2_y} \\ \n
\lesssim & \| \bar{u}_{yy} x^{1- 2m} \langle \eta \rangle \|_{L^\infty} (C_\lambda \mathcal{I}_{4 + \frac12, 0}(x) + \lambda \mathcal{D}_{5,0} + \eps^{\frac14} \mathcal{I}_5(x))^{\frac12} (\overline{\mathcal{D}}_{k,0}(x))^{\frac12}.
\end{align}
Motivated by the bound above, we are led to define the ``coefficient norm" (stated here for an abstract function, $f(x, y)$):
\begin{align}
\| f \|_{\text{Coeff},1} := \| f x^{1-2m} \langle \eta \rangle \|_{L^\infty}.
\end{align}
To estimate $\bold{D}_k^{(Err, 2)}$, we invoke the equation \eqref{lov:1} to substitute for $u_{yy}$ as follows:
\begin{align*}
\bold{D}_k^{(Err, 2)} = & \int_{\mathbb{R}_+}\eps ( \bar{u} + \eps u) u_x \mathcal{Q}_k \p_y \mathcal{U}_k x^{2k - \frac{1}{100}} \ud y +  \int_{\mathbb{R}_+}\eps \bar{v} u_y  \mathcal{Q}_k \p_y \mathcal{U}_k x^{2k - \frac{1}{100}} \ud y \\
& +   \int_{\mathbb{R}_+}\eps\bar{u}_x u  \mathcal{Q}_k \p_y \mathcal{U}_k x^{2k - \frac{1}{100}} \ud y + \int_{\mathbb{R}_+}\eps (\bar{u}_y + \eps u_y) v  \mathcal{Q}_k \p_y \mathcal{U}_k x^{2k - \frac{1}{100}} \ud y = \sum_{i = 1}^4 \bold{D}_k^{(Err, 2, i)}.
\end{align*}
In order to bound these terms, we notice due to \eqref{est:hu:1} it suffices to check that each coefficient appearing above is bounded in the norm $\| \cdot \|_{\text{Coeff},1}$. Indeed, we have 
\begin{align*}
\| \bar{u} u_x \|_{\text{Coeff},1} = & \| \bar{u} u_x x^{1-2m} \langle \eta \rangle \|_{L^\infty} \lesssim \| \bar{u} x^{-m} \|_{L^\infty} \| u_x x^{1-m} \langle \eta \rangle \|_{L^\infty} \lesssim \eps^{-\frac12}, \\
\| u u_x \|_{\text{Coeff},1} = & \| u u_x x^{1-2m} \langle \eta \rangle \|_{L^\infty} \lesssim \| u x^{-m} \|_{L^\infty} \| u_x x^{1-m} \langle \eta \rangle \|_{L^\infty} \lesssim \eps^{-1}, \\
\| \bar{v} u_y \|_{\text{Coeff},1} = & \| \bar{v} u_y x^{1-2m} \langle \eta \rangle \|_{L^\infty} \lesssim \| \frac{\bar{v}}{\langle \eta \rangle} x^{\frac12(1-m)} \|_{L^\infty} \| u_y x^{\frac12 - \frac{3m}{2}} \langle \eta \rangle^2 \|_{L^\infty} \lesssim \eps^{-\frac12}, \\
\| \bar{u}_x u \|_{\text{Coeff},1} = & \| \bar{u}_x u x^{1-2m} \langle \eta \rangle \|_{L^\infty} \lesssim \| \bar{u}_x x^{-m + 1} \|_{L^\infty} \| u x^{-m} \langle \eta \rangle \|_{L^\infty} \lesssim \eps^{-\frac12}, \\
\| \bar{u}_y v \|_{\text{Coeff},1} = &\| \bar{u}_y v x^{1-2m} \langle \eta \rangle \|_{L^\infty} \lesssim \| \bar{u}_y x^{\frac12 - \frac{3m}{2}} \langle \eta \rangle^2 \|_{L^\infty} \| \frac{v}{\langle \eta \rangle} x^{\frac12(1-m)} \|_{L^\infty} \lesssim \eps^{-\frac12} \\
\| u_y v \|_{\text{Coeff},1} = &\| u_y v x^{1-2m} \langle \eta \rangle \|_{L^\infty} \lesssim \| u_y x^{\frac12 - \frac{3m}{2}} \langle \eta \rangle^2 \|_{L^\infty} \| \frac{v}{\langle \eta \rangle} x^{\frac12(1-m)} \|_{L^\infty} \lesssim \eps^{-1} 
\end{align*}
all upon invoking the bounds \eqref{jurassic:1} -- \eqref{jurassic:3}. From here, we have 
\begin{align*}
|\bold{D}_k^{(Err, 2)} | \lesssim & \eps \| u_{yy} \|_{\text{Coeff},1} (C_\lambda \mathcal{I}_{4 + \frac12, 0}(x) + \lambda \mathcal{D}_{5,0} + \eps^{\frac14} \mathcal{I}_5(x))^{\frac12} (\overline{\mathcal{D}}_{k,0}(x))^{\frac12} \\
\lesssim & \eps^{\frac12} (C_\lambda \mathcal{I}_{4 + \frac12, 0}(x) + \lambda \mathcal{D}_{5,0} + \eps^{\frac14} \mathcal{I}_5(x))^{\frac12} (\overline{\mathcal{D}}_{k,0}(x))^{\frac12}.
\end{align*}
We now move to the estimate of $\bold{D}_k^{(Err, 3)}$, for which we first notice again upon invoking \eqref{lov:1}
\begin{align*}
\| u_{yy} x^{1-2m}\|_{L^\infty} \lesssim &\| \bar{u} u_x x^{1-2m} \|_{L^\infty} + \eps \| u u_x x^{1-2m} \|_{L^\infty} + \| \bar{v} u_y x^{1-2m} \|_{L^\infty} \\
& + \| \bar{u}_x u x^{1-2m} \|_{L^\infty}  + \| \bar{u}_y v x^{1-2m} \|_{L^\infty} + \eps \| u_y v x^{1-2m} \|_{L^\infty} \lesssim \eps^{-\frac12},
\end{align*}
where we have appealed to the bootstrap bounds \eqref{jurassic:1} -- \eqref{jurassic:3}. Correspondingly, we have 
\begin{align*}
|\bold{D}_k^{(Err, 3)}| \lesssim & \eps \| u_{yy} x^{1-2m} \|_{L^\infty} \| \mathcal{U}_k x^{k - \frac12 + m - \frac{1}{200} }\|_{L^2}^2 \lesssim \eps^{\frac12}  \| \mathcal{U}_k x^{k - \frac12 + m- \frac{1}{200} }\|_{L^2}^2 \\
\lesssim & \eps^{\frac12} C_\lambda \mathcal{I}_{4 + \frac12, 0}(x) + \lambda \mathcal{D}_{5,0} + \eps^{\frac14} \mathcal{I}_5(x).
\end{align*}
This concludes the proof of the lemma. 
\end{proof}

\section{Bounds on Error Terms} \label{sec:6}

\subsection{Linear Error Terms}

\begin{lemma} \label{lem:61} The following bounds are valid on $F_{\text{Comm},k}$:
\begin{align} \n
\| \frac{1}{\sqrt{\bar{u}}} F_{\text{Comm},k} \langle \eta \rangle^{10-k} x^{k + \frac12 - \frac{m}{2}- \frac{1}{200}} \|_{L^2_y}^2 \lesssim   & \sum_{k' = 0}^{k-1} d_{k', 10-k'}(x) +  \sum_{k' = 0}^{k-1} c_{k', 10-k'}(x) \\ \label{tom:petty:1}
&+  \sum_{k' = 0}^{k-1} d_{k' + \frac12, 10-(k' + 1)}^{(Y)}(x).
\end{align}
\end{lemma}
\begin{proof} We proceed term by term through \eqref{Fcommk}, which we label $I_1, \dots, I_4$. First, we have 
\begin{align*}
|I_1| \lesssim & \| \frac{1}{\sqrt{\bar{u}}} F^{(1)}_{\text{Comm},k} \langle \eta \rangle^{10-k} x^{k + \frac12 - \frac{m}{2}- \frac{1}{200}} \|_{L^2_y}^2 \lesssim  \sum_{k' = 0}^{k-1} \| \frac{1}{\bar{u}} \p_x^{k-k'} \bar{u} u^{(k'+1)} \langle \eta \rangle^{10-k} x^{k+ \frac12- \frac{1}{200}} \|_{L^2_y}^2 \\
\lesssim &\sum_{k' = 0}^{k-1} \Big\| \frac{\p_x^{k-k'}}{\bar{u}} x^{k-k'} \Big\|_{L^\infty}^2 \| u^{(k')}_x \langle \eta \rangle^{10-k} x^{k' + \frac12- \frac{1}{200}} \|_{L^2_y}^2 \\
\lesssim & \sum_{k' = 0}^{k-1} d_{k' + \frac12, 10-k}^{(Y)}(x) \\
\lesssim & \sum_{k' = 0}^{k-1} d_{k' + \frac12, 10-(k' + 1)}^{(Y)}(x).
\end{align*}
Second, we have 
\begin{align*}
|I_2| \lesssim &\| \frac{1}{\sqrt{\bar{u}}} F^{(2)}_{\text{Comm},k} \langle \eta \rangle^{10-k} x^{k + \frac12 - \frac{m}{2}- \frac{1}{200}} \|_{L^2_y}^2 \lesssim  \sum_{k' = 0}^{k-1} \| \frac{1}{\bar{u}}\p_x^{k-k'+1} \bar{u} u^{(k')} \langle \eta \rangle^{10-k} x^{k+ \frac12- \frac{1}{200}} \|_{L^2_y}^2 \\
\lesssim &  \sum_{k' = 0}^{k-1} \Big\| \frac{\p_x^{k-k' + 1} \bar{u}}{\bar{u}} x^{k-k' + 1} \Big\|_{L^\infty}^2 \| u^{(k')} \langle \eta \rangle^{10-k} x^{k' - \frac12- \frac{1}{200}} \|_{L^2_y}^2 \\
\lesssim & \sum_{k' = 0}^{k-1} c_{k', 10-k}(x) \\
\lesssim & \sum_{k' = 0}^{k-1} c_{k', 10-k'}(x).
\end{align*}
Third, we have 
\begin{align*}
|I_3| \lesssim &\| \frac{1}{\sqrt{\bar{u}}} F^{(3)}_{\text{Comm},k} \langle \eta \rangle^{10-k} x^{k + \frac12 - \frac{m}{2}- \frac{1}{200}} \|_{L^2_y}^2 \lesssim  \sum_{k' = 0}^{k-1} \| \frac{1}{\bar{u}}\p_x^{k-k'} \bar{v} u^{(k')}_y \langle \eta \rangle^{10-k} x^{k+ \frac12- \frac{1}{200}} \|_{L^2_y}^2 \\
\lesssim & \sum_{k' = 0}^{k-1} \Big\| \frac{\p_x^{k-k'} \bar{v}}{\bar{u}} \frac{1}{\langle \eta \rangle} x^{k-k' + \frac12 + \frac{m}{2}}\Big\|_{L^\infty}^2 \| u^{(k')}_y x^{k' - \frac{m}{2}- \frac{1}{200}} \langle \eta \rangle^{11-k} \|_{L^2_y}^2 \\
\lesssim &  \sum_{k' = 0}^{k-1} d_{k', 10-k'}(x).
\end{align*}
Fourth, we have 
\begin{align*}
|I_4| \lesssim &\| \frac{1}{\sqrt{\bar{u}}} F^{(4)}_{\text{Comm},k} \langle \eta \rangle^{10-k} x^{k + \frac12 - \frac{m}{2}- \frac{1}{200}} \|_{L^2_y}^2 \lesssim \sum_{k' = 0}^{k-1} \| \frac{1}{\bar{u}} \p_x^{k-k'} \bar{u}_{y} v^{(k')} \langle \eta \rangle^{10-k} x^{k+ \frac12- \frac{1}{200}} \|_{L^2_y}^2 \\
\lesssim & \sum_{k' = 0}^{k-1} \| \frac{1}{\bar{u}} \p_x^{k-k'} \bar{u}_{y} v^{(k')} \langle \eta \rangle^{10-k} x^{k+ \frac12- \frac{1}{200}} \chi(\eta \le 1) \|_{L^2_y}^2 + \sum_{k' = 0}^{k-1} \| \frac{1}{\bar{u}} \p_x^{k-k'} \bar{u}_{y} v^{(k')} \langle \eta \rangle^{10-k} x^{k+ \frac12- \frac{1}{200}} \chi(\eta > 1) \|_{L^2_y}^2 \\
:= & I_{4, \le} + I_{4, >}.  
\end{align*}
First, we treat the far-field term, where $\bar{u} \gtrsim x^m$. In this case, we have 
\begin{align*}
|I_{4, >}| \lesssim &\sum_{k' = 0}^{k-1} \|  \p_x^{k-k'} \bar{u}_{y} v^{(k')} \langle \eta \rangle^{10-k} x^{k+ \frac12-m- \frac{1}{200}} \chi(\eta > 1) \|_{L^2_y}^2 \\
\lesssim &\sum_{k' = 0}^{k-1} \|  \p_x^{k-k'} \bar{u}_{y} y x^{k-k' - m} \langle \eta \rangle^{10-k}  \|_{L^\infty}^2 \Big\| \frac{ v^{(k')} }{y} x^{k' + \frac12- \frac{1}{200}}  \Big\|_{L^2_y}^2 \\
\lesssim &\sum_{k' = 0}^{k-1}  \Big\| u^{(k')}_x x^{k' + \frac12- \frac{1}{200}}  \|_{L^2_y}^2 \\
\lesssim & \sum_{k' = 0}^{k-1} d_{k' + \frac12, 0}^{(Y)}(x).
\end{align*}
Next, we treat the near-field term. In this case $\bar{u} \gtrsim x^{m} \eta = x^m \frac{y}{x^{\frac12(1-m)}}$. Consequently, we have 
\begin{align*}
|I_{4, \le}| \lesssim &\sum_{k' = 0}^{k-1} \|  \p_x^{k-k'} \bar{u}_{y} \frac{ v^{(k')}}{y} \langle \eta \rangle^{10-k} x^{k+ \frac12-m- \frac{1}{200}} x^{\frac12(1-m)} \chi(\eta \le 1) \|_{L^2_y}^2 \\
\lesssim &\sum_{k' = 0}^{k-1} \| \p_x^{k-k'} \bar{u}_y x^{(k-k') + \frac12 - \frac{3m}{2}} \langle \eta \rangle^{10-k} \|_{L^\infty}^2 \Big\| \frac{v^{(k')}}{y} x^{k' + \frac12- \frac{1}{200}} \|_{L^2_y}^2 \\
\lesssim &\sum_{k' = 0}^{k-1}  \Big\| u^{(k')}_x x^{k' + \frac12- \frac{1}{200}}  \|_{L^2_y}^2 \\
\lesssim & \sum_{k' = 0}^{k-1} d_{k' + \frac12, 0}^{(Y)}(x).
\end{align*}
\end{proof}

\subsection{Nonlinear Error Terms}

\begin{lemma} \label{lem:62} For $0 \le k \le 4$, 
\begin{align} \label{bob:seger:1}
\| \frac{1}{\bar{u}} \p_x^k \mathcal{Q}[u, v] \langle \eta \rangle^{10-k } x^{k + \frac12- \frac{1}{200}} \|_{L^2_y}^2 \lesssim & \eps^{-1} \Big( \sum_{k' = 0}^{4} d_{k', 10-k'}(x) +  \sum_{k' = 0}^{4} d_{k' + \frac12, 10-(k' + 1)}^{(Y)}(x) \Big).
\end{align}
\end{lemma}
\begin{proof} We first address the term $u u_x$, which will contribute: 
\begin{align} \n
 &\sum_{k' = 0}^k \| \frac{1}{\bar{u}} \p_x^{k-k'} u \p_x^{k'} u_x x^{k + \frac12- \frac{1}{200}} \langle \eta \rangle^{10-k} \|_{L^2_y} \\
 = & \sum_{k' = 0}^k \bold{1}_{k - k' \le 3} \| \frac{1}{\bar{u}} \p_x^{k-k'} u \p_x^{k'} u_x x^{k + \frac12- \frac{1}{200}} \langle \eta \rangle^{10-k} \|_{L^2_y} +  \bold{1}_{k = 4} \| \frac{1}{\bar{u}} \p_x^{4} u  u_x x^{4 + \frac12- \frac{1}{200}} \langle \eta \rangle^{6} \|_{L^2_y}.
\end{align} 
Above, we have split out the case when $k - k' = 4$, and noticed that this can only happen if $k =4$ and $k' = 0$. In this case, one notices furthermore that the second term on the right-hand side above is identical to the term when $k = 4$ and $k' = 3$ from the first term on the RHS above. Therefore, it suffices to estimate just the first term, for all $0 \le k \le 4$. In this case, we have 
\begin{align*}
 &\sum_{k' = 0}^k \bold{1}_{k - k' \le 3} \| \frac{1}{\bar{u}} \p_x^{k-k'} u \p_x^{k'} u_x x^{k + \frac12- \frac{1}{200}} \langle \eta \rangle^{10-k} \|_{L^2_y}^2 \\
 \lesssim & \sum_{k' = 0}^k \bold{1}_{k - k' \le 3} \Big\| \frac{1}{\bar{u}} \p_x^{k-k'} u  \langle \eta \rangle x^{k-k' + \frac14 + \frac{3m}{4}} \|_{L^\infty}^2 \| u^{(k')}_x x^{k' + \frac12- \frac{1}{200}} \langle \eta \rangle^{10-(k+1)'} \|_{L^2_y} \\
 \lesssim & \eps^{-1} \sum_{k' = 0}^4 d_{k' + \frac12, 10 - (k'+1)}^{(Y)}(x). 
\end{align*}
We next address the term $v u_y$, which will contribute: 
\begin{align*}
&\sum_{k' = 0}^k \| \frac{1}{\bar{u}} \p_x^{k-k'} u_y \p_x^{k'} v x^{k + \frac12- \frac{1}{200}} \langle \eta \rangle^{10-k} \|_{L^2_y}^2 \\
= & \sum_{k' = 0}^k \bold{1}_{k' \le 2} \| \frac{1}{\bar{u}} \p_x^{k-k'} u_y \p_x^{k'} v x^{k + \frac12- \frac{1}{200}} \langle \eta \rangle^{10-k} \|_{L^2_y}^2 +  \sum_{k' = 0}^k \bold{1}_{3 \le k' \le 4}  \| \frac{1}{\bar{u}} \p_x^{k-k'} u_y \p_x^{k'} v x^{k + \frac12- \frac{1}{200}} \langle \eta \rangle^{10-k} \|_{L^2_y}^2 \\
=: & \mathcal{K}_1 + \mathcal{K}_2. 
\end{align*}
To estimate $\mathcal{K}_1$, we have 
\begin{align*}
|\mathcal{K}_1| \lesssim &  \sum_{k' = 0}^k \bold{1}_{k' \le 2} \Big\| \frac{1}{\bar{u}} v^{(k')} x^{k' + \frac34 + \frac{5m}{4}} \Big\|_{L^\infty}^2 \| u^{(k-k')}_y x^{k-k' - \frac{m}{2}- \frac{1}{200}} \langle \eta \rangle^{10-k} \|_{L^2_y}^2 \\
\lesssim & \eps^{-1}  \sum_{k' = 0}^k \| u^{(k-k')}_y x^{k-k' - \frac{m}{2}- \frac{1}{200}} \langle \eta \rangle^{10-(k-k')} \|_{L^2_y}^2 \\
\lesssim & \eps^{-1} \sum_{k' = 0}^4 d_{k', 10-k'}(x). 
\end{align*}
To estimate $\mathcal{K}_2$, we notice that $3 \le k' \le 4$ implies that $0 \le k - k' \le 1$ (since $k \le 4$). This allows us to invoke the estimate \eqref{trader:joes:1}. We proceed as follows:
\begin{align*}
|\mathcal{K}_2| \lesssim &  \sum_{k' = 0}^k \bold{1}_{3 \le k' \le 4}  \| \frac{1}{\bar{u}} \p_x^{k-k'} u_y \p_x^{k'} v x^{k + \frac12- \frac{1}{200}} \langle \eta \rangle^{10-k} \chi(\eta \ge 1) \|_{L^2_y}^2 \\
& + \sum_{k' = 0}^k \bold{1}_{3 \le k' \le 4}  \| \frac{1}{\bar{u}} \p_x^{k-k'} u_y \p_x^{k'} v x^{k + \frac12- \frac{1}{200}} \langle \eta \rangle^{10-k} \chi(\eta \le 1) \|_{L^2_y}^2 =: \mathcal{K}_{2, \ge} + \mathcal{K}_{2,\le}. 
\end{align*}
For the bound of $\mathcal{K}_{2,\ge}$, we again invoke the bound that $\bar{u} \ge x^m$ in the support of $\chi(\eta \ge 1)$. We also use that $k \ge 3$ implies that $\langle \eta \rangle^{10-k} \le \langle \eta \rangle^7$. This then gives 
\begin{align*}
|\mathcal{K}_{2,\ge}| \lesssim &  \sum_{k' = 0}^k \bold{1}_{3 \le k' \le 4}  \|  \p_x^{k-k'} u_y \p_x^{k'} v x^{k + \frac12-m- \frac{1}{200}} \langle \eta \rangle^{10-k} \chi(\eta \ge 1) \|_{L^2_y}^2 \\
\lesssim &  \sum_{k' = 0}^k \bold{1}_{3 \le k' \le 4}  \|  \p_x^{k-k'} u_y \p_x^{k'} v x^{k + \frac12-m- \frac{1}{200}} \langle \eta \rangle^{7} \chi(\eta \ge 1) \|_{L^2_y}^2 \\
\lesssim &  \sum_{k' = 0}^k \bold{1}_{3 \le k' \le 4}  \|  \p_x^{k-k'} u_y \langle \eta \rangle^7y x^{k-k'  - m} \|_{L^\infty}^2 \Big\| \frac{  \p_x^{k'} v}{y} x^{k' + \frac12- \frac{1}{200}} \Big\|_{L^2_y}^2 \\
\lesssim &  \sum_{k' = 0}^k \bold{1}_{3 \le k' \le 4}  \|  \p_x^{k-k'} u_y \langle \eta \rangle^8 x^{k-k' + \frac12 - \frac{3m}{2}} \|_{L^\infty}^2 \Big\| \frac{  \p_x^{k'} v}{y} x^{k' + \frac12- \frac{1}{200}} \Big\|_{L^2_y}^2 \\
\lesssim & \eps^{-1} \sum_{k' = 0}^k  \| u^{(k')}_x x^{k' + \frac12- \frac{1}{200}} \|_{L^2_y}^2 \\
\lesssim & \eps^{-1} \sum_{k' = 0}^4 d^{(Y)}_{k' + \frac12, 0}.
\end{align*}
The treatment of the case $\mathcal{K}_{2,\le}$ is completely analogous. 
\end{proof}

\begin{lemma} \label{lem:63} For $k = 5$, the quadratic term defined in \eqref{def:qklo} satisfies the following bound: 
\begin{align}   \label{bob:seger:2}
\|   \mathcal{Q}_{5,lo}[u, v]  x^{5 + \frac12-m- \frac{1}{200}} \|_{L^2_y}^2 \lesssim & \eps^{-1} \Big( \sum_{k' = 0}^{5} d_{k', 0}(x) +  \sum_{k' = 0}^{4} d_{k' + \frac12, 0}^{(Y)}(x) \Big).
\end{align}
\end{lemma}
\begin{proof} We first treat the term 
\begin{align*}
&\|  \sum_{k' = 0}^4 u^{(5-k')} u^{(k' + 1)}  x^{k + \frac12-m- \frac{1}{200}} \|_{L^2_y}^2 \\
\lesssim & \| \sum_{k' = 0}^{2} u^{(5-k')} u^{(k' + 1)}  x^{k + \frac12-m- \frac{1}{200}} \|_{L^2_y}^2 + \|\sum_{k' = 3}^{4} u^{(5-k')} u^{(k' + 1)} x^{k + \frac12-m- \frac{1}{200}} \|_{L^2_y}^2.
\end{align*}
We notice that the second sum on the RHS above (corresponding to $k' = 3, 4$) is symmetric (each such term is represented in the first sum). Therefore, it suffices to estimate the first sum. We have 
\begin{align*}
&\| \sum_{k' = 0}^2 u^{(5-k')} u^{(k' + 1)} x^{5 + \frac12-m- \frac{1}{200}} \|_{L^2_y}^2 \lesssim  \sum_{k' = 0}^2 \| u^{(k'+1)} x^{k' + 1-m} \|_{L^\infty}^2 \| u^{(4-k')}_x x^{(4-k') + \frac12- \frac{1}{200}} \|_{L^2_y}^2 \\
\lesssim & \eps^{-1} \sum_{k' = 0}^2 \| u^{(4-k')}_x x^{(4-k') + \frac12- \frac{1}{200}} \|_{L^2_y}^2 \lesssim \eps^{-1} \sum_{k' = 2}^4 d^{(Y)}_{k' + \frac12, 0}(x). 
\end{align*}
It remains to treat the term
\begin{align*}
&\|  \sum_{k' = 0}^{4} \p_y u^{(5-k')} v^{(k')}  x^{5 +\frac12-m- \frac{1}{200}} \|_{L^2_y}^2 \\
\lesssim & \sum_{k' = 0}^5 \bold{1}_{k' \le 2} \| \p_x^{5-k'} u_y \p_x^{k'} v x^{5 + \frac12-m- \frac{1}{200}} \langle \eta \rangle^{5} \|_{L^2_y}^2 +  \sum_{k' = 0}^5 \bold{1}_{3 \le k' \le 4}  \|  \p_x^{5-k'} u_y \p_x^{k'} v x^{5 + \frac12-m- \frac{1}{200}} \langle \eta \rangle^{5} \|_{L^2_y}^2 \\
=: & \mathcal{K}_1 + \mathcal{K}_2. 
\end{align*}
For $\mathcal{K}_1$, we put the $v$ term in $L^\infty$ as follows: 
\begin{align*}
\mathcal{K}_1 \lesssim & \sum_{k' = 0}^5 \bold{1}_{k' \le 2} \|\p_x^{k'} v x^{k' + \frac12 - \frac{m}{2} }\|_{L^\infty}^2 \| \p_x^{5-k'} u_y x^{5-k'- \frac{m}{2}- \frac{1}{200}}  \|_{L^2_y}^2 \\
\lesssim &  \eps^{-1} \sum_{k' = 0}^5 \bold{1}_{k' \le 2}  \| \p_x^{5-k'} u_y x^{5-k'- \frac{m}{2}- \frac{1}{200}} \|_{L^2_y}^2 \\
\lesssim & \eps^{-1} \sum_{k' = 3}^5 d_{k', 0}(x). 
\end{align*}
For $\mathcal{K}_2$, we put the $u_y$ term in $L^\infty$ as follows. Notice when $k' =  3, 4$, $5- k' \le 2$, and so we can use the bound \eqref{jurassic:4}.
\begin{align*}
\mathcal{K}_2 \lesssim &  \sum_{k' = 3}^4  \|  \p_x^{5-k'} u_y   y x^{5-k'-m} \|_{L^\infty}^2 \| \frac{\p_x^{k'} v}{y} x^{k' + \frac12- \frac{1}{200}} \|_{L^2_y}^2 \\
\lesssim &  \sum_{k' = 3}^4  \|  \p_x^{5-k'} u_y   \langle \eta \rangle x^{5-k' + \frac12 - \frac{3m}{2}} \|_{L^\infty}^2 \| \frac{\p_x^{k'} v}{y} x^{k' + \frac12- \frac{1}{200}} \|_{L^2_y}^2 \lesssim  \eps^{-1} \sum_{k' = 3}^4  \| u^{(k')}_x x^{k' + \frac12- \frac{1}{200}} \|_{L^2_y}^2 \\
\lesssim & \eps^{-1}\sum_{k' = 3}^4 d^{(Y)}_{k' + \frac12, 0}.
\end{align*}

\end{proof}

Pairing the bounds \eqref{tom:petty:1}, \eqref{bob:seger:1}, and \eqref{bob:seger:2} with the bounds \eqref{G:League:na:1} -- \eqref{G:League:na:4} yields the following: 
\begin{corollary}\label{wehave}The following bounds are valid on $F_{\text{Comm},k}$:
\begin{align} \label{tom:petty:2}
\| \frac{1}{\sqrt{\bar{u}}} F_{\text{Comm},k} \langle \eta \rangle^{10-k} x^{k + \frac12 - \frac{m}{2}- \frac{1}{200}} \|_{L^2_y}^2 \le & C_\lambda \widehat{\mathcal{I}}_{(k-1) + \frac12}(x) + \lambda \mathcal{D}_{k,0}(x), \\ \label{tp:3}
\| \frac{1}{\bar{u}} \p_x^k \mathcal{Q}[u, v] \langle \eta \rangle^{10-k } x^{k + \frac12- \frac{1}{200}} \|_{L^2_y}^2 \lesssim & \eps^{-1}  \widehat{\mathcal{I}}_5(x), \\ \label{tp:4}
\|   \mathcal{Q}_{5,lo}[u, v]  x^{5 + \frac12-m- \frac{1}{200}} \|_{L^2_y}^2 \lesssim & \eps^{-1} \widehat{\mathcal{I}}_5(x)
\end{align}
\end{corollary}

\section{Completing the Bootstrap and Proof of Theorem \ref{thm:1}} \label{sec:7}

\begin{proposition} \label{prop:main:1} Assume $X_\ast$ is such that the bootstrap \eqref{boots:1} is valid. Then in fact the constant of $10$ is improved to $5$: 
\begin{align} \label{boots:2}
\sup_{0 \le x \le X_\ast} \mathcal{E}_{\text{Quasi}}(x) + \int_0^{X_\ast} (\mathcal{D}(x) + \mathcal{CK}(x) + \mathcal{B}(x) ) dx \le 5 \mathcal{E}_{\text{Init}}, 
\end{align}
Therefore $X_\ast = \infty$. 
\end{proposition}
\begin{proof} The proof follows in three steps which are delineated below. 

\vspace{2 mm}

\noindent \underline{Step 1: Source Terms} We estimate using the definition \eqref{Skn:def:1} and \eqref{Pr:lin2} 
\begin{align} \n
|\bold{S}_{k,10-k}(x)| \lesssim & |\int F_{k,comm} U_k x^{2k- \frac{1}{100}} \langle \eta \rangle^{20-2k}| +  |\int \eps \p_x^k \mathcal{Q}[u,v] U_k x^{2k- \frac{1}{100}} \langle \eta \rangle^{20-2k}| \\ \n
\lesssim & (\| F_{k,comm} x^{k -m + \frac12- \frac{1}{200}} \langle \eta \rangle^{10-k} \|_{L^2_y} + \eps \| \p_x^k \mathcal{Q} x^{k -m + \frac12- \frac{1}{200}} \langle \eta \rangle^{10-k} \|_{L^2_y} ) \\ \n
& \times  \| U_k x^{k + m - \frac12- \frac{1}{200}} \langle \eta \rangle^{10-k} \|_{L^2_y} \\ \label{midwest:1}
\lesssim &\Big(C_\lambda \widehat{\mathcal{I}}_{(k-1) + \frac12}(x)^{\frac12} + \lambda \mathcal{D}_{k,0}(x)^{\frac12}  + \eps^{\frac12} \widehat{\mathcal{I}}_5(x)^{\frac12} \Big) \mathcal{CK}_{k,10-k}(x)^{\frac12},
\end{align}
for $0 \le k \le 4$. Next, for $k = 5$, we have 
\begin{align} \n
|\bold{S}_{5,0}(x)| \lesssim &  |\int F_{5,comm} \mathcal{U}_5 x^{10 - \frac{1}{100}} | +  |\int \eps \mathcal{Q}_{5,lo}[u,v] \mathcal{U}_5 x^{10- \frac{1}{100}} | \\ \n
\lesssim & (\| F_{5,comm} x^{5 + \frac12 - m- \frac{1}{200}} \|_{L^2_y} + \eps \| \mathcal{Q}_{5,lo}[u, v] x^{5 + \frac12 - m- \frac{1}{200}} \|_{L^2_y}) \| \mathcal{U}_5 x^{5 + m - \frac12- \frac{1}{200}} \|_{L^2_y} \\ \n
\lesssim & \Big(C_\lambda \widehat{\mathcal{I}}_{4 + \frac12}(x)^{\frac12} + \lambda \mathcal{D}_{5,0}(x)^{\frac12}  + \eps^{\frac12}  \widehat{\mathcal{I}}_5(x)^{\frac12} \Big)\overline{\mathcal{CK}}_{5,0}(x)^{\frac12} \\ \label{merced:1}
\lesssim & \Big(C_\lambda \widehat{\mathcal{I}}_{4 + \frac12}(x)^{\frac12} + \lambda \mathcal{D}_{5,0}(x)^{\frac12}  + \eps^{\frac12}  \widehat{\mathcal{I}}_5(x)^{\frac12} \Big)(C_\lambda \mathcal{I}_{4 + \frac12, 0}(x) + \lambda \mathcal{D}_{5,0}(x) + \eps^{\frac14} \mathcal{I}_5(x))^{\frac12}
\end{align}
where we have invoked the bounds \eqref{tom:petty:2} -- \eqref{tp:4} as well as the bound \eqref{ohm:1}. Next, recalling the definition \eqref{def:skY}, we have 
\begin{align} \n
|\bold{S}^{(Y)}_{k,10-(k+1)}(x)| \lesssim & |\int F_{k,comm} \p_x U_k x^{2k+1- \frac{1}{100}} \langle \eta \rangle^{18-2k}| +  |\int \eps \p_x^k \mathcal{Q}[u,v]\p_x U_k x^{2k+1- \frac{1}{100}} \langle \eta \rangle^{18-2k}| \\ \n
\lesssim & (\| F_{k,comm} x^{k -m + \frac12- \frac{1}{200}} \langle \eta \rangle^{10-k} \|_{L^2_y} + \eps \| \p_x^k \mathcal{Q} x^{k -m + \frac12- \frac{1}{200}} \langle \eta \rangle^{10-k} \|_{L^2_y} ) \\
& \times  \| \p_x U_k x^{k + m + \frac12- \frac{1}{200}} \langle \eta \rangle^{10-(k+1)} \|_{L^2_y} \\ \n
\lesssim &\Big(C_\lambda \widehat{\mathcal{I}}_{(k-1) + \frac12}(x)^{\frac12} + \lambda \mathcal{D}_{k,0}(x)^{\frac12}  + \eps^{\frac12} \widehat{\mathcal{I}}_5(x)^{\frac12} \Big) \\ \label{merced:2}
& \times (C_\lambda \mathcal{D}_{k + \frac12, 9-k}^{(Y)}(x) + \lambda \mathcal{D}_{k + 1, 0}(x) + \lambda ( \mathcal{D}_{k,0}(x) + \mathcal{CK}_{k,0}(x)))^{\frac12},
\end{align}
where we have invoked the bound \eqref{train:2} as well as \eqref{tom:petty:2} -- \eqref{tp:3}. Next, we recall the definition \eqref{szA}, after which we proceed as follows. Notice that both $F_{comm,k}|_{y = 0} = 0$ and $\p_x^k \mathcal{Q}[u,v]|_{y = 0} = 0$ which allows us to integrate by parts in $y$ which produces 
\begin{align*}
\bold{S}^{(Z)}_{k,9-k}(x) = &  -\int_{\mathbb{R}_+} G_k \p_{y}^2 U_k \langle \eta \rangle^{18-2k} x^{2k + 1 - m- \frac{1}{100}} \ud y   - (18-2k) \int_{\mathbb{R}_+} G_k \p_{y} U_k \langle \eta \rangle^{{17-2k}} x^{2k + \frac12 - \frac{m}{2}- \frac{1}{100}} \ud y \\
= & \bold{S}^{(Z, 1)}_{k,9-k}(x) + \bold{S}^{(Z, 2)}_{k,9-k}(x).
\end{align*} 
We first estimate 
\begin{align} \n
|\bold{S}^{(Z, 1)}_{k,9-k}(x)| \lesssim & \| \frac{1}{\sqrt{\bar{u}}} F_{k,comm} x^{k + \frac12 - \frac{m}{2}- \frac{1}{200}} \langle \eta \rangle^{9-k} \|_{L^2_y} \| \sqrt{\bar{u}} \p_y^2 U_k x^{k + \frac12 - \frac{m}{2}- \frac{1}{200}} \langle \eta \rangle^{9-k} \|_{L^2_y} \\ \n
& +  \| \frac{1}{\sqrt{\bar{u}}} \eps \p_x^k \mathcal{Q}[u,v] x^{k + \frac12 - \frac{m}{2}- \frac{1}{200}} \langle \eta \rangle^{9-k} \|_{L^2_y} \| \sqrt{\bar{u}} \p_y^2 U_k x^{k + \frac12 - \frac{m}{2}- \frac{1}{200}} \langle \eta \rangle^{9-k} \|_{L^2_y} \\ \label{merced:3}
\lesssim & \Big(C_\lambda \widehat{\mathcal{I}}_{(k-1) + \frac12}(x)^{\frac12} + \lambda \mathcal{D}_{k,0}(x)^{\frac12}  + \eps^{\frac12} \widehat{\mathcal{I}}_5(x)^{\frac12} \Big)  (\mathcal{D}^{(Z)}_{k + \frac12, 9-k}(x))^{\frac12},
\end{align}
and very similarly, 
\begin{align} \n
|\bold{S}^{(Z, 2)}_{k,9-k}(x)| \lesssim & \| F_{k,comm} x^{k + \frac12 - m- \frac{1}{200}} \langle \eta \rangle^{9-k} \|_{L^2_y} \| \p_y U_k x^{k + \frac{m}{2}- \frac{1}{200}} \langle \eta \rangle^{9-k} \|_{L^2_y} \\ \n
& +  \| \eps \p_x^k \mathcal{Q}[u,v] x^{k + \frac12 -m- \frac{1}{200}} \langle \eta \rangle^{9-k} \|_{L^2_y} \| \p_y U_k x^{k + \frac{m}{2}- \frac{1}{200}} \langle \eta \rangle^{9-k} \|_{L^2_y} \\ \label{merced:4}
\lesssim & \Big(C_\lambda \widehat{\mathcal{I}}_{(k-1) + \frac12}(x)^{\frac12} + \lambda \mathcal{D}_{k,0}(x)^{\frac12}  + \eps^{\frac12} \widehat{\mathcal{I}}_5(x)^{\frac12} \Big)  \Big((\mathcal{D}_{k, 9-k}(x))^{\frac12} +( \mathcal{D}^{(Z)}_{k + \frac12, 0}(x))^{\frac12}\Big).
\end{align}
\vspace{2 mm}

\noindent \underline{Step 2: Linear Combination} Inserting the bounds \eqref{midwest:1}, \eqref{merced:1}, \eqref{merced:2}, \eqref{merced:3}, and \eqref{merced:4} into \eqref{cal:1}, \eqref{cal:4}, \eqref{cal:2}, \eqref{cal:3} respectively we find that there are constants (here $0 \le k \le 4$)
\begin{align} \n
&\frac{\p_x}{2} \mathcal{E}_{k,10-k}(x) + \mathcal{CK}_{k,10-k}(x) + \mathcal{CK}^{(P)}_{k,10-k}(x) +  \mathcal{B}_{k}(x) + \mathcal{D}_{k,10-k}(x) \\ \label{pin:1}
&  \qquad \qquad \qquad \qquad   \qquad \qquad   \qquad \qquad \le   C_k \widehat{\mathcal{I}}_{(k-1) + \frac12}(x)  + C_k \eps^{\frac12} \widehat{\mathcal{I}}_5(x), \\ \n
&\frac{\p_x}{2} \mathcal{E}_{k+\frac12,9-k}^{(Y)}(x) + \mathcal{D}^{(Y)}_{k+\frac12,9-k}(x) \le   C_{\delta_k}  \mathcal{D}_{k, 10-k}(x) + \delta_k (  \mathcal{D}_{k+1, 0}(x) + \mathcal{CK}_{k+1,0}(x)   + \mathcal{D}^{(Z)}_{k + \frac12, 0}(x)) \\ \label{pin:2}
& \qquad \qquad \qquad \qquad   \qquad \qquad   \qquad \qquad  + C_k \eps^{\frac12} \widehat{\mathcal{I}}_5(x), \\ \n
&\frac{\p_x}{2} \mathcal{E}_{k+\frac12,9-k}^{(Z)}(x)+  \mathcal{B}^{(Z)}_{k+\frac12}(x)  + \mathcal{D}^{(Z)}_{k+\frac12,9-k}(x) \le C_{\delta_k} ( \mathcal{D}_{k,10-k}(x) +  \mathcal{CK}_{k,10-k}(x) + \mathcal{B}_k(x) ) \\ \label{pin:3}
& \qquad \qquad \qquad \qquad   \qquad \qquad   \qquad \qquad +  \delta_k \mathcal{D}^{(Y)}_{k + \frac12,0}(x) + C_k \eps^{\frac12} \widehat{\mathcal{I}}_5(x), \\ \label{pin:4}
&\frac{\p_x}{2} \overline{\mathcal{E}}_{5,0}(x) +\overline{\mathcal{CK}}_{5,0}(x) +  \overline{\mathcal{CK}}^{(P)}_{5,0}(x)+  \overline{\mathcal{B}}_{0}(x) + \overline{\mathcal{D}}_{5,0}(x) \le C_5 \widehat{\mathcal{I}}_{4 + \frac12}(x)  + C_5 \eps^{\frac12} \widehat{\mathcal{I}}_5(x).
\end{align}
We now choose $\delta_k$ ($0 \le k \le 4$) according to the rule $100 \delta_k C_{k+1} \le \frac{1}{100}$. Subsequently taking the linear combination of $\eqref{pin:2}_k + \eqref{pin:3}_k + 100 \delta_k \eqref{pin:1}_{k+1}$ results in closing the bound for the $k+\frac12$ level and $k+1$ level in terms of level $k$ and below (for some new constants $A_k$):
\begin{align}
&\frac{\p_x}{2} \mathcal{E}_{k+\frac12,9-k}^{(Y)}(x) + \mathcal{D}^{(Y)}_{k+\frac12,9-k}(x) \le   A_k  \widehat{I}_{k}+ A_k \eps^{\frac12} \widehat{\mathcal{I}}_5(x), \\ \n
&\frac{\p_x}{2} \mathcal{E}_{k+\frac12,9-k}^{(Z)}(x)+  \mathcal{B}^{(Z)}_{k+\frac12}(x)  + \mathcal{D}^{(Z)}_{k+\frac12,9-k}(x) \le A_k \widehat{I}_{k}+ A_k \eps^{\frac12} \widehat{\mathcal{I}}_5(x), \\
&\frac{\p_x}{2} \mathcal{E}_{k+1,10-(k+1)}(x) + \mathcal{CK}_{k+1,10-(k+1)}(x)  + \mathcal{CK}^{(P)}_{k+1,10-(k+1)}(x)  +  \mathcal{B}_{k+1}(x) + \mathcal{D}_{k+1,10-(k+1)}(x) \\
& \qquad \qquad \le   C_k \widehat{\mathcal{I}}_{(k-1) + \frac12}(x)  + C_k \eps^{\frac12} \widehat{\mathcal{I}}_5(x).
\end{align}
From here, it is clear there exists a choice of (decreasing) weights $1 = \sigma_0 < \sigma_{0 + \frac12}^{(Y)} = \sigma_{0 + \frac12}^{(Z)} < \sigma_1 < \dots$ depending only on the constants $A_k, C_k$ appearing in the estimates above such that the linear combination defined by \eqref{mellow:1} -- \eqref{mellow:4} satisfies
\begin{align} \label{qwayseye:1}
\frac{\p_x}{2} \mathcal{E}_{\text{Quasi}} +\frac{1}{2} \mathcal{D}_{\text{Quasi}} +\frac{1}{2} \mathcal{CK}_{\text{Quasi}}  +  \mathcal{CK}^{(P)}_{\text{Quasi}} + \frac{1}{2} \mathcal{B}_{\text{Quasi}} \le  \eps^{\frac14} \widehat{I}_5(x).
\end{align} 
\vspace{2 mm}

\noindent \underline{Step 3: Closing the Bootstrap} Inequality \eqref{qwayseye:1} is valid as long as the bootstrap hypothesis \eqref{boots:1} is valid. We integrate this inequality to produce 
\begin{align}
\sup_{0 \le x \le X_\ast} \mathcal{E}_{\text{Quasi}} + \int_0^{X_\ast} \frac12 \mathcal{D}_{\text{Quasi}} +\frac12 \mathcal{CK}_{\text{Quasi}}+ \mathcal{CK}^{(P)}_{\text{Quasi}} +\frac12 \mathcal{B}_{\text{Quasi}} \le \mathcal{E}_{\text{Quasi}}(0) +  \int_0^{X_\ast}  \eps^{\frac14} \widehat{I}_5(x).
\end{align}
First of all, we notice that due to the equivalence we have proven in \eqref{norm:equi:1} -- \eqref{norm:equi:3}, we obtain 
\begin{align}
\sup_{0 \le x \le X_\ast} \mathcal{E}_{\text{Quasi}} + \int_0^{X_\ast} \frac12\mathcal{D}+ \frac12\mathcal{CK} + \mathcal{CK}^{(P)}+ \frac12\mathcal{B} \le 2 \mathcal{E}_{\text{Quasi}}(0) + 2 \int_0^{X_\ast}  \eps^{\frac14} \widehat{I}_5(x).
\end{align}
Recalling the definition of $\widehat{I}_5(x)$, \eqref{defn:hat:IK}, we notice that the final term on the right-hand side can be absorbed to the left, which gives 
\begin{align}
\sup_{0 \le x \le X_\ast} \mathcal{E}_{\text{Quasi}} + \int_0^{X_\ast} \frac12\mathcal{D}+\frac12 \mathcal{CK}+ \mathcal{CK}^{(P)}+\frac12 \mathcal{B} \le \frac{2}{1- \eps^{1/5}} \mathcal{E}_{\text{Quasi}}(0) 
\end{align}
To conclude the argument, we notice trivially that $\mathcal{E}_{\text{Quasi}}(0) \le 2 \mathcal{E}_{Init}$. Therefore, we have 
\begin{align}
\sup_{0 \le x \le X_\ast} \mathcal{E}_{\text{Quasi}} + \int_0^{X_\ast} \frac12\mathcal{D}+\frac12 \mathcal{CK}+ \mathcal{CK}^{(P)}+\frac12 \mathcal{B} \le \frac{4}{1- \eps^{1/5}} \mathcal{E}_{Init}, 
\end{align}
which improves the bootstrap \eqref{boots:1}. 

\vspace{2 mm}

\noindent \underline{Step 4: Improved Decay Rates:} At this point, we notice the following identities: 
\begin{align}
\p_x \{ x^{3m} \mathcal{E}_{k,n} \} \le & x^{3m} (\mathcal{E}_{k,n} + \mathcal{CK}^{(P)}_{k,n}  ), \\
\p_x \{ x^{3m} \mathcal{E}^{(Y)}_{k + \frac12 ,n} \} \le & x^{3m} (  \mathcal{E}^{(Y)}_{k + \frac12,n} + C_m \mathcal{D}_{k,n} ), \\
\p_x \{ x^{3m} \mathcal{E}^{(Z)}_{k + \frac12,n} \} \le & x^{3m} (  \mathcal{E}^{(Z)}_{k + \frac12,n} + C_m \mathcal{D}_{k,n} ).
\end{align}
We first absorb the right-hand side of \eqref{qwayseye:1} into the dissipation terms-CK as follows
\begin{align} \label{qwayseye41}
\frac{\p_x}{2} \mathcal{E}_{\text{Quasi}} +\frac{1}{4} \mathcal{D}_{\text{Quasi}} +\frac{1}{4} \mathcal{CK}_{\text{Quasi}}  +  \mathcal{CK}^{(P)}_{\text{Quasi}} + \frac{1}{4} \mathcal{B}_{\text{Quasi}} \le 0.
\end{align} 
Therefore, multiplying both sides by $x^{3m}$ we obtain 
\begin{align} \label{qwayseye:2}
\frac{\p_x}{2} (x^{3m} \widetilde{\mathcal{E}}_{\text{Quasi}}) +\frac{1}{4} (x^{3m} \widetilde{\mathcal{D}}_{\text{Quasi}} )+\frac{1}{4} (x^{3m} \widetilde{\mathcal{CK}}_{\text{Quasi}}) + \frac{1}{4}(x^{3m} \widetilde{\mathcal{B}}_{\text{Quasi}} )\le 0,
\end{align} 
where the functionals $\widetilde{\mathcal{E}}_{\text{Quasi}}$, etc... are obtained by potentially changing the constants $\sigma_{k}$ that appear in the linear combination. From here, we find that 
\begin{align}
\sup_{0 \le x \le X_\ast} (x^{3m} \mathcal{E}) + \int_0^{X_\ast} x^{3m}( \mathcal{D}+ \mathcal{CK}+ \mathcal{B}) \lesssim C_m \mathcal{E}_{Init}.
\end{align}
For the final step, we note that to estimate the initial datum, we use the definition of $\mathcal{T}_k$, \eqref{defTkdata}, as follows: 
\begin{align*}
\mathcal{E}_{init} \lesssim & \sum_{k = 0}^5 \| \bar{u} U_{IN; k} \langle y \rangle^{10-k} \|_{L^2_y}^2 + \sum_{k = 0}^4 (\| \sqrt{\bar{u}} \p_y U_{IN; k} \langle y \rangle^{9-k} \|_{L^2_y}^2 + \| \bar{u} \p_y U_{IN; k} \langle y \rangle^{9-k} \|_{L^2_y}^2) \\
\lesssim & \sum_{k = 0}^5 \| \bar{u} \mathcal{T}_k[ U_{IN}] \langle y \rangle^{10-k} \|_{L^2_y}^2 + \sum_{k = 0}^4 (\| \sqrt{\bar{u}} \p_y \mathcal{T}_k[ U_{IN}] \langle y \rangle^{9-k} \|_{L^2_y}^2 + \| \bar{u} \p_y \mathcal{T}_k[ U_{IN}] \langle y \rangle^{9-k} \|_{L^2_y}^2) \\
\lesssim & \| u_{IN} \langle y \rangle^{10} \|_{H^{15}_y}.
 \end{align*}
This proves the decay rates in the main theorem upon first reproving a variant of the Sobolev embedding \eqref{greer:1} upon multiplying the left-hand side by $x^{3m}$, and subsequently multiplying the inequalities \eqref{G:League:na:1} -- \eqref{G:League:na:4} also by $x^{3m}$ (which are $x$ by $x$). Note that we have proven the main result when the top order derivative is $k_0 = 5$ and hence $0 \le k \le 3$ in the main theorem and with weight $\langle \eta \rangle^3$ according to \eqref{jurassic:1}. However, once the nonlinear iteration has been closed we may repeat these estimates to recover the sharp weight in $\langle \eta \rangle$ as well as obtain higher derivative bounds as long as the compatibility conditions are satisfied.  
\end{proof}

\section{Sharp Decay Estimates \& Proof of Theorem \ref{thm:2}} \label{sec:8}

\subsection{Weighted Functionals}

In this section we close estimates on the following weighted versions of our energy-CK-dissipation functionals: 
\begin{align}
\widehat{\mathcal{E}}_{k,n}(x) := &\int_{\mathbb{R}_+} \bar{u}^2 U_k^2 x^{2k - \frac{1}{100}} \langle \eta \rangle^{2n} \langle \psi \rangle \ud y, \\
\widehat{\mathcal{CK}}_{k,n}(x) := &\frac{1}{100} \int_{\mathbb{R}_+} \bar{u}^2 U_k^2 x^{2k -1 - \frac{1}{100}} \langle \eta \rangle^{2n} \langle \psi \rangle \ud y, \\
\widehat{\mathcal{CK}}_{k,n}^{(P)}(x) := & \int_{\mathbb{R}_+} (-\frac{3}{2}\p_x p_E(x))U_k^2 x^{2k- \frac{1}{100}} \langle \eta \rangle^{2n} \langle \psi \rangle  \ud y  \\
\widehat{\mathcal{D}}_{k,n}(x) := & \int_{\mathbb{R}_+} \bar{u} |\p_y U_k|^2 x^{2k- \frac{1}{100}} \langle \eta \rangle^{2n}  \langle \psi \rangle \ud y, \\
\mathcal{B}_{k}(x) := &  \bar{u}_y |U_k(x, 0)|^2 x^{2k- \frac{1}{100}}.
\end{align}

We also need the following ``half-level" energy-CK-dissipation functionals: 
\begin{align}
\widehat{\mathcal{E}}^{(Y)}_{k+\frac12,n}(x) := & \int_{\mathbb{R}_+} \bar{u} |\p_y U_k|^2 x^{1 - \frac{1}{100}+ 2k} \langle \eta \rangle^{2n} \langle \psi \rangle \ud y, \\
\widehat{\mathcal{D}}^{(Y)}_{k + \frac12,n}(x) := & \int_{\mathbb{R}_+} \bar{u}^2 |\p_x U_k|^2 x^{1 - \frac{1}{100}+ 2k} \langle \eta \rangle^{2n} \langle \psi \rangle \ud y,
\end{align} 
and 
\begin{align}
\widehat{\mathcal{E}}^{(Z)}_{k+\frac12,n}(x) := & \int_{\mathbb{R}_+} \bar{u}^2 |\p_y U_k|^2 x^{2k + 1 - m- \frac{1}{100}}\langle \eta \rangle^{2n}\langle \psi \rangle \ud y, \\
\widehat{\mathcal{D}}^{(Z)}_{k + \frac12,n}(x) := &\int_{\mathbb{R}_+} \bar{u} |\p_y^2 U_k|^2 x^{1 + 2k-m- \frac{1}{100}} \langle \eta \rangle^{2n} \langle \psi \rangle \ud y, \\
\mathcal{B}^{(Z)}_{k + \frac12}(x) := &  \bar{u}_y |\p_y U_k(x, 0)|^2 x^{2k + 1 -m- \frac{1}{100}}
\end{align} 
Our total energy functional will be as follows: 
\begin{align} 
\widehat{\mathcal{E}}(x) := & \sum_{k = 0}^4  \sigma_{k} \widehat{\mathcal{E}}_{k,10-k}(x) + \sum_{k = 0}^3  (\sigma_{k + \frac12}^{(Y)} \widehat{\mathcal{E}}^{(Y)}_{k+ \frac12,9-k}(x)  + \sigma_{k + \frac12}^{(Z)} \widehat{\mathcal{E}}^{(Z)}_{k+ \frac12,9-k}(x)), \\ 
\widehat{\mathcal{D}}(x) := &\sum_{k = 0}^4 \sigma_{k} \widehat{\mathcal{D}}_{k,10-k}(x) + \sum_{k = 0}^3  (\sigma_{k + \frac12}^{(Y)}  \widehat{\mathcal{D}}^{(Y)}_{k+ \frac12,9-k}(x) + \sigma_{k + \frac12}^{(Z)}\widehat{ \mathcal{D}}^{(Z)}_{k+ \frac12,9-k}(x) ), \\
\widehat{\mathcal{CK}}(x) := &  \sum_{k = 0}^4 \sigma_{k} \widehat{\mathcal{CK}}_{k,10-k}(x) , \\
\widehat{\mathcal{CK}}^{(P)}(x) := &  \sum_{k = 0}^4 \sigma_{k} \widehat{\mathcal{CK}}^{(P)}_{k,10-k}(x), \\
\mathcal{B}(x) := &\sum_{k = 0}^4 \sigma_{k} \mathcal{B}_{k}(x) + \sum_{k = 0}^3 \sigma_{k + \frac12}^{(Z)} \mathcal{B}^{(Z)}_{k+ \frac12}(x).
\end{align}
We need a weighted analogue of the $\widehat{\mathcal{I}}_{k}(x), \widehat{\mathcal{I}}_{k + \frac12}(x)$:  
\begin{align} \n
\widehat{\mathcal{J}}_{ k}(x) := & \sum_{k' = 0}^k  (  \widehat{\mathcal{D}}_{k,10-k}(x) +  \widehat{ \mathcal{CK}}_{k,10-k}(x) +   \mathcal{B}_{k}(x)   )+  \bold{1}_{k \ge 1} \sum_{k' = 0}^{k-1}  ( \widehat{\mathcal{D}}^{(Y)}_{k + \frac12,9-k}(x) + \widehat{\mathcal{D}}^{(Z)}_{k + \frac12,9-k}(x) \\ \label{defn:hat:IK}
& + \mathcal{B}^{(Z)}_{k + \frac12}(x)  ) , \\ \n
\widehat{\mathcal{J}}_{k + \frac12}(x) := & \sum_{k' = 0}^k  (  \widehat{\mathcal{D}}_{k,10-k}(x) +  \widehat{ \mathcal{CK}}_{k,10-k}(x) +   \mathcal{B}_{k}(x)   )+   \sum_{k' = 0}^{k}  ( \widehat{\mathcal{D}}^{(Y)}_{k + \frac12,9-k}(x) + \widehat{\mathcal{D}}^{(Z)}_{k + \frac12,9-k}(x) \\
&+ \mathcal{B}^{(Z)}_{k + \frac12}(x)  ).
\end{align}

\subsection{Virial Estimates}

\begin{lemma} For any $k, n \ge 0$, the following energy inequality is valid: 
\begin{align} \n
&\frac{\p_x}{2} \widehat{\mathcal{E}}_{k,n}(x) + \widehat{\mathcal{CK}}_{k,n}(x)  + \widehat{\mathcal{CK}}^{(P)}_{k,n}(x) +  \mathcal{B}_{k}(x) + \widehat{\mathcal{D}}_{k,n}(x) \\ \label{cal:1:hat}
& \qquad \lesssim  \delta_{n > 0} \widehat{\mathcal{CK}}_{k,n-1}(x) +  \delta_{k > 0} \widehat{\mathcal{CK}}_{k-1,n}(x) +C |\widehat{\bold{S}}_{k,n}(x)|, 
\end{align}
where 
\begin{align} \label{Skn:def:1:hat}
\widehat{\bold{S}}_{k,n}(x) := \int_{\mathbb{R}_+} G_k U_k x^{2k - \frac{1}{100}} \langle \eta \rangle^{2n} \langle \psi \rangle \ud y. 
\end{align}

\end{lemma}
\begin{proof} For simplicity of the resulting expressions, we prove it in the case of $n = 0, k = 0$, with the inclusion of general $(n, k)$ following as in Lemma \ref{X:est:1}. Essentially, we focus on the commutator terms arising from the inclusion of the weight $\psi$. We record here the following computations: 
\vspace{2 mm}

\noindent \textit{Transport Terms:} We have 
\begin{align} \n
\bold{T}(x)  = &\int_{\mathbb{R}_+} (\bar{u}^2 \p_x U + \bar{u} \bar{v} \p_y U + \bar{u}_{yyy}Q + 2 (\bar{u}_{yy} - \p_x p_E(x))U) U \psi  \ud y \\ \n
= & \bold{T}^{(Main)}(x) + \bold{T}^{(Weight)}(x).   
\end{align}
We define above the weighted version of \eqref{as:in}:
\begin{align} \n
\bold{T}^{(Main)}(x) := & \frac{\p_x}{2} \int_{\mathbb{R}_+} \bar{u}^2 U^2 x^{- \frac{1}{100}}  \psi \ud y- \int_{\mathbb{R}_+} \bar{u} \bar{u}_x |U|^2 x^{ - \frac{1}{100}}  \psi \ud y \\ \n
&- \frac12 \int_{\mathbb{R}_+} \p_y \{ \bar{u} \bar{v} \} |U|^2 x^{- \frac{1}{100}}  \psi \ud y- \frac12 \int_{\mathbb{R}_+} \bar{u}_{yyyy} |Q|^2  x^{ - \frac{1}{100}}  \psi \ud y \\
&+  \int_{\mathbb{R}_+} 2(\bar{u}_{yy} - \p_x p_E(x)) |U|^2 x^{ - \frac{1}{100}}  \psi \ud y + \frac{1}{100} \int_{\mathbb{R}_+} \bar{u}^2 |U|^2 x^{ - 1- \frac{1}{100}} \psi \ud y, 
\end{align}
and the terms which arise as commutators due to the $\psi$ weight: 
\begin{align} \n
\bold{T}^{(Weight)}(x) := & - \frac{1}{2} \int_{\mathbb{R}_+} \bar{u}^2 U^2 x^{- \frac{1}{100}}  \p_x \psi \ud y - \frac{1}{2} \int_{\mathbb{R}_+} \bar{u} \bar{v} U^2 x^{-\frac{1}{100}}  \p_y \psi \ud y  - \frac{1}{2} \int_{\mathbb{R}_+} \bar{u}_{yyy} Q^2  \p_y \psi \ud y \\ \n
=& - \frac{1}{2} \int_{\mathbb{R}_+} \bar{u}^2 U^2 x^{- \frac{1}{100}} (-\bar{v}) - \frac{1}{2} \int_{\mathbb{R}_+} \bar{u} \bar{v} U^2 x^{-\frac{1}{100}}  \bar{u} \ud y  - \frac{1}{2} \int_{\mathbb{R}_+} \bar{u} \bar{u}_{yyy} Q^2  \ud y  \\ \label{hghg}
= & - \frac{1}{2} \int_{\mathbb{R}_+} \bar{u} \bar{u}_{yyy} Q^2  \ud y.
\end{align}
Above, we have used the special cancellation of the transport terms $\bar{u}^2 \p_x + \bar{u} \bar{v} \p_y$ which arises due to the choice of $\psi$ weight. It will turn out the remaining term above will be cancelled by a diffusive commutator term. 

\vspace{2 mm}

\noindent \textit{Diffusive Terms:} We have 
\begin{align*}
 \bold{D}^{(Main)}(x) := & \int_{\mathbb{R}_+} \bar{u} |\p_y U|^2 x^{-\frac{1}{100}} \psi \ud y  +\bar{u}_y |U|^2 x^{- \frac{1}{100}}|_{y = 0} - \int_{\mathbb{R}_+} 2 \bar{u}_{yy} |U|^2 \psi \ud y \\
 &+\frac12 \int_{\mathbb{R}_+} \bar{u}_{yyyy} |Q|^2  x^{- \frac{1}{100}} \psi \ud y,
\end{align*}
and 
\begin{align*}
 \bold{D}^{(Weight)}(x) := & -  \int_{\mathbb{R}_+} \bar{u} \bar{u}_y U^2  \ud y + \frac{1}{2} \int_{\mathbb{R}_+} \bar{u} \bar{u}_{yyy} Q^2  \ud y - \int_{\mathbb{R}_+} \bar{u} \bar{u}_{yy} Q U  \ud y \\
 &+ 2(1-\delta) \int_{\mathbb{R}_+} \bar{u} \bar{u}_y U^2  \ud y   -  \int_{\mathbb{R}_+} \bar{u} \bar{u}_y U^2 \ud y   + \int \bar{u} \bar{u}_{yy} Q U \psi^{-\delta} \ud y \\
 = &   \frac{1}{2} \int_{\mathbb{R}_+} \bar{u} \bar{u}_{yyy} Q^2  \ud y.
\end{align*}
\end{proof}

\begin{lemma} \label{tioga:2}
For any $0 < \delta << 1$, the following energy inequality is valid: 
\begin{align} \n
\frac{\p_x}{2} \widehat{\mathcal{E}}_{k+\frac12,n}^{(Y)}(x) + \widehat{\mathcal{D}}^{(Y)}_{k+\frac12,n}(x) \le &  C_\delta  \widehat{\mathcal{D}}_{k, n+1}(x) + \delta (  \widehat{\mathcal{D}}_{k+1, 0}(x) + \widehat{\mathcal{CK}}_{k+1,0}(x)   + \widehat{\mathcal{D}}^{(Z)}_{k + \frac12, 0}(x)) \\ \label{cal:2:vir}
&+C |\widehat{\bold{S}}^{(Y)}_{k,n}(x)|,
\end{align}
where the source term is defined as 
\begin{align} \label{def:skY:vir}
\widehat{\bold{S}}^{(Y)}_{k,n}(x) :=  \int_{\mathbb{R}_+} G_k\p_x U_k \langle \eta \rangle^{2n} x^{2k+1 - \frac{1}{100}} \langle \psi \rangle \ud y.
\end{align}
\end{lemma}
\begin{proof} This follows upon applying the multiplier $\p_x U_k x^{2k + 1 - \frac{1}{100}} \langle \eta \rangle^{2n} \psi$, and following the proof of Lemma \ref{lem:trav:1}. For simplicity, we carry out the calculations for the case $k = n = 0$, with the adaptation to general $(k, n)$ as in Lemma \ref{lem:trav:1}. Moreover, we do not repeat all the details; instead we remark that the only difference is the treatment of commutator terms arising from the weight $\psi$. These arise from the diffusive term, namely 
\begin{align*}
- \int_{\mathbb{R}_+} \bar{u} U_{yy} U_x x^{1- \frac{1}{100}} \psi \ud y =& \int_{\mathbb{R}_+} \bar{u}^2 U_y U_x x^{1-\frac{1}{100}}  \ud y + \int_{\mathbb{R}_+} \bar{u}_y U_{y} U_x x^{1- \frac{1}{100}} \psi \ud y \\
&+ \int_{\mathbb{R}_+} \bar{u} U_{y} U_{xy} x^{1- \frac{1}{100}} \psi \ud y \\ \n
= & \int_{\mathbb{R}_+} \bar{u}^2 U_y U_x x^{1-\frac{1}{100}} \ud y + \int_{\mathbb{R}_+} \bar{u}_y U_{y} U_x x^{1- \frac{1}{100}} \psi \ud y \\
&+ \frac{\p_x}{2} \int_{\mathbb{R}_+} \bar{u} U_{y}^2 x^{1- \frac{1}{100}} \psi \ud y - \frac12 \int_{\mathbb{R}_+} \bar{u}_x U_{y}^2 x^{1- \frac{1}{100}} \psi \ud y \\
& - \frac{1 - \frac{1}{100}}{2} \int_{\mathbb{R}_+} \bar{u} U_{y}^2 x^{- \frac{1}{100}} \psi \ud y + \frac{1}{2} \int_{\mathbb{R}_+} \bar{u} \bar{v} U_y^2 x^{1- \frac{1}{100}}.
\end{align*}  
There are two commutators due to $\psi$ appearing above, namely the first and last term on the right-hand side, which we treat as follows. First,  
\begin{align*}
| \int_{\mathbb{R}_+} \bar{u}^2 U_y U_x x^{1-\frac{1}{100}}  \ud y| = & | \int_{\mathbb{R}_+} \bar{u}^2 U_y U_x x^{1-\frac{1}{100}} \frac{1}{\psi} \psi \ud y| \\
\lesssim & \Big\| \frac{\bar{u}}{\psi} x^{\frac12(1-m)} \Big\|_{L^\infty} \| U_y x^{\frac{m}{2}} \psi^{\frac12} \|_{L^2_y} \| \bar{u} U_x x^{\frac12 - \frac{1}{200}} \psi^{\frac12 } \|_{L^2_y}
\end{align*}
Second, we have 
\begin{align*}
|\int_{\mathbb{R}_+} \bar{u} \bar{v} U_y^2 x^{1 - \frac{1}{100}} \ud y| \lesssim \Big\| \frac{\bar{v}}{\langle \eta \rangle} x^{\frac12 - \frac{m}{2}} \Big\|_{L^\infty} \Big\| \frac{\bar{u}}{\psi} x^{\frac12(1-m)} \Big\|_{L^\infty} \| U_y x^{\frac{m}{2}} \langle \eta \rangle^{\frac12} x^{- \frac{1}{200}} \psi^{\frac12}\|_{L^2_y}^2
\end{align*}
\end{proof}

\begin{lemma}For any $0 < \delta << 1$, the following energy inequality is valid:
\begin{align} \n
\frac{\p_x}{2} \widehat{\mathcal{E}}_{k+\frac12,n}^{(Z)}(x)+  \mathcal{B}^{(Z)}_{k+\frac12}(x)  + \widehat{\mathcal{D}}^{(Z)}_{k+\frac12,n}(x) \le& C_{\delta} ( \widehat{\mathcal{D}}_{k,n+1}(x) + \widehat{ \mathcal{CK}}_{k,n}(x) + \mathcal{B}_k(x) )+  \delta \widehat{\mathcal{D}}^{(Y)}_{k + \frac12,0}(x) \\ \label{cal:3:vir}
& + C |\widehat{\bold{S}}^{(Z)}_{k,n}(x)|,
\end{align}
where the source term 
\begin{align} \label{szA:vir}
\widehat{\bold{S}}^{(Z)}_{k,n}(x) := &  \int_{\mathbb{R}_+} \p_y G_k \p_y U_k \langle \eta \rangle^{2n} x^{2k + 1 - m - \frac{1}{100}} \langle \psi \rangle \ud y.
\end{align}
\end{lemma}
\begin{proof} This follows upon repeating the arguments in Lemma \ref{lem:Tioga:1} with the multiplier (again in vorticity form) $\p_y U_k x^{2k+1 - m - \frac{1}{100}} \langle \eta \rangle^{2n} \psi$. The commutator terms arising from the inclusion of the weight $\psi$ are of an identical form to Lemma \ref{tioga:2}. We omit repeating the details. 
\end{proof}

We have the analogue of Corollary \ref{wehave}:
\begin{lemma}\label{wehave:2}The following bounds are valid on $F_{\text{Comm},k}$ for $0 \le k \le 3$:
\begin{align} \label{tom:petty:278}
\| \frac{1}{\sqrt{\bar{u}}} F_{\text{Comm},k} \langle \eta \rangle^{10-k} x^{k + \frac12 - \frac{m}{2}- \frac{1}{200}} \langle \psi \rangle^{\frac12} \|_{L^2_y}^2 \le & C_\lambda \widehat{\mathcal{J}}_{(k-1) + \frac12}(x) + \lambda \widehat{\mathcal{D}}_{k,0}(x), \\ \label{tp:783}
\| \frac{1}{\bar{u}} \p_x^k \mathcal{Q}[u, v] \langle \eta \rangle^{10-k } x^{k + \frac12- \frac{1}{200}}\langle \psi \rangle^{\frac12} \|_{L^2_y}^2 \lesssim & \eps^{-1}  \widehat{\mathcal{J}}_3(x).
\end{align}
\end{lemma}

By repeating almost verbatim the arguments in Proposition \ref{prop:main:1}, we obtain: 
\begin{proposition} The following estimate is valid:  
\begin{align} \label{boots:2}
\sup_{0 \le x < \infty} \widehat{\mathcal{E}}(x) + \int_0^{\infty} (\widehat{\mathcal{D}}(x) + \widehat{\mathcal{CK}}(x) + \mathcal{B}(x) ) dx \le 5 \widehat{\mathcal{E}}_{\text{Init}}, 
\end{align} 
\end{proposition}

\subsection{Nash Type Inequality}

From here, we may conclude using a Nash-type argument. First, we establish
\begin{lemma} Define $q(\eta)$ to be a smooth function such that there are two finite, nonzero constants $c_0, C_0$ satisfying $c_0 \le \frac{q(\eta)}{x^{-m} \bar{u}} \le C_0$. Then the following lower bound is valid: 
\begin{align} \label{nash:type:1}
\int_{\mathbb{R}_+} \bar{u} |\p_y U|^2 \ud y \gtrsim \min \{ x^{- \frac14 + \frac{3m}{4}} \frac{\Big( \int_{\mathbb{R}_+} \bar{u}^2 U^2 \ud y \Big)^{\frac52}}{\Big( \int_{\mathbb{R}_+} \bar{u}^2 U^2 y q(\eta) x^m \ud y \Big)^{\frac32}}, x^m \frac{\Big( \int_{\mathbb{R}_+} \bar{u}^2 U^2 \ud y \Big)^3}{\Big( \int_{\mathbb{R}_+} \bar{u}^2 U^2 y q(\eta) x^m \ud y \Big)^2} \}.
\end{align}
\end{lemma}
\begin{proof} We define the following ($x$-dependent) quantities: 
\begin{align}
\gamma^2 := \int_{\mathbb{R}_+} \bar{u}^2 U^2 \ud y, \qquad A^2 := \int_{\mathbb{R}_+} \bar{u}^2 U^2 y q(\eta) x^m \ud y, \qquad B^2 = \int_{\mathbb{R}_+} \bar{u} |\p_y U|^2 \ud y. 
\end{align}
It is convenient to split the argument into two cases. 

\vspace{2 mm}

\noindent \underline{Case 1:  $\frac{A}{B} \le x^{\frac34 + \frac{m}{4}}$}. In this case, we define the parameter (again, $x$-dependent)
\begin{align}
\alpha = \frac{1}{x^{\frac{3}{10} + \frac{m}{10}}} \Big( \frac{A}{B} \Big)^{\frac25} \le 1. 
\end{align}
We subsequently use $\alpha$ as a threshold to cutoff: 
\begin{align}
\int_{\mathbb{R}_+} \bar{u}^2 U^2 \ud y \le \int_{\mathbb{R}_+} \bar{u}^2 U^2 \chi(\eta \le \alpha) \ud y + \int_{\mathbb{R}_+} \bar{u}^2 U^2 \chi(\eta > \alpha) \ud y = K_{near} + K_{Far}.
\end{align}
We estimate the far-field component using $A$ as follows: 
\begin{align*}
|K_{Far}| = \int_{\mathbb{R}_+} \bar{u}^2 U^2 \frac{y q(\eta) x^m}{y q(\eta) x^m} \chi(\eta > \alpha) \ud y \lesssim \frac{x^{-m}}{\alpha^2 x^{\frac12(1-m)}} A^2.
\end{align*}
Above, we have used the lower bound $q(\eta) \ge \eta \ge \alpha$ since $\alpha < 1$. For the near-field component, we integrate by parts as follows:
\begin{align*}
|K_{Near}| \le & \alpha^2 x^{2m} \int_{\mathbb{R}_+}  U^2 \chi(\eta \le \alpha) \ud y := \overline{K}_{Near}
\end{align*}
To estimate $\overline{K}_{Near}$, we have 
\begin{align} \n
\overline{K}_{Near} \le & \alpha^2 x^{2m} \int_{\mathbb{R}_+} \p_y \{ y \} U^2 \chi(\eta \le \alpha) \ud y \\ \n
= &  -2\alpha^2 x^{2m} \int_{\mathbb{R}_+} y U U_y \chi(\eta \le \alpha) \ud y  -\alpha^2 x^{2m} \int_{\mathbb{R}_+}  U^2 \frac{y}{\alpha x^{\frac12(1-m)}} \chi'(\eta \le \alpha) \ud y \\ \label{close:nash:1}
=& : K_{Near, Near} + K_{Near, Far}.
\end{align}
The $K_{Near, Far}$ can be treated as the far-field term using that on the support of $\chi'(\eta \le \alpha)$, we have $\eta \sim \alpha$ and therefore $\bar{u} \sim x^m \alpha$:
\begin{align} \label{close:nash:2}
|K_{Near, Far}| \lesssim \int \bar{u}^2 U^2 \chi'(\eta \le \alpha) \ud y \lesssim \frac{x^{-m}}{\alpha^2 x^{\frac12(1-m)}} A^2.
\end{align}
For $K_{Near, Near}$, we have by Young's inequality, 
\begin{align} \n
|K_{Near, Near}| \le & \frac12 \overline{K}_{Near} + C\alpha^2 x^{2m} \int y^2 U_y^2 \chi(\eta \le \alpha) \ud y \\ \n
\le & \frac12 \overline{K}_{Near} + C\alpha^2 x^{2m} x^{1-m} \int \eta^2 U_y^2 \chi(\eta \le \alpha) \ud y \\  \n
\le & \frac12 \overline{K}_{Near} + C\alpha^3 x^{2m} x^{1-m} \int x^{-m} \bar{u} U_y^2 \chi(\eta \le \alpha) \ud y \\ \label{close:nash:3}
\le & \frac12 \overline{K}_{Near} + C\alpha^3 x B^2. 
\end{align}
Inserting \eqref{close:nash:3}, \eqref{close:nash:2} into \eqref{close:nash:1} and using the $1/2$ factor to absorb the $\overline{K}_{Near}$ to the left-hand side closes the following estimate for $\overline{K}_{Near}$ (and therefore also the estimate for $K_{Near}$):
\begin{align*}
|\overline{K}_{Near}| \lesssim \frac{x^{-m}}{\alpha^2 x^{\frac12(1-m)}} A^2 + \alpha^3 x B^2.
\end{align*}
Consolidating the bounds and finally invoking our choice of the parameter $\alpha$, we have
\begin{align} \n
\gamma^2 \lesssim & \frac{x^{-m}}{\alpha^2 x^{\frac12(1-m)}} A^2 + \alpha^3 x B^2 \\ \n
\lesssim &   \frac{x^{-m}}{( \frac{1}{x^{\frac{3}{10} + \frac{m}{10}}} \Big( \frac{A}{B} \Big)^{\frac25})^2 x^{\frac12(1-m)}} A^2 + ( \frac{1}{x^{\frac{3}{10} + \frac{m}{10}}} \Big( \frac{A}{B} \Big)^{\frac25})^3 x B^2  \\ \label{case:1:bd}
\lesssim & x^{\frac{1}{10} - \frac{3m}{10}} A^{\frac65} B^{\frac45}.
\end{align}
This concludes the treatment of Case 1. 

\vspace{2 mm}

\noindent \underline{Case 2:  $\frac{A}{B} > x^{\frac34 + \frac{m}{4}}$}. In this case, we define the parameter (again, $x$-dependent)
\begin{align}
\alpha = \frac{1}{x^{\frac12 + \frac{m}{6}}} \Big( \frac{A}{B} \Big)^{\frac23}> 1. 
\end{align}
We again use $\alpha$ as a threshold to cutoff: 
\begin{align}
\int_{\mathbb{R}_+} \bar{u}^2 U^2 \ud y \le \int_{\mathbb{R}_+} \bar{u}^2 U^2 \chi(\eta \le \alpha) \ud y + \int_{\mathbb{R}_+} \bar{u}^2 U^2 \chi(\eta > \alpha) \ud y = J_{near} + J_{Far}.
\end{align}
We begin with the estimate of $J_{Far}$. Here, since $\alpha > 1$, the weight function $q(\eta)$ is bounded below. Therefore, we have 
\begin{align*}
|J_{Far}| = \int_{\mathbb{R}_+} \bar{u}^2 U^2 \frac{y q(\eta) x^m}{y q(\eta) x^m} \chi(\eta > \alpha) \ud y \lesssim \frac{x^{-m}}{\alpha x^{\frac12(1-m)}} A^2.
\end{align*}
 For the near-field component, we have to use a cheaper bound on $\bar{u}$, namely 
\begin{align*}
|J_{Near}| \le &  x^{2m} \int_{\mathbb{R}_+}  U^2 \chi(\eta \le \alpha) \ud y := \overline{J}_{Near}.
\end{align*}
To estimate $\overline{J}_{Near}$, we again have by integration by parts 
\begin{align} \n
|\overline{J}_{Near}| =& x^{2m} \int_{\mathbb{R}_+} \p_y \{ y \} U^2 \chi(\eta \le \alpha) \ud y \\
= & - 2 x^{2m} \int_{\mathbb{R}_+}  y U U_y \chi(\eta \le \alpha) \ud y-  x^{2m} \int_{\mathbb{R}_+}   U^2 \frac{y}{\alpha x^{\frac12(1-m)}} \chi'(\eta \le \alpha) \ud y \\ \label{pbf:3}
=& \overline{J}_{Near, Near} + \overline{J}_{Near, Far}.
\end{align}
For the $\overline{J}_{Near, Far}$, we use that on the support of $\chi'(\eta \le \alpha)$, we have $\eta \sim \alpha > 1$, and therefore $x^{2m} \sim \bar{u}^2$, after which we estimate in much the same way as $J_{Far}$: 
\begin{align} \label{pbf:2}
|\overline{J}_{Near, Far}| \lesssim \frac{x^{-m}}{\alpha x^{\frac12(1-m)}} A^2.
\end{align}
For the $\overline{J}_{Near, Near}$ term, we have by Cauchy-Schwartz and Young's
\begin{align} \n
|\overline{J}_{Near, Near}| \le & \frac12 |\overline{J}_{Near}| + Cx^{2m} \int_{\mathbb{R}_+} y^2 U_y^2 \chi(\eta \le \alpha) \\ \n
\le &  \frac12 |\overline{J}_{Near}| + Cx^{2m} x^{1-m} \int_{\mathbb{R}_+} \eta^2 U_y^2 \chi(\eta \le \alpha) \\ \n
\le &  \frac12 |\overline{J}_{Near}| + Cx^{2m} x^{1-m} \alpha^2 \int_{\mathbb{R}_+} q(\eta) U_y^2 \chi(\eta \le \alpha) \\ \n
\le &  \frac12 |\overline{J}_{Near}| + Cx^{2m} x^{1-m}x^{-m} \alpha^2 \int_{\mathbb{R}_+} (x^{m} q(\eta)) U_y^2 \chi(\eta \le \alpha) \\ \label{pbf:1}
\le &  \frac12 |\overline{J}_{Near}| + C x \alpha^2 B^2. 
\end{align}
We insert \eqref{pbf:1}, \eqref{pbf:2} into \eqref{pbf:3} to close the estimate on $\overline{J}_{Near}$ (and therefore on $J_{Near}$ itself) as follows: 
\begin{align*}
\overline{J}_{Near} \lesssim \frac{x^{-m}}{\alpha x^{\frac12(1-m)}} A^2 + x \alpha^2 B^2
\end{align*}
Consolidating the bounds on $J_{Near}$ and $J_{Far}$ and using our choice of $\alpha$, we have 
\begin{align} \n
\gamma^2 \lesssim & \frac{x^{-m}}{\alpha x^{\frac12(1-m)}} A^2 + x \alpha^2 B^2 \\ \n
\lesssim & \frac{x^{-m}}{(\frac{1}{x^{\frac12 + \frac{m}{6}}} \Big( \frac{A}{B} \Big)^{\frac23}) x^{\frac12(1-m)}} A^2 + x (\frac{1}{x^{\frac12 + \frac{m}{6}}} \Big( \frac{A}{B} \Big)^{\frac23})^2 B^2 \\ \label{case:2:bd}
\lesssim & x^{-\frac{m}{3}} A^{\frac43} B^{\frac23}. 
\end{align}
We are now done with treating Case 2. 

Combining the two cases, \eqref{case:1:bd} and \eqref{case:2:bd}, we have 
\begin{align}
\gamma^2 \lesssim \max\{  x^{\frac{1}{10} - \frac{3m}{10}} A^{\frac65} B^{\frac45}  , x^{-\frac{m}{3}} A^{\frac43} B^{\frac23}\}, 
\end{align}
and 
\begin{align}
B^2 \gtrsim \min \{ x^{- \frac14 + \frac{3m}{4}} \frac{\gamma^5}{A^3}, x^m \frac{\gamma^6}{A^4} \}.
\end{align}
\end{proof}

\subsection{ODE Arguments}

\begin{lemma} \label{lem:ODE} Suppose $f(\Gamma)$ satisfies 
\begin{align}
f(\Gamma) \ge c_\ast \min\{ x^{-\frac14 + \frac{3m}{4}} \Gamma^{\frac52}, x^m \Gamma^3 \} 
\end{align}
for any $c_\ast > 0$. Let $\Gamma(x) \ge 0$ satisfy 
\begin{align}
\frac{\p_x}{2} \Gamma + f(\Gamma) \le 0, \qquad \Gamma(0) = \Gamma_0.
\end{align}
Then 
\begin{align}
\Gamma \lesssim x^{-\frac12 - \frac{m}{2}}.
\end{align}
\end{lemma}
\begin{proof} First, we claim there exists constants $c_{\ast \ast} > 0$ and $\Gamma_{\ast \ast} > 0$ so that the trajectory defined by 
\begin{align} \label{gamma:up}
\frac{\p_x}{2} \Gamma_{\text{Up}} + c_{\ast \ast} x^{-\frac14 + \frac{3m}{4}} \Gamma_{\text{Up}}^{\frac52} = 0, \qquad \Gamma_{\text{Up}}(0) =  \Gamma_{\ast \ast},
\end{align}
satisfies 
\begin{align}
\Gamma(x) \le \Gamma_{\text{Up}}(x), \qquad x \ge 1. 
\end{align}
The explicit solution to \eqref{gamma:up} is of the form 
\begin{align} \label{expl:1}
\Gamma_{\text{Up}}(x) = \frac{1}{[c_{\ast \ast}( x^{\frac34(1 + m)} - 1) + \Gamma_{\ast \ast}^{-\frac32}]^{\frac23}}.
\end{align}
We set the two terms in the min equal to each other to define the curve $\lambda(x) := x^{-\frac12 - \frac{m}{2}}$. Next, we consider the trajectory 
\begin{align*}
\frac{\p_x}{2}\Lambda + c_{\ast} x^{-\frac14 + \frac{3m}{4}} \Lambda^{\frac52} = 0, \qquad \Lambda(0) =  \Gamma_{0},
\end{align*}
By choosing $c_{\ast \ast}, \Gamma_{\ast \ast}$ appropriately, we ensure the following inequalities: 
\begin{align}
\Gamma_{\text{Up}}(x) > \Lambda(x), \qquad \Gamma_{\text{Up}}(x) > \lambda(x),
\end{align}
and we also note that from the explicit solution, \eqref{expl:1}, the following decay estimate is valid:  
\begin{align*}
\Gamma_{\text{Up}}(x) \lesssim x^{-\frac12 - \frac{m}{2}}.
\end{align*}
Now, consider the trajectory $\Gamma(x)$. If $x$ is such that $\Gamma(x) \le \lambda(x)$, then we are automatically done. There are two ways in which $x$ could be such that $\Gamma(x) > \lambda(x)$. Either $x \in [1, x_1]$ or $x \in [x_n, x_{n+1}]$ for some $n \in \mathbb{N}$, where $\{x_n\}_{n = 1}^\infty$ are the crossings of $\Gamma$ with the curve $\lambda(x)$, in increasing order. Assume $x \in [1, x_1]$. Then for all $x \in [1, x_1)$, $\Gamma(x) > \lambda(x)$, and therefore satisfies the equation: 
\begin{align*}
\frac{\p_x}{2}\Gamma + c_{\ast} x^{-\frac14 + \frac{3m}{4}} \Gamma^{\frac52} \le 0.
\end{align*} 
By comparison principle, $\Gamma \le \Lambda < \Gamma_{\text{Up}}(x)$. Consider now the case when $x \in [x_n, x_{n+1}]$. Then $\Gamma(x_n) = \lambda(x_n)$. We now use comparison principle as follows:
\begin{align}
\frac{\p_x}{2}\Gamma + c_{\ast} x^{-\frac14 + \frac{3m}{4}} \Gamma^{\frac52} \le 0 = \frac{\p_x}{2} \Gamma_{\text{Up}} + c_{\ast \ast} x^{-\frac14 + \frac{3m}{4}} \Gamma_{\text{Up}}^{\frac52}, 
\end{align} 
and 
\begin{align}
 \qquad \Gamma(x_n) = \lambda(x_n) < \Gamma_{\text{Up}}(x_n). 
\end{align}
By comparison principle again, $\Gamma < \Gamma_{\text{Up}}(x)$ for $x \in [x_n, x_{n+1}]$. 
\end{proof}

\begin{proof}[Proof of Theorem \ref{thm:2}] The proceed follows in four steps.

\vspace{2 mm}

\noindent \textit{Step 1: Truncated Energy-Dissipation-CK Inequality:} We introduce the following functionals: 
\begin{align} 
\mathcal{E}(x) := & \sum_{k = 0}^3  \sigma_{k} \mathcal{E}_{k,10-k}(x) + \sum_{k = 0}^2  (\sigma_{k + \frac12}^{(Y)} \mathcal{E}^{(Y)}_{k+ \frac12,9-k}(x)  + \sigma_{k + \frac12}^{(Z)} \mathcal{E}^{(Z)}_{k+ \frac12,9-k}(x)), \\ 
\mathcal{D}(x) := &\sum_{k = 0}^3 \sigma_{k} \mathcal{D}_{k,10-k}(x) + \sum_{k = 0}^2  (\sigma_{k + \frac12}^{(Y)}  \mathcal{D}^{(Y)}_{k+ \frac12,9-k}(x) + \sigma_{k + \frac12}^{(Z)} \mathcal{D}^{(Z)}_{k+ \frac12,9-k}(x) ), \\
\mathcal{CK}(x) := &  \sum_{k = 0}^3 \sigma_{k} \mathcal{CK}_{k,10-k}(x), \\
\mathcal{CK}^{(P)}(x) := &  \sum_{k = 0}^3 \sigma_{k} \mathcal{CK}^{(P)}_{k,10-k}(x)  \\
\mathcal{B}(x) := &\sum_{k = 0}^3 \sigma_{k} \mathcal{B}_{k}(x) + \sum_{k = 0}^2 \sigma_{k + \frac12}^{(Z)} \mathcal{B}^{(Z)}_{k+ \frac12}(x) 
\end{align}
We first develop the following energy-CK-dissipation identity: 
\begin{align} \label{dfrty:1}
\frac{\p_x}{2} \mathcal{E} +\frac{1}{4} \mathcal{D} +\frac{1}{4} \mathcal{CK}  +  \mathcal{CK}^{(P)} + \frac{1}{4} \mathcal{B} \le 0.
\end{align} 
Again multiplying by $x^{3m}$, we have 
\begin{align} \label{dfrty:1}
\frac{\p_x}{2} (x^{3m} \mathcal{E}) +\frac{1}{4} (x^{3m}\mathcal{D}) \le 0.
\end{align} 

\vspace{2 mm}

\noindent \textit{Step 2: Nash Inequality:} As a direct consequence of our Nash inequality, \eqref{nash:type:1}, as well as our virial estimates, \eqref{boots:2}, we have for a potentially small constant $c_\ast > 0$, 
\begin{align} \label{dfrty:4}
\frac{\p_x}{2}\Gamma(x) +c_\ast \min \{  x^{- \frac14 + \frac{3m}{4}} \Gamma^{\frac52}, x^m \Gamma^3 \} \le 0.
\end{align} 
for $\Gamma := x^{3m} \mathcal{E}$. 

\vspace{2 mm}

\noindent \textit{Step 3: ODE Lemma:} By invoking Lemma \ref{lem:ODE}, we have 
\begin{align*}
\Gamma \lesssim x^{-\frac12 - \frac{m}{2}}, \qquad  \mathcal{E} \lesssim x^{-\frac12 - \frac{7m}{2}}.
\end{align*} 

\vspace{2 mm}

\noindent \textit{Step 4: Sobolev Embedding:} By invoking the standard Sobolev interpolation, we obtain the main result, \eqref{thm:main:est:1}.

\end{proof}

\appendix

\section{Derivation of Power Law from Potential Flow} \label{app:pl}

We are considering a background solution to stationary Euler equations as shown below.
\begin{figure}[h]
\centering
\begin{tikzpicture}
\draw[ultra thick, dashed, -] (0,0) -- (5,0);
\draw[ultra thick, blue, ->] (0,0) -- (2,2);
\draw[ultra thick, blue, ->] (2,2) -- (3,3);
\draw[ultra thick, blue, ->] (-4,0) -- (-2,0);
\draw[ultra thick, blue, ->] (-2,0) -- (0,0);
\node [below] at (-2, 2) {$\Delta \psi_E = 0$};
\node [below] at (4,1) {Solid Wedge};
\node [below] at (2.4,1) {\textcolor{red}{$\beta \frac{\pi}{2}$}};
\node [below] at (1.3,2.5) {\textcolor{blue}{$\psi_E = 0$}};
\node [below] at (-2, 0.6) {\textcolor{blue}{$\psi_E = 0$}};
\draw[ultra thick,red] (2,0) arc (0:45:2);
\end{tikzpicture}
\caption{Background Euler Flow} \label{fig:bkg:eul}
\end{figure}
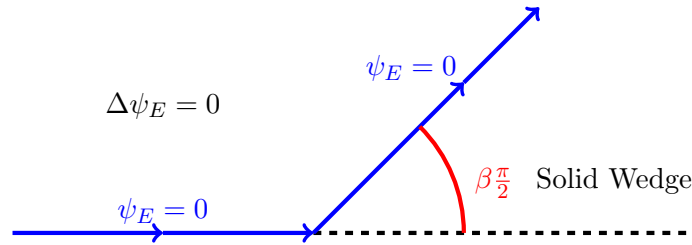
We can solve this by going to polar coordinates; the domain is $\beta \frac{\pi}{2} < \theta < \pi$ and $r > 0$. This yields an explicit solution 
\begin{align} \label{psiE}
\psi_E(\theta, r) = & r^{\frac{2}{2-\beta}} \sin( \frac{2}{2-\beta} (\pi - \theta)), 
\end{align}
We introduce the crucial parameter
\begin{align}
m = \frac{\beta}{2-\beta} \iff \beta = \frac{2m}{m+1}
\end{align}
This parameter enters through the elementary identity $\frac{2}{2-\beta} = 1 +m$. 

It is convenient to introduce the shifted angle 
\begin{align}
\varphi := \theta - \beta \frac{\pi}{2}, \qquad 0 < \varphi < \pi - \beta \frac{\pi}{2} = \frac{1}{1+m} \pi.
\end{align}
In terms of $\varphi$ and the parameter $m$ the stream function reads (we abuse notation)
\begin{align}
\psi_E(\varphi, r) = & r^{1 + m} \sin( \pi - (1 + m) \varphi)
\end{align}

We now introduce our coordinate system 
\begin{align}
x = r \cos(\theta - \beta \frac{\pi}{2}) = r \cos(\varphi), \qquad Y := r \sin(\theta - \beta \frac{\pi}{2}) = r \sin(\varphi). 
\end{align}
We record the identities 
\begin{align}
\p_Y = &\sin(\varphi) \p_r + \frac{\cos(\varphi)}{r} \p_{\varphi}, \\
\p_x = & \cos(\varphi) \p_r - \frac{\sin(\varphi )}{r} \p_{\varphi}
\end{align}
Then we get 
\begin{align*}
u_E = & \p_Y \{ r^{1 + m} \sin(\pi - (1 + m) \varphi)\} \\
= & \sin(\varphi) \p_r \{ r^{1 + m} \sin(\pi - (1+m) \varphi) \} + \frac{\cos(\varphi)}{r} \p_{\varphi} \{ r^{1 + m} \sin(\pi - (1+m) \varphi) \} \\
= & (1 + m) r^m \sin(\varphi) \sin(\pi - (1 +m) \varphi) - (1 + m) r^m \cos(\varphi) \cos(\pi - (1 + m) \varphi) \\
= & - (1 + m) r^m \cos( \pi - m \varphi   ) \\
= &  (1 + m) r^m \cos(m \varphi), 
\end{align*}
and 
\begin{align*}
v_E = & (\bar{v}_E) = - \p_x \{  r^{1 + m} \sin(\pi - (1 + m) \varphi) \} \\
= & - \cos(\varphi) \p_r \{ r^{1 + m} \sin(\pi - (1 + m) \varphi) \} +  \frac{\sin(\varphi )}{r} \p_{\varphi} \{ r^{1 + m} \sin(\pi - (1 + m) \varphi) \} \\
= & - (1 + m) r^m \cos(\varphi) \sin(\pi - (1 + m) \varphi) - (1+m) r^m \sin(\varphi) \cos(\pi - (1 +m) \varphi) \\
= & - (1 + m) r^m \sin( \pi - m \varphi  ) \\
= & - (1 + m) r^m \sin(m\varphi), 
\end{align*}
where we have used the identities 
\begin{align*}
\sin(a) \sin(b) - \cos(a) \cos(b) = & - \cos(a + b), \\
 \cos(\pi - a) = & - \cos(a), \\
\sin(a) \cos(b) + \sin(b) \cos(a) = & \sin(a + b), \\
 \sin(\pi - a) = & \sin(a). 
\end{align*}
We remark that $v_E$ \textit{does not} vanish on $\varphi = \frac{1}{1+m} \pi$ because the definition of $v_E$ is $\p_x$, that is, adapted to the Cartesian coordinates as opposed to $\p_\varphi$ of the stream function. However, we do importantly retain that $v_E|_{\varphi = 0} = 0$. 

It is now convenient to introduce 
\begin{align} \label{Euler}
\phi_E(x, Y) = \psi_E(\varphi, r), \qquad [u_E, v_E](x, Y) = [ \p_Y \phi_E, -\p_{x} \phi_E](x, Y). 
\end{align}
We compute the Euler trace (we abuse notation) $u_E(\tau) = u_E(\tau, 0)$. 
\begin{align} \n
\frac{d}{d\theta}|_{\theta = \beta \frac{\pi}{2}} \psi_E(\theta, r) = & \frac{d}{d\theta}|_{\theta = \beta \frac{\pi}{2}} \phi_E(r \cos(\theta - \beta \frac{\pi}{2}), r \sin(\theta - \beta \frac{\pi}{2})) \\
= & r \p_Y \phi_E(\tau, 0) =  \tau u_E(\tau, 0). 
\end{align}
On the other hand, using \eqref{psiE}, we obtain 
\begin{align}
\frac{d}{d\theta}|_{\theta = \beta \frac{\pi}{2}} \psi_E(\theta, r) = \frac{2}{2-\beta} r^{\frac{2}{2-\beta}}= \frac{2}{2-\beta} \tau^{\frac{2}{2-\beta}}
\end{align}
Equating these two identities, we obtain 
\begin{align} \label{uE}
u_E(\tau) = u_E(\tau, 0) = \frac{2}{2-\beta} \tau^{\frac{\beta}{2-\beta}} = \frac{2}{2-\beta} \tau^m = (1+m) \tau^m. 
\end{align}

\begin{remark} We observe that \eqref{uE} is exactly the outer Euler flow corresponding to the $\beta$ FS profile.   
\end{remark}

\vspace{5 mm}

\noindent \textbf{Acknowledgements:} The author gratefully acknowledges support from NSF DMS-2306528 and a UC Davis Hellman Foundation Fellowship.

\bibliographystyle{abbrv}
\bibliography{biblio_SI_FSPaper}

\end{document}